\documentclass[12pt,a4paper,reqno]{amsart}
\usepackage[dvipsnames]{xcolor}
\usepackage{amsmath, amssymb, amsfonts, mathtools, mathrsfs, bbm, dsfont, comment}
\usepackage{empheq}
\usepackage{marginnote}
\usepackage{ esint }
\usepackage{etoolbox}
\usepackage{amsthm}
\newtheorem{theorem}{Theorem}[section]
\newtheorem{lemma}{Lemma}[section]
\newtheorem{proposition}{Proposition}[section]
\newtheorem{corollary}{Corollary}[section]

\newtheorem{remark}{Remark}[section]

\newcommand{\h}{\hspace}
\usepackage[numbers,sort]{natbib} 

\makeatletter
\newcommand\Const[3]{%
	\@ifundefined{#1-#2}%
	{\stepcounter{#3}\expandafter\xdef\csname #1-#2\endcsname{\arabic{#3}}}%
	{}%
	\ifnum\pdfstrcmp{#1}{eps}=0 
	\varepsilon_{\csname #1-#2\endcsname}%
	\else
	#1_{\csname #1-#2\endcsname}%
	\fi
}
\makeatother
\newcommand\C[1]{\Const{C}{#1}{Ccnt}} 

\def\Xint#1{\mathchoice  
	{\XXint\displaystyle\textstyle{#1}}%
	{\XXint\textstyle\scriptstyle{#1}}%
	{\XXint\scriptstyle\scriptscriptstyle{#1}}%
	{\XXint\scriptscriptstyle\scriptscriptstyle{#1}}%
	\int}
\def\XXint#1#2#3{{\setbox0=\hbox{$#1{#2#3}{\int}$}
		\vcenter{\hbox{$#2#3$}}\kern-.58\wd0}}

\def\dashint{\Xint-}

\makeatletter
\def\l@subsection{\@tocline{2}{0pt}{2.5em}{3.5em}{}}
\makeatother

\numberwithin{equation}{section}
\allowdisplaybreaks

\usepackage{graphicx, float, subfigure}

\usepackage[a4paper, top=2.7cm, bottom=2.8cm, left=2.4cm, right=2.4cm,marginparwidth=21mm, marginparsep=1mm]{geometry}

\allowdisplaybreaks

\usepackage{ifthen,changepage}
\usepackage{xargs}
\usepackage{ragged2e} 

\newcommandx{\leftmarginnote}[2][2=0pt]{%
	\checkoddpage
	\ifoddpage
	{\reversemarginpar
		\marginnote{\RaggedRight\tiny #1}[#2]}%
	\else
	{\marginnote{\RaggedRight\tiny #1}[#2]}%
	\fi
}
\newcommandx{\rightmarginnote}[2][2=0pt]
{\checkoddpage
	\ifoddpage
	{\marginnote{#1}[#2]}
	\else
	{\reversemarginpar\marginnote{#1}[#2]}
	\fi}

\newcommand{\p}{\partial}

\newcommand{\pcm}[1]{{\color{purple}#1}}

\usepackage[shortlabels]{enumitem}

\makeatletter
\renewcommand\subsection{\@startsection{subsection}{2}{\z@}%
	{1.25\baselineskip \@plus .2\baselineskip \@minus .2\baselineskip}
	{.6\baselineskip}
	{\normalfont\centering}}

\renewcommand\subsubsection{\@startsection{subsubsection}{3}{\z@}%
	{1.0\baselineskip \@plus .2\baselineskip \@minus .2\baselineskip}%
	{.5\baselineskip}%
	{\normalfont\centering\bfseries}}
\makeatother

\usepackage[colorlinks=true, linkcolor=black, citecolor=black, urlcolor=black]{hyperref}
\usepackage{cleveref}

\usepackage{setspace}
\usepackage{scalerel, stackengine}
\newcommand\pig[1]{\scalerel*[5.5pt]{\big#1}{\ensurestackMath{\addstackgap[1.15pt]{\big#1}}}}
\newcommand\pigl[1]{\mathopen{\pig{#1}}}
\newcommand\pigr[1]{\mathclose{\pig{#1}}}

\title{Local Exact Controllability of Landau--Lifshitz--Gilbert Equation}

\author{Chunxi Jiao}
\address{School of Mathematics and Statistics, UNSW Sydney, NSW 2052, Australia}
\email{\texttt{chunxi.jiao@unsw.edu.au}}
\thanks{}

\author{Ho Man Tai}
\address{School of Mathematics and Statistics, University of Sydney, Sydney, NSW 2006, Australia}
\email{\texttt{homan.tai@sydney.edu.au}}
\thanks{}

\begin{document}

    \begin{abstract}
We prove a local exact controllability result for controlled Landau--Lifshitz--Gilbert equations on $\mathbb T^2$: if the initial energy is sufficiently small, then for any terminal time \(T>0\), there is a localised external magnetic field such that the system can be steered exactly to the terminal value of any nearby uncontrolled trajectory. We first transform the equation to a quasilinear parabolic system on \(\mathbb R^2\) by a suitable stereographic chart. Then the Carleman estimate is established for the linearised system through a decomposition adapted to the self-adjoint and skew-adjoint structure of the conjugated adjoint operator. This yields observability and \(L^\infty\)-null controllability for the linearised system. The nonlinear projected equation is then recovered by a Kakutani fixed-point argument. We also obtain a semi-global controllability result under a hemisphere condition.

\end{abstract}

	\maketitle
	\tableofcontents
	
	\section{Introduction}

    The Landau--Lifshitz--Gilbert (LLG) equation \cite{Gilbert1955ALF,landau1935theory} is a fundamental model in micromagnetics describing the evolution of the magnetization in ferromagnetic materials. In the regime below the Curie temperature, the magnetization (magnetic moment or spin) has fixed magnitude and only its orientation varies, so the state variable takes values on the unit sphere $\mathbb{S}^2:=\{x\in \mathbb{R}^{3}:\, |x|=1\}$. Let the region occupied by the ferromagnetic material be the 2-dimensional periodic domain $\mathbb{T}^2:=(\mathbb{R}/2\pi \mathbb{Z}) \times (\mathbb{R}/2\pi \mathbb{Z})$. The orientation of the magnetization $m: \mathbb{T}^2 \times[0,\infty)  \to \mathbb{S}^2$ satisfies the following equation:
	\begin{equation*}
		\p_t m (x,t)
		= \gamma\,m(x,t)\times H_{\textup{eff}}(x,t)
		- \alpha\,m(x,t)\times\bigl[m(x,t)\times H_{\textup{eff}}(x,t)\bigr]
	\end{equation*}
	where $\gamma\neq 0$ is the gyromagnetic ratio and $\alpha>0$ is the Gilbert damping constant. The effective field $H_{\textup{eff}}$ is induced from the negative variational gradient of the total magnetic free energy. In this paper, we consider the controllability of the LLG equation, where the control acts as an external magnetic field supported in a non-empty open subset $\omega \subseteq \mathbb T^2$. This problem is physically and technologically motivated and meaningful. Indeed, externally applied magnetic fields can be used to steer the magnetization toward a desired switched state in precessional magnetization-reversal experiments; see, for example, \cite{back1999minimum,gerrits2002ultrafast}. In microwave-assisted switching, the control is given by the applied reversing field together with a transverse microwave field, while the target is to drive the magnetization toward the opposite stable state, see \cite{4407618}.\smallskip

    We restrict attention to the exchange-energy case, for which the effective field is given by the Laplacian of the magnetization. The resulting system is a nonlinear geometric evolution equation combining a dissipative part and a precessional part. More precisely, one replaces $H_{\textup{eff}}$ by $H_{\textup{eff}}+ \chi_\omega u$, see \cite[Section 1]{hernandez2023global}, \cite[Section 5.1]{lu2016mini}:
	\begin{align*}
		\p_t m (x,t)
		=\,& \gamma\,m(x,t)\times \big[\Delta m(x,t) 
		+  \chi_\omega u(x,t)\big]\nonumber\\
		&- \alpha\,m(x,t)\times\pigl\{m(x,t)\times \big[\Delta m(x,t) +  \chi_\omega u(x,t)\big]\pigr\}.
	\end{align*}
	Using the triple product formula and the unit length condition for $m$, we have 
	\begin{align}\label{eq. controlled LLG}
		\left\{
		\begin{aligned}
			\p_t m (x,t)
			&= \alpha \big[ \Delta m (x,t)
			+|\nabla m (x,t)|^2 m (x,t) \big] +\gamma\,m(x,t)\times \Delta m (x,t)
			\\
			&\quad+\gamma  m (x,t) \times [ \chi_\omega u(x,t)] 
			- \alpha\,m(x,t)\times\pigl\{m(x,t)\times 
			\big[ \chi_\omega u(x,t)\big]\pigr\}
			&&\text{on $\mathbb{T}^2\times(0,\infty)$}; 
			\\
			m(x,0) &= m_0(x) &&\text{on $\mathbb{T}^2$.} 
		\end{aligned}
		\right. 
	\end{align}

   \subsection{Literature review}

Controllability for parabolic equations have been studied extensively in the literature, see the comprehensive works \cite{micu2004introduction, ammar2011recent,doi:10.1137/S0363012904439696, Yamamoto_2009, fursikov1996controllability, Glowinski_Lions_1994}. More specifically, null controllability for linear parabolic equations can be found in \cite{fernandez2006null,FU20091333,doi:10.1137/S0363012904439696,martinez2006carleman,rousseau2010carleman,COCV_2008__14_2_284_0,martin2016null,lopez2000null,alabau2006carleman,micu2006controllability,doi:10.1137/060678191,MAITY2019153} and the reference therein. For local and global exact controllability for linear parabolic equations, we can refer to \cite{fattorini1971exact, ammar2009kalman,CAO1997174}. For approximate controllability for linear parabolic equations, we can refer to \cite{CAO1997174,10.57262/ade/1356651338,doi:10.3233/ASY-201623,duprez2018positive,zuazua2002controllability,10.57262/ade/1355867413,refId012321,boyer2014approximate}. \smallskip

Nonlinear parabolic equations are particularly close in spirit to the present work. In this direction, \cite{liu2026multiplicativeMobile} studies quasi-linear parabolic equations governed by multiplicative controls with mobile support and proves null controllability. For reaction--diffusion systems, \cite{lebalch2019twoSpecies} proves null-controllability for a two-species system with nonlinear coupling by means of the reflexive uniqueness method, while \cite{lebalch2020localRD} proves local exact controllability to non-negative stationary states using an affine change of variables, a cascade system, spectral inequalities, and an inverse mapping theorem. Bilinear and multiplicative control problems are treated in \cite{alabau2022eigensolutions,khapalov2002nonnegative,khapalov2003multiplicative}: the former obtains local controllability to the (j)-th eigensolution and semi-global controllability in large time for evolution equations of parabolic type, whereas the latter studies global non-negative approximate controllability for semilinear equations governed by bilinear or multiplicative controls in the reaction term. Carleman and fixed-point techniques are used to obtain null controllability or exact controllability to trajectories for semilinear equations in negative-order Sobolev spaces \cite{imanuvilov2003carleman}, heat equations with Fourier or nonlinear boundary Fourier conditions \cite{fernandez2006fourier,doubova2004nonlinearBoundary}, coupled semilinear heat equations with discontinuous diffusion coefficients \cite{doubova2002discontinuous}, and quasi-linear equations with nonlinear terms involving the state and the gradient \cite{doubova2002gradient}. Finally, \cite{anita2002dissipative} considers the null controllability problem for the dissipative semilinear heat equation and shows that, for even mildly superlinear nonlinearities, global null controllability in arbitrarily short time fails.\smallskip

The Navier--Stokes equations form a big class in nonlinear PDE controllability. Carleman-based approaches were developed in \cite{imanuvilov1998exactNS,imanuvilov2001remarks,fursikov1999exactNSBoussinesq,fernandez2004localExactNS}, treating local exact controllability and exact controllability to trajectories for incompressible Navier--Stokes and Boussinesq systems with localized controls. The dimensions and structure of controls were further refined in \cite{fernandez2006nMinusOne,carreno2013localNullNminusOne}, where controllability is obtained with \(N-1\) scalar controls or with one vanishing control component. Another line uses degenerate or finite-dimensional forcing: \cite{agrachev2006degenerate} develops the Agrachev--Sarychev method for 2D Euler and Navier--Stokes equations, while \cite{shirikyan2006approx3DNS,shirikyan2007projections} prove approximate controllability and exact controllability in finite-dimensional projections for the 3D Navier--Stokes equations. Boundary-control and boundary-layer techniques are pursued in \cite{coron2020smallTimeNS,coron2019rectanglePhantom,liao2022smoothNS}, which establish small-time global exact controllability under Navier slip-with-friction conditions, treat the no-slip rectangular case with a small phantom force, and extend the analysis to smooth and Lagrangian controllability. Underactuated internal controls are considered in \cite{coron2014localNull3DNS}, which proves local null controllability of the 3D system with two vanishing control components. In the compressible setting, \cite{ervedoza2012compressible1D} proves local exact boundary controllability for the one-dimensional compressible Navier--Stokes equation. Finally, \cite{bonkile2018burgersReview} surveys Burgers' equation, a scalar nonlinear convection--diffusion model closely related to viscous fluid dynamics.\smallskip

To the best of our knowledge, the present paper is the first to address controllability for the LLG equation. Controllability results for geometric flows remain comparatively sparse. The closest works are \cite{liu2018control} and \cite{CORON2025103761}, both concerning harmonic map heat flow. The former studies the controllability of 2D or 3D harmonic map heat flow by means of an external field. The recently work \cite{CORON2025103761} studies the controllability of the harmonic map heat flow from \(\mathbb T^1\) to \(\mathbb S^k\). The main controllability mechanism is to first exploit the natural dissipation of the uncontrolled flow, which drives the solution close to a harmonic map, and then use localized controls to cross critical energy levels by a power-series expansion argument. Once the solution is close to a constant state, local null controllability is obtained through stereographic projection. The authors further prove controllability between harmonic maps by reducing the dynamics along closed geodesics to a controlled scalar heat equation, and then combining this with deformation arguments within the relevant homotopy class. This yields global controllability to harmonic maps, as well as small-time global exact controllability between harmonic maps in the same homotopy class.

\subsection{Main Results}

To state our main results, we consider the uncontrolled LLG equation 
	\begin{align}\label{eq. uncontrolled LLG}
		\left\{
		\begin{aligned}
			\p_t m (x,t)
			&= \alpha \big[ \Delta m (x,t)
			+|\nabla m (x,t)|^2 m (x,t) \big] +\gamma\,m(x,t)\times \Delta m (x,t) 
			&&\text{on $\mathbb{T}^2\times(0,\infty)$}; 
			\\
			m(x,0) &= m_0(x) &&\text{on $\mathbb{T}^2$.} 
		\end{aligned}
		\right. 
	\end{align}  
     For any $\varepsilon>0$ and \(m_0\in H^1(\mathbb T^2;\mathbb S^2)\), we define the target set
\[
\mathcal T(\varepsilon,T,m_0):=
\left\{
\overline m(\cdot,T)\;:\;
\begin{aligned}
&\overline m \text{ solves \eqref{eq. uncontrolled LLG} with initial condition $\overline m_0 \in H^1(\mathbb T^2;\mathbb S^2) $; }\\
& \text{and } \|m_0-\overline m_0\|_{H^{1}(\mathbb T^2)}\le \varepsilon
\end{aligned}
\right\}.
\]

	\noindent The main theorem is the following local exact controllability under the small energy assumption:
	
	\begin{theorem}[\bf Local exact controllability to trajectories]\label{thm:nullcontrol} Let $\omega$ be an open subset of $\mathbb{T}^2$ and $T>0$. There is $\varepsilon>0$ satisfying the following: for any initial state $m_0\in H^1(\mathbb{T}^2; \mathbb{S}^2)$ with $E(m_0):=\frac12\int_{\mathbb{T}^2}|\nabla m_0|^2<\varepsilon$ and any $m_T\in \mathcal T(\varepsilon,T,m_0)$, there is a control $u\in L^\infty(\omega\times(0,T);\mathbb{R}^3)$ such that the corresponding strong solution of the controlled LLG equation \eqref{eq. controlled LLG} 
    satisfies $m(\cdot,T)=m_T(\cdot) $.
	\end{theorem}
The proof is presented in \Cref{sec. Controllability with small initial energy}. Based on this theorem, we can establish two corollaries, whose proofs are given in \Cref{sec. Controllability with small initial energy,sec. Controllability with hemisphere condition} respectively.

    \begin{corollary}
        [\bf Local exact controllability to any constant targets]\label{cor:null control} Let $\omega$ be an open subset of $\mathbb{T}^2$. There is $T>0$ and $\varepsilon>0$ satisfying the following: for any initial state $m_0\in H^1(\mathbb{T}^2; \mathbb{S}^2)$ with $E(m_0)<\varepsilon$ and any constant $p^* \in \mathbb{S}^2$, there is a control $u\in L^\infty(\omega\times(0,T);\mathbb{R}^3)$ such that the corresponding strong solution of the controlled LLG equation \eqref{eq. controlled LLG} 
        satisfies $m(\cdot,T)=p^* $.
    \end{corollary}

    \begin{corollary} [\bf Semi global exact controllability to trajectories]
\label{cor:two-stage-global-a}
Let $T>0$, $\omega$ be an open subset of $\mathbb{T}^2$, $p_0\in \mathbb{S}^2$ with $p_{0,3} >0$, and $
m_0\in H^2(\mathbb T^2;\mathbb{S}^2)$ for some $p>2$ with $
\inf_{x\in \mathbb T^2} m_0(x)\cdot e_3>0$. There are $\varepsilon>0$ and a pair of nonnegative constants $(H_1,H_2)$ such that for any $m_T \in \mathcal{T}(\varepsilon,T/3,p_0)$, there is $u\in L^\infty(\omega\times(2T/3,T);\mathbb{R}^3)$ in which the strong
solution $m$ of
\begin{equation}
\label{eq:LLG-two-stage}
\left\{
\begin{aligned}
\partial_t m
&=
\alpha\bigl(\Delta m+|\nabla m|^2m\bigr)
+\gamma\,m\times \Delta m
+\gamma\,m\times B
-\alpha\,m\times\bigl(m\times B\bigr)
&&\text{in }\mathbb T^2\times(0,T);\\
m(x,0)&=m_0(x)
&&\text{on }\mathbb T^2,
\end{aligned}
\right.
\end{equation}
with the control
$$
B(x,t):=
\begin{cases}
H_1 e_3, & 0\le t\le T/3,\\
H_2 p_0, & T/3<t\le 2T/3,\\
\chi_\omega u(x,t), & 2T/3<t\le T,\\
\end{cases}$$ 
satisfies $m(\cdot,T)=m_T(\cdot)$. 
    \end{corollary}

To prove \Cref{thm:nullcontrol}, we first let the LLG flow evolve without control. If \(E(m_0)\) is sufficiently small, then there exists a time \(t_0>0\) at which the uncontrolled solution is close to \(p_0\) in the \(W^{2,p}\)-norm, by energy dissipation and parabolic regularization. Thus the state at time \(t_0\) lies in a stereographic chart. The stereographic projection transforms the LLG equation into a quasilinear parabolic system on \(\mathbb R^2\). In these coordinates, the exact controllability problem is reduced to a null controllability problem for the projected system.
\smallskip

The main technical step is the null controllability of the projected parabolic system. The first difficulty comes from the principal part of its linearization, which contains principal matrix
\[
A=\alpha I+\gamma
\begin{pmatrix}
0&-1\\
1&0
\end{pmatrix}.
\]
When \(\gamma\neq0\), this matrix is non-diagonalizable over \(\mathbb R\) and therefore standard Carleman estimate cannot be applied in this case. One possible approach is to identify \(\mathbb R^2\) with \(\mathbb C\) and apply the controllability result for equations with complex principal part in \cite{FU20091333}; another is to use the general theory of non-diagonalizable parabolic systems in \cite{fernandez2015controllability}, where the diffusion matrix is allowed to have Jordan blocks of dimension at most 4. In our work, we prove the required Carleman estimate directly by using a decomposition adapted to the special structure of \(A\). More precisely, after conjugating the adjoint operator by the Carleman weight, we split it according to its self-adjoint and skew-adjoint components, rather than according to the order of the Carleman parameters. This allows the mixed imaginary terms generated by the precessional part to be absorbed into positive \(L^2\)-norm terms, leaving only lower-order commutators that can be controlled. 
\smallskip

The second difficulty comes from the nonlinear terms in the projected equation. The associated linearized null controllability problem may have nonunique controls, so the self map is naturally multivalued. We therefore use Kakutani's fixed-point theorem. The main task is to prove compactness and upper semicontinuity of this multivalued self map. This is delicate because the frozen coefficients depend on both the trajectory and its gradient, while the controls are only uniformly bounded in \(L^\infty\). We overcome this by combining the uniform \(L^\infty\)-control estimate with maximal-regularity estimates for the controlled linearized system, obtaining uniform bounds in a suitable Sobolev space. The compact embedding of this space into \(C([0,T];C^1(\mathbb T^2))\) gives strong convergence of the frozen trajectories, which is sufficient to pass to the limit in the coefficients. At the same time, the \(L^\infty\)-bound gives weak-star compactness of the controls, and the strong convergence of the states preserves the null terminal condition. Hence the self map has a closed graph and is upper semicontinuous.
\smallskip

\Cref{cor:null control} is obtained by iterating \Cref{thm:nullcontrol} along a finite chain of nearby constant states. \Cref{thm:nullcontrol} first drives the solution to the nearby constant state \(p_0\in \mathbb{S}^2\) induced by the small initial energy. To reach an arbitrary prescribed \(p^\ast\in\mathbb S^2\), we join \(p_0\) to \(p^\ast\) by a minimizing geodesic on the sphere and divide it into finitely many sufficiently short pieces. Consecutive constant states in this partition are close enough in \(H^1(\mathbb T^2)\) for \Cref{thm:nullcontrol} to be applied at each step. \smallskip

The terminal time in \Cref{cor:null control} cannot be prescribed arbitrarily small by this concatenation argument. Indeed, the small energy threshold \(\varepsilon_*=\varepsilon_*(T)\) in \Cref{thm:nullcontrol} depends on the time horizon and tends to zero exponentially as \(T\to0^+\). If one tries to steer \(p_0\) to \(p^\ast\) in a fixed time \(T>0\) by dividing the geodesic into \(N \in \mathbb{N}\) pieces, then each step has length about \(T/N\). Although consecutive constant states are then separated by order \(1/N\), the corresponding small energy threshold \(\varepsilon_*(T/N)\) may be exponentially small in \(N/T\) by \eqref{5871}. Thus refining the partition decreases the distance between successive constant targets, but it also reduces the time interval available for each application of \Cref{thm:nullcontrol} from \(T\) to approximately \(T/N\). Consequently, the relevant small energy threshold is no longer \(\varepsilon_*(T)\), but \(\varepsilon_*(T/N)\), which may be exponentially smaller. Hence this method yields controllability after a sufficiently large concatenated time, but not for every prescribed \(T>0\).\smallskip

\Cref{cor:two-stage-global-a} uses a large-field stabilization argument to replace the small-energy assumption, under the hemisphere condition on the initial datum. The basic mechanism is that, if the solution initially has positive projection along some direction \(p^\dagger\in\mathbb S^2\), then applying a sufficiently large constant magnetic field in the direction \(p^\dagger\) drives the solution close to \(p^\dagger\). Indeed, passing to stereographic coordinates, the constant field produces a large damping effect on the projected variable, and an energy estimate shows that this projected variable becomes small as the field strength increases. This mechanism is used twice. First, a large constant field in the \(e_3\)-direction drives the initial state close to \(e_3\). Since the prescribed point \(p_0\) satisfies \(p_{0,3}>0\), the resulting state lies in the hemisphere centered at \(p_0\). Second, a large constant field in the \(p_0\)-direction drives the solution close to \(p_0\) in \(H^1\). Then \Cref{thm:nullcontrol} applies on the remaining time interval and steers the solution to any target in the corresponding local target class. Concatenating these two constant-field stages with the final localized control proves \Cref{cor:two-stage-global-a}.

\subsection{Notation and organisation}

Throughout the paper, we denote by
\[
I:=\begin{pmatrix}1&0\\[2pt]0&1\end{pmatrix},\quad
J:=\begin{pmatrix}0&-1\\[2pt]1&0\end{pmatrix}
\quad \text{and} \quad
A:=\alpha I+\gamma J.
\]
We write \(\chi_\omega\) for the characteristic function of the set \(\omega\). Moreover, we set $Q:=\mathbb{T}^2\times(0,T)$, except in Section 3, we  write \(Q:=\mathbb{T}^n\times(0,T)\) for $n\in \mathbb{N}$. Here $\mathbb{T}^n:=(\mathbb{R}/2\pi \mathbb{Z}) \times \ldots \times (\mathbb{R}/2\pi \mathbb{Z})$ is the $n$-dimensional periodic domain. If we consider an open set \(\omega_0\subset\subset\omega\), then it means that \(\overline{\omega_0}\subset\omega\). Finally, for \(u:\mathbb{T}^2\to\mathbb{R}^m\), we define the energy functional by $E(u):=\frac12\int_{\mathbb{T}^2}|\nabla u|^2.$ \smallskip

The paper is organised as follows. In Section 2, we prove that small energy implies quantitative closeness to a constant state and derive the regularization estimates needed to enter a stereographic chart. Section 3 is devoted to Carleman estimates for the relevant linear parabolic systems, including the case of divergence-form lower-order terms. In Section 4, we deduce observability and null controllability for the linearized system, together with an $L^\infty$ control construction. Finally, in Section 5, we apply a fixed-point argument to the projected nonlinear equation and complete the proof of the main results for the original LLG system.

	\section{Uniform distance estimate from a point on the sphere}\label{sec Uniform distance estimate from a point on the sphere}

	In this section, we let the solution \(m\) evolve under the uncontrolled equation \eqref{eq. uncontrolled LLG} up to a suitable time \(t_0\). The purpose of introducing this time \(t_0\) is to ensure that the state \(m(\cdot,t_0)\) satisfies smallness estimates needed to serve as an admissible initial datum for the controlled equation \eqref{eq. controlled LLG} after passing to stereographic coordinates.\smallskip
	
	By Poincar\'e's inequality, if the Dirichlet energy of a map is below a suitable threshold, then its \(L^2\)-distance to some constant state can be quantitatively controlled by the energy.
	\begin{lemma}[\bf $L^2$-distance estimate by energy]\label{lem:small-energy-const}
		For every $u\in H^1(\mathbb T^2;\mathbb S^2)$ with $E(u):=\dfrac{1}{2}\displaystyle\int_{\mathbb T^2}|\nabla u|^2 < 2\pi^2$,
		there exists $p_0\in\mathbb S^2$ satisfying
		\[
		\|u-p_0\|_{H^1(\mathbb T^2)}\le 3\sqrt{2}\sqrt{E(u)}.
		\]
	\end{lemma}
	
	\begin{proof}
		Let $\overline{u}:=\frac{1}{(2\pi)^2}\int_{\mathbb T^2}u \in\mathbb R^3$. Expanding $ |u-\overline{u}|^2$ and using Poincar\'e's inequality yield
		\begin{equation}\label{eq:ubar-lower1}
			1-|\overline{u}|^2=\dfrac{1}{(2\pi)^2}
			\int_{\mathbb T^2}|u-\overline{u}|^2
			\le \dfrac{1}{2\pi^2} E(u)<1,
		\end{equation}
		where the Poincar\'e constant is $1$ in this case. Therefore, we can define $
		p_0:=\frac{\overline{u}}{|\overline{u}|}\in\mathbb S^2$ which satisfies
		\begin{align}\label{1475}
			\|u-p_0\|_{L^2(\mathbb T^2)}\le \|u-\overline{u}\|_{L^2(\mathbb T^2)}+\|\overline{u}-p_0\|_{L^2(\mathbb T^2)}
			= \|u-\overline{u}\|_{L^2(\mathbb T^2)}+2\pi \pig| |\overline{u}|-1 \pig|.
		\end{align}
		Moreover, using \eqref{eq:ubar-lower1}, we have
		\[
		\pig||\overline{u}|-1\pig|
		\le 1-|\overline{u}|^2
		=\dfrac{1}{(2\pi)^2}\int_{\mathbb T^2}|u-\overline{u}|^2.
		\]
		Therefore, \eqref{1475} and Poincar\'e's inequality give
		\[
		\|u-p_0\|_{L^2(\mathbb T^2)}
		\le \|u-\overline{u}\|_{L^2(\mathbb T^2)}+\dfrac{1}{2\pi}\|u-\overline{u}\|_{L^2(\mathbb T^2)}^2
		\leq  \sqrt{2E(u)}+\dfrac{1}{\pi} E(u)
		\leq 2\sqrt{2}\sqrt{E(u)}.
		\]
		Summing with the energy, we conclude the result.
	\end{proof} 

	Using the above lemma, we see that \(\|m_0-p_0\|_{H^1(\mathbb T^2)}\) can be made arbitrarily small by choosing the initial energy sufficiently small. As a consequence, there exists a time \(t_0\) such that \(m(\cdot,t_0)\) is uniformly close to the constant state \(p_0\in\mathbb S^2\) by the following lemma. In particular, \(m(\cdot,t_0)\) remains in a neighborhood of \(p_0\) where the stereographic projection is well defined. Moreover, for the subsequent controllability argument, we also establish that the quantity \(\|m(\cdot,t_0)-p_0\|_{L^\infty(\mathbb T^2)}\) is sufficiently small.

    \begin{lemma}[\bf Higher order estimates]\label{lem:LLG-hemisphere-entry}
Fix a constant \(p_0\in \mathbb S^2\), \(T>0\), and \(m_0,\overline m_0\in H^1(\mathbb T^2;\mathbb S^2)\). Let $
m,\overline m:\mathbb T^2\times[0,T]\to \mathbb S^2$ be the corresponding solutions of the uncontrolled LLG equation
\eqref{eq. uncontrolled LLG} subject to the initial conditions
\(m_0\) and \(\overline m_0\), respectively. For every \(\varepsilon>0\) and \(p\in [2,\infty)\), there exists
\(\delta\in (0,4\pi^2)\), depending on
\(\varepsilon,T,\alpha,\gamma,p\), such that if $\|m_0-p_0\|_{H^1(\mathbb T^2)}\le \delta$ and $\|\overline m_0-p_0\|_{H^1(\mathbb T^2)}\le \delta$, then there exists a time \(t_0\in(0,T]\) satisfying
\[
\|m(\cdot,t_0)-p_0\|_{L^\infty(\mathbb T^2)}+
\|m(\cdot,t_0)-\overline m(\cdot,t_0)\|_{L^\infty(\mathbb T^2)}
+
\|m(\cdot,t_0)-\overline m(\cdot,t_0)\|_{W^{2,p}(\mathbb T^2)}
\le \varepsilon .
\]
\end{lemma}

	\begin{proof}
		We first choose $\delta$ small enough, depending on $\alpha$ and $\gamma$, such that $m$ and $\overline{m}$ are smooth solutions. We consider the solution $m$ first.

	\noindent{\bf Step 1. Energy dissipation and bound of $\int_0^T\|\Delta m\|_{L^2(\mathbb{T}^2)}^2$:}
	Let
	\[
	\tau(m):=\Delta m+|\nabla m|^2m.
	\]
	Since $|m|=1$, we have $m\cdot \Delta m=-|\nabla m|^2$, hence
	\[
	\frac{d}{dt}E(m)
	=-\int_{\mathbb{T}^2}\Delta m\cdot \partial_t m
	=-\alpha\int_{\mathbb{T}^2}\Delta m\cdot \tau(m)
	=-\alpha\int_{\mathbb{T}^2}|\tau(m)|^2.
	\]
	Therefore
\begin{align}\label{1756}
		E(m(\cdot,t))+\alpha\int_0^t\|\tau(m)\|_{L^2(\mathbb{T}^2)}^2
	=E(m_0)\leq \delta^2/2 \h{5pt} \text{for any $t\in[0,T]$.}
\end{align}
Hence, by the Gagliardo--Nirenberg inequality and the identity
	$\|\nabla^2m\|_{L^2(\mathbb{T}^2)}=\|\Delta m\|_{L^2(\mathbb{T}^2)}$, we obtain
	\[
	\|\nabla m\|_{L^4(\mathbb{T}^2)}^4
	\lesssim \|\nabla m\|_{L^2(\mathbb{T}^2)}^2\|\Delta m\|_{L^2(\mathbb{T}^2)}^2
	=2E(m)\|\Delta m\|_{L^2(\mathbb{T}^2)}^2
	\le \delta^2\|\Delta m\|_{L^2(\mathbb{T}^2)}^2.
	\]
Choosing $\delta$ small enough, we may absorb this term and obtain
	\[
	\|\tau(m)\|_{L^2(\mathbb{T}^2)}^2
	= \|\Delta m\|_{L^2(\mathbb{T}^2)}^2-\|\nabla m\|_{L^4(\mathbb{T}^2)}^4
	\ge \frac12\|\Delta m\|_{L^2(\mathbb{T}^2)}^2.
	\]
	Hence, \eqref{1756} implies
	\[
	\int_0^T\|\Delta m\|_{L^2(\mathbb{T}^2)}^2
	\le \frac{\delta^2}{\alpha}.
	\]
	The same estimate holds for $\overline{m}$.
	
	\noindent{\bf Step 2: Time $t_0$ with small $H^2$-distance from $p_0$.}
	By averaging over $[T/2,T]$, there exists $t_0\in[T/2,T]$ such that
	\[
	\|\Delta m(\cdot,t_0)\|_{L^2(\mathbb{T}^2)}^2+\|\Delta \overline{m}(\cdot,t_0)\|_{L^2(\mathbb{T}^2)}^2
	\le \frac{2}{T}\int_0^T\|\Delta m\|_{L^2(\mathbb{T}^2)}^2+\|\Delta \overline{m}\|_{L^2(\mathbb{T}^2)}^2 
	\le  \dfrac{4\delta^2}{\alpha T}.
	\]
	Next,
	\[
	\frac12\frac{d}{dt}\|m-p_0\|_{L^2(\mathbb{T}^2)}^2
	=\int_{\mathbb{T}^2}(m-p_0)\cdot \partial_t m.
	\]
	Using the uncontrolled LLG equation \eqref{eq. uncontrolled LLG} on the right-hand side, we have
	\[
	\alpha\int_{\mathbb{T}^2}(m-p_0)\cdot \Delta m
	=-\alpha\int_{\mathbb{T}^2}|\nabla m|^2,
	\]
	while
	\[
	\alpha\int_{\mathbb{T}^2}(m-p_0)\cdot\big(|\nabla m|^2m\big)
	=\alpha\int_{\mathbb{T}^2}|\nabla m|^2(1-p_0\cdot m).
	\]
	The contribution of the precession term vanishes after integration by parts:
	\[
	\int_{\mathbb{T}^2}(m-p_0)\cdot (m\times \Delta m)
	= -\int_{\mathbb{T}^2} p_0\cdot (m\times \Delta m)
	= -\int_{\mathbb{T}^2} \Delta m\cdot (p_0\times m).
	\]
	Since $p_0$ is constant, integrating by parts gives
	\[
	-\int_{\mathbb{T}^2} \Delta m\cdot (p_0\times m)
	= \sum_{j=1}^2 \int_{\mathbb{T}^2} \partial_j m\cdot (p_0\times \partial_j m)
	=0.
	\]
	Therefore, $\int_{\mathbb{T}^2}(m-p_0)\cdot (m\times \Delta m)=0$ and 
	\[
	\frac12\frac{d}{dt}\|m-p_0\|_{L^2(\mathbb{T}^2)}^2
	\le \alpha\int_{\mathbb{T}^2}|\nabla m|^2 
	\le 2\alpha E(m_0)
	\le \alpha\delta^2.
	\]
	Integrating in time and using the assumption that $\|m_0-p_0\|_{L^2}\le \delta$, we get
	\[
	\sup_{t\in[0,T]}\|m(\cdot,t)-p_0\|_{L^2(\mathbb{T}^2)}^2
	\le \delta^2+2\alpha T\delta^2.
	\]
	The elliptic estimate on $\mathbb{T}^2$ now yields
	\[
	\|m(\cdot,t_0)-p_0\|_{H^2(\mathbb{T}^2)}
	\lesssim \|\Delta m(\cdot,t_0)\|_{L^2(\mathbb{T}^2)}+\|m(\cdot,t_0)-p_0\|_{L^2(\mathbb{T}^2)}
	\lesssim_{\,\alpha,T}\delta.
	\]
	In particular, since $H^2(\mathbb{T}^2)\hookrightarrow L^\infty(\mathbb{T}^2)$ continuously,
	\[
	\|m(\cdot,t_0)-p_0\|_{L^\infty(\mathbb{T}^2)}
	\lesssim_{\,\alpha,T}\delta.
	\]
    The same inequality holds for $\overline{m}_0$, then triangle inequality implies that
    	\[
	\|m(\cdot,t_0)-\overline{m}(\cdot,t_0)\|_{L^\infty(\mathbb{T}^2)}
	\lesssim_{\,\alpha,T}\delta.
	\]
	
	\noindent{\bf Step 3. $W^{2,p}$ bound:}
	As the solution $m$ is smooth, then 
	\[
	\sup_{t\in[T/2,T]}\|m\|_{H^3(\mathbb{T}^2)}\lesssim_{\,\alpha,\gamma,m_0,T} 1.
	\]
	Hence, by Gagliardo--Nirenberg interpolation,
	\[
	\|\nabla^2 m(\cdot,t_0)\|_{L^p(\mathbb{T}^2)}
	\lesssim_{\,p}
	\|\nabla^3 m(\cdot,t_0)\|_{L^2(\mathbb{T}^2)}^{1-\frac2p}
	\|\nabla^2 m(\cdot,t_0)\|_{L^2(\mathbb{T}^2)}^{\frac2p}
	\lesssim_{\,\alpha,\gamma,T,p\,}\delta^{\frac2p},
	\]
	for any $p\ge 2$. Since $H^2(\mathbb{T}^2)\hookrightarrow W^{1,p}(\mathbb{T}^2)$ continuously,
	\[
	\|m(\cdot,t_0)-p_0\|_{W^{2,p}(\mathbb{T}^2)}
	\lesssim_{\,\alpha,\gamma,T,p\,} \big(\delta+\delta^{\frac2p}\big).
	\]
    The same inequality holds for $\overline{m}_0$. Choosing $\delta>0$ small enough, depending on $\alpha,\gamma,T,p,\varepsilon$, this proves the lemma.
	\end{proof}

	\section{The Carleman estimate for Linear Systems of Parabolic Equations}

	Based on the estimates in \Cref{sec Uniform distance estimate from a point on the sphere}, we safely pass to stereographic coordinates and study the corresponding equation in the plane, taking the projected state associated with \(m(\cdot,t_0)\) from \Cref{lem:LLG-hemisphere-entry} as the initial datum. This transformation allows us to replace the geometric constraint \(|m|=1\) by a quasilinear parabolic system. The null controllability of the projected equation will be obtained by a fixed-point argument, and the key ingredient for this step is the null controllability of a suitable linearised system. \smallskip

   The main goal of this section is therefore to establish a Carleman estimate for a class of linear parabolic systems that includes the linearisation of the projected LLG equation. However, the principal part of the system involves a non-diagonalizable (over \(\mathbb R\)) matrix. Consequently, the standard Carleman estimates for scalar parabolic equations cannot be applied or modified directly, and a separate analysis is needed.\smallskip

	Let the stereographic projection from $\mathbb{R}^2$ to $\mathbb{S}^2 \setminus \{-e_3\}$ and its inverse be
	\begin{align} 
		\Psi_0(v_1,v_2)=\left (\frac{2v_1}{1+|v|^2},\frac{2v_2}{1+|v|^2},\frac{1-|v|^2}{1+|v|^2}\right)^\top,
		\h{5pt} \text{and} \h{5pt}
		v=\Psi^{-1}_0(m)=\frac{(m_1,m_2)^\top}{1+m_3}
		\label{def. stereographic projection and inverse}
	\end{align} respectively. If $m$ is a solution to the controlled equation \eqref{eq. controlled LLG} subject to control $u$, then the transformed equation that $v=\Psi_0^{-1}(m)$ satisfies is
	\begin{align}
		\partial_t v=(\alpha I+\gamma J)
		\left\{\Delta v-2\,(\nabla v)\,\nabla\log h(v)+\frac{2|\nabla v|^2}{h(v)}\,v
		+\chi_\omega\frac{[h(v)]^2}{4}\,(\nabla_v\Psi_0(v))^{ \top}u\right\}
		\label{eq. projected eq}
	\end{align}
	where $J:=\begin{pmatrix}0&-1\\[2pt]1&0\end{pmatrix}$, $h(v):=1+|v|^2$ and $\nabla\log h(v)$ means $\dfrac{1}{h(v)}(\nabla v)^\top\nabla_v h(v)$. The computation is given in the Appendix \ref{app. Derivation of projected equation under stereographic coordinates}. We then linearise the above equation. Denote $A:=\alpha I+\gamma J$. Let $ f$ and $z$ be some functions, we consider the following version of linearisation of \eqref{eq. projected eq}:
	\begin{align}\label{eq. projected and linearised eq}
		\partial_t v=A
		\left[\Delta v-2\,(\nabla v)\,\nabla\log h(z)
		+\frac{2|\nabla z|^2}{h(z)}\,v\right] 
		+\chi_\omega f
		= A\pig[\Delta v+(B(z)\cdot\nabla) v
		+c(z)v\pig]
		+\chi_\omega f,
	\end{align}
	where $B(z):=-2\nabla\log h(z)$ and $c(z):=\dfrac{2|\nabla z|^2}{h(z)}$. \smallskip

	Null controllability results for scalar linear equations are available in, for instance, \cite[Proposition~1]{FernándezCara06} and \cite[Theorem~2]{fernandez2006null}. However, those results concern scalar equations and apply to \eqref{eq. projected and linearised eq} only in the special case \(\gamma=0\), namely the harmonic map heat flow. In our setting, the principal part involves the matrix \(A\), which is non-symmetric as soon as \(\gamma\neq 0\). For this reason, the arguments developed in the scalar case cannot be transferred directly to our system, nor can they be modified in a simple way to cover the present case. This non-symmetry is the main source of difficulty in establishing the Carleman estimate for the principal adjoint system in \Cref{lem Carleman est principal ad eq}; see the discussion following that lemma. We therefore establish a new null controllability result for a general class of linear parabolic systems that includes \eqref{eq. projected and linearised eq} as a particular case.
	
	\begin{proposition}[\bf Null controllability of linear systems]\label{thm null control of linear eq}\sloppy Let $n\in \mathbb{N}$ and assume that $B_i,\,\Gamma \in L^\infty(\mathbb{T}^n\times(0,T);\mathbb{R}^{2\times 2})$ for all $i=1,2,\dots,n$. For any $T>0$,  non-empty open set $\omega\subseteq \mathbb T^n$ and initial datum $u_0\in L^2(\mathbb T^n;\mathbb R^2)$, there exists a control
		$f\in L^\infty\big(\omega\times(0,T);\mathbb R^2\big)$ such that the weak solution
		$u\in C\big([0,T];L^2(\mathbb T^n;\mathbb R^2)\big)\cap L^2\big(0,T;H^1(\mathbb T^n;\mathbb R^2)\big)$
		to the linear equation 
\begin{align}\label{eq. general linear eq.}
		\p_t u(x,t)\;- A\,\Delta u(x,t)
		\;-\sum_{i=1}^n\;B_i(x,t) \p_i u(x,t)
		\;-\;\Gamma(x,t)\,u(x,t)
		\;=\; \chi_\omega f(x,t),
\end{align}
		subject to $u(\cdot,0)=u_0$, satisfies
		\[
		u(\cdot,T)=0\quad\text{a.e. in }\mathbb T^n.
		\]
		Moreover, there exists a constant $\C{4}=\C{4} \left(T,\omega,\alpha,\gamma,\|B\|_{\infty},\|\Gamma\|_{\infty}\right)>0$
		such that one can choose $f$ satisfying
		\[
		\|f\|_{L^\infty(\omega\times(0,T))} \le \C{4}\,\|u_0\|_{L^2(\mathbb T^n)},
		\]
        where the constant $\C{4}$ has the form in \eqref{5129}.
        \label{thm null control of linear eq.}
	\end{proposition}

    \begin{remark}\label{rmk existence of solution to general linear eq.}
    Although \eqref{eq. general linear eq.} is a coupled system, the existence of a weak solution
		$u\in C\big([0,T];L^2(\mathbb T^n;\mathbb R^2)\big)\cap L^2\big(0,T;H^1(\mathbb T^n;\mathbb R^2)\big)$ is standard by Galerkin approximation as $(Ap)\cdot p=\alpha|p|^2$ for any $p\in \mathbb{R}^2$ and $\alpha>0$. 
    \end{remark}
	
	The idea of proof of \Cref{thm null control of linear eq} follows the classical observability–duality strategy. We study the backward adjoint equation and proves that every adjoint solution is observable from the control region $\omega$, more precisely, we can estimate the $L^2$ size of the adjoint solution at time $0$ by looking only inside $\omega$ during the time interval $(0,T)$. Thus the restriction of the adjoint solution to the control region $\omega$ contains enough information to recover the size on the whole domain at time $0$. This observability relies on the Carleman estimate.\smallskip

	Once observability is established, null controllability for $L^2$ control follows by a standard duality argument: the control is defined as the Riesz representative of the functional induced by the initial datum on the restrictions to $\omega \times (0,T)$ of adjoint solutions. Observability says ensures that this functional can be determined continuously from what that adjoint trajectory does only inside $\omega$. Finally, the $L^\infty$ control is obtained by the arguments from \cite[Proposition 1]{FernándezCara06}. It combines the $L^2$ control with suitable time and space cut-off and using interior parabolic regularity on the resulting localised forcing terms. 
	
	\subsection{Principal adjoint system}\label{sec. Principal adjoint system}
	In this section, we first establish the Carleman estimate of the principal adjoint system \eqref{eq:backward}. Then we can deduce the Carleman estimate of the general adjoint system in \Cref{lem carleman-div}. For linear scalar equations, there are Carleman estimates in \cite[Lemma 1]{fernandez2006null} and \cite{fursikov1996imanuvilov}. However, the results there cannot be simply modified to our case. We will discuss more after stating some necessary notations.\smallskip
	
	Denote $Q:=\mathbb T^n \times (0,T)$. Let $\lambda\geq 1$ and $\eta_0$ be a scalar function on $\mathbb{T}^n$. Define the usual Carleman weights on $Q=\mathbb T^n\times(0,T)$:
	\begin{equation}\label{eq:weights}
		\xi(x,t)=\frac{e^{\lambda\eta_0(x)}}{t(T-t)},\qquad
		\varphi(x,t)=\frac{e^{2\lambda\|\eta_0\|_\infty}-e^{\lambda\eta_0(x)}}{t(T-t)}.
	\end{equation} 
	
		\begin{lemma}[\bf Carleman estimate for principal adjoint system]\label{lem Carleman est principal ad eq}
		Let $T>0$ and $\omega\subseteq \mathbb T^n$ be a non-empty open set.
		Fix an open set $\omega_0 \subset\subset \omega$ (meaning $\overline{\omega_0} \subset \omega$) and $\eta_0\in C^4(\mathbb T^n;\mathbb{R}_+)$ such that\footnote{the existence of this $\eta_0$ can be ensured by \cite{fursikov1996imanuvilov}}
		\begin{equation}\label{eq:eta0}
			\eta_0 >0 \h{10pt} \text{ in $ \mathbb T^n$ and} \h{10pt} 
			|\nabla \eta_0|>0 \h{5pt}\text{ in $\mathbb T^n\setminus \omega_0$}.
		\end{equation} 
		\sloppy Let $F\in L^2(Q;\mathbb R^2)$ and $q_T\in L^2(\mathbb T^n;\mathbb R^2)$, we denote
		$q\in L^2(0,T;H^1(\mathbb T^n;\mathbb R^2))\cap C([0,T];L^2(\mathbb T^n;\mathbb R^2))$ the corresponding weak solution of
		\begin{equation}\label{eq:backward}
			\begin{cases}
				-\p_t u - A^\top \Delta u = F & \text{in }\mathbb T^n \times (0,T);\\
				u(\cdot,T)=q_T & \text{in }\mathbb T^n.
			\end{cases}
		\end{equation}
		\sloppy Then there exist constants $\lambda_\ast\ge 1$, $\sigma_\ast\geq1$ and $\C{1}>0$ depending only on
		$\alpha,\eta_0,\gamma$ such that for every $
		\lambda\ge\lambda_\ast$ and $ s\ge \sigma_\ast\big(e^{2\lambda\|\eta_0\|_\infty}T+T^2\big)$, the solution $q$ satisfies  
		\begin{equation}\label{eq:carleman-complex-system}
			\begin{aligned}
				&\int_{\mathbb T^n \times (0,T)}  
				e^{-2s\varphi} \Big[
				(s\xi)^{-1}\bigl(|\p_t q |^2+|\Delta q|^2\bigr) 
				+s^3\lambda^4\xi^3\,|q|^2\Big]
				\\
				&\le
				\C{1}\left(
				\int_{\mathbb T^n \times (0,T)} e^{-2s\varphi}|F|^2 
				+ 
				\int_{\omega\times(0,T)} e^{-2s\varphi} s^3\lambda^4\xi^3 |q|^2 
				\right).
			\end{aligned}
		\end{equation}
		The corresponding $\xi$ and $\varphi$ are defined in \eqref{eq:weights}.
	\end{lemma}
	
	We identify $\mathbb R^2$ with $\mathbb C$ by
	\begin{align}\label{def. z=q1+iq2, G=F1+iF2}
		q=(q_1,q_2)^\top \longleftrightarrow z:=q_1+i q_2,
		\qquad
		F=(F_1,F_2)^\top \longleftrightarrow G:=F_1+iF_2.
	\end{align}
	Since multiplication by $A^\top=\alpha I-\gamma J$ on $\mathbb R^2$
	corresponds to that by $
	\beta:=\alpha-i\gamma$ on $\mathbb C$, the system is equivalent to
	\begin{equation}\label{eq:complex-backward}
		-\p_t z  -\beta\Delta z = G
		\qquad\text{in }Q,
	\end{equation}
	with $z(\cdot,T)=z_T:=q_{T,1}+iq_{T,2}$. We set
	\begin{align}\label{def psi}
		\psi:=e^{-s\varphi}z \h{10pt} \text{for $s>0$}.
	\end{align} 
By \Cref{rmk existence of solution to general linear eq.}, the weak second order spatial derivatives and time derivative of $z$ and $\psi$ exist in a strong sense on $Q$. Therefore, by \eqref{def psi}, a direct computation gives
	\[
	e^{-s\varphi}\bigl(-\p_t z  -\beta\Delta z\bigr)
	=
	-\p_t \psi-s\p_t \varphi\psi
	-\beta\Bigl(\Delta\psi+2s\nabla\varphi\cdot\nabla\psi
	+s(\Delta\varphi)\psi+s^2|\nabla\varphi|^2\psi\Bigr).
	\]
	Introduce the operators
	\begin{align}\label{def. A_0 and B_0}
		\mathcal{A}_0:=-\Delta-s^2|\nabla\varphi|^2,
		\qquad
		\mathcal{B}_0:=-2s\nabla\varphi\cdot\nabla-s(\Delta\varphi).
	\end{align}
	Then 
	\begin{align}\label{eq. of z and psi}
		e^{-s\varphi}\bigl(-\p_t z  -\beta\Delta z\bigr)
		=
		\mathcal S_\beta\psi+\mathcal K_\beta\psi
		=e^{-s\varphi}G
	\end{align} 
	where
	\begin{align}\label{def. S_beta and K_beta}
		\mathcal S_\beta:=\alpha \mathcal{A}_0-i\gamma \mathcal{B}_0-s\p_t \varphi,
		\qquad
		\mathcal K_\beta:=-\partial_t+\alpha \mathcal{B}_0-i\gamma \mathcal{A}_0.
	\end{align}
	Therefore, equation \eqref{eq. of z and psi} infers
	\begin{equation}\label{eq:basic-identity}
		\|\mathcal S_\beta\psi\|_{L^2(Q)}^2
		+
		\|\mathcal K_\beta\psi\|_{L^2(Q)}^2
		+
		2\textup{Re}\left(  \langle\mathcal S_\beta\psi,\mathcal K_\beta\psi\rangle_{L^2(Q)}\right) 
		=
		\|e^{-s\varphi}G\|_{L^2(Q)}^2,
	\end{equation}
	where $\langle w_1,w_2\rangle_{L^2(Q)}:=\int_Q w_1 \overline{w_2}$. Since $\mathcal{A}_0$ is self-adjoint and $\mathcal{B}_0$ is skew-adjoint in $L^2(Q;\mathbb C)$,
	we have
	\[
	2\textup{Re}\left(  \langle\mathcal S_\beta\psi,\mathcal K_\beta\psi\rangle_{L^2(Q)}\right) 
	=
	\big\langle[\mathcal S_\beta,\mathcal K_\beta]\psi,\psi\big\rangle_{L^2(Q)},
	\]
	where $[\mathcal S_\beta,\mathcal K_\beta]\psi
	:=
	(\mathcal S_\beta\mathcal K_\beta-\mathcal K_\beta\mathcal S_\beta)\psi$. A direct expansion gives
	\begin{align}\label{378}
		[\mathcal S_\beta,\mathcal K_\beta]
		=
		(\alpha^2+\gamma^2)[\mathcal{A}_0,\mathcal{B}_0]
		-2\alpha s[\p_t \varphi,\mathcal{B}_0]
		+2i\gamma s[\p_t \varphi,\mathcal{A}_0]
		-s\p_{tt} \varphi
	\end{align} when it acts on $\psi$.\smallskip

	We now explain why the splitting \eqref{eq. of z and psi} is preferable in the present system of equations setting. The main difficulty in proving \Cref{lem Carleman est principal ad eq} is the non-symmetry of the diffusion matrix \(A\). For this reason, we decompose into a self-adjoint part \(\mathcal S_\beta\) and a non-adjoint part \(\mathcal K_\beta\). By contrast, the splitting used in \cite[Lemma 1]{fernandez2006null} for the scalar equation groups together the terms according to their size, namely the \(s\lambda^2\)-terms and the \(s^2\lambda^2\)-terms. If one repeats that construction here, then \eqref{eq. of z and psi} can be written as $
	M^*\psi+M^{**}\psi=F_\beta,$ where
	\[
	\begin{aligned}
		M^*_\beta \psi &:=2\beta s\lambda^2\xi |\nabla \eta_0|^2 \psi
		+2\beta s\lambda\xi \nabla \eta_0 \cdot \nabla \psi -\partial_t \psi,\\
		M^{**}_\beta \psi &:= - \beta s^2 \lambda^2 \xi^2 |\nabla \eta_0|^2 \psi 
		-\beta \Delta \psi 
		- s\partial_t \varphi \,\psi,\\
		F_\beta &:= e^{-s\varphi}G   - \beta s\lambda\xi\Delta \eta_0 \psi
		+ \beta s\lambda^2\xi |\nabla \eta_0|^2\psi.
	\end{aligned}
	\]
	In the present complex case, the cross term $\textup{Re}\left(\bigl\langle M^*_\beta\psi,M^{**}_\beta\psi\bigr\rangle_{L^2(Q)}\right)$ produces mixed terms involving the imaginary part, namely
	$$\gamma\int_Q \textup{Im}\bigl(\p_t \psi \,\overline{\Delta\psi}\bigr),\quad\gamma s^2\lambda\,
	\textup{Im}\left(\int_Q
	\xi\,\p_t \varphi\,
	(\nabla\eta_0\cdot\nabla\psi)\,\overline{\psi}\ \right)	
	,\quad
	\gamma s^2\lambda^2
	\int_Q \xi^2|\nabla\eta_0|^2\, \textup{Im}(\p_t \psi \overline{\psi}).$$ 
	These terms are hard to deal with because they are not naturally controlled by the Carleman weights appearing in the desired estimate. \smallskip
	
	The advantage of the decomposition \eqref{eq. of z and psi} is that these imaginary contributions are absorbed into the non-negative term $	\|\mathcal S_\beta\psi\|_{L^2(Q)}^2
	+
	\|\mathcal K_\beta\psi\|_{L^2(Q)}^2$ in \eqref{eq:basic-identity}, which is dropped at the end of the estimate. The only explicit lower-order contribution of the imaginary part that remains after passing to the commutator is $
	2i\gamma s[\partial_t\varphi,\mathcal A_0]$, and this term can be estimated because $
	[\partial_t\varphi,\mathcal A_0]
	=
	2\nabla(\partial_t\varphi)\cdot\nabla+\Delta(\partial_t\varphi)$
	is a genuine lower-order perturbation.\smallskip


	We prove \Cref{lem Carleman est principal ad eq} by considering the terms in \eqref{eq:basic-identity} and \eqref{eq:carleman-complex-system}.
	\begin{lemma}\label{lem ineq of  dt z + Delta z}
		Assume the hypotheses of \Cref{lem Carleman est principal ad eq}, and recall the definitions in
		\eqref{def. z=q1+iq2, G=F1+iF2} and \eqref{def psi}. Then there is a constant $C=C(\alpha,\eta_0,\gamma)>0$ such that for any $\lambda\geq 1$, $\sigma_* \geq 1$ and $s\ge \sigma_*(e^{2\lambda \|\eta_0\|_\infty}T+T^2)$, it holds that \begin{align}\label{ineq. d_t z + Delta z}  
			&\int_{\mathbb T^n \times (0,T)} e^{-2s\varphi}(s\xi)^{-1}
			\left( |\p_t z  |^2 +|\Delta z|^2 \right) \nonumber \\
			&\le
			C\left[\dfrac{1}{\sigma_*}
			\|\mathcal S_\beta\psi\|_{L^2(\mathbb T^n \times (0,T))}^2
			+
			\int_{\mathbb T^n \times (0,T)} 
			\left(s\lambda^2\xi\,|\nabla\psi|^2
			+
			\dfrac{1}{\sigma_*}s^3\lambda^4\xi^3\,|\psi|^2 +\dfrac{1}{\sigma_*}e^{-2s\varphi}|G|^2\right)\right]. 
		\end{align}
	\end{lemma}

	\begin{proof}
		
		\smallskip
		
		\noindent\textbf{Step 1. Estimate of \(\mathcal{A}_0\psi\):} Since \(\alpha>0\), the definition in \eqref{def. S_beta and K_beta} shows that
		\begin{equation}\label{eq:A0-pre}
			\int_Q (s\xi)^{-1}|\mathcal{A}_0\psi|^2 
			\lesssim_{\,\alpha,\gamma}
			\int_Q (s\xi)^{-1}|\mathcal S_\beta\psi|^2 
			+
			 (s\xi)^{-1}|\mathcal{B}_0\psi|^2 
			+
			s\xi^{-1}|\p_t \varphi |^2|\psi|^2.
		\end{equation} 
		We first estimate the integral involving \(\mathcal S_\beta\). Since $	\xi(x,t)=\frac{e^{\lambda\eta_0(x)}}{t(T-t)}
		\ge \frac4{T^2}$ for all $(x,t)\in Q$ and \(s\ge \sigma_*(e^{2\lambda \|\eta_0\|_\infty}T+T^2)\), we have
		\begin{align}\label{ineq. 1/s xi }
			(s\xi)^{-1}\le \frac{T^2}{4s}\lesssim \dfrac{1}{\sigma_*} \h{10pt} \text{and} \h{10pt}(s\xi)^{-1} \lesssim \dfrac{T}{\sigma_*}.
		\end{align}
		Thus
		\begin{equation}\label{eq:Sbeta-weighted}
			\int_Q (s\xi)^{-1}|\mathcal S_\beta\psi|^2
			\lesssim \dfrac{1}{\sigma_*}
			\|\mathcal S_\beta\psi\|_{L^2(Q)}^2.
		\end{equation}
		
		Next, using
		\begin{align}\label{eq. D varphi, D^2 varphi}
			\nabla\varphi=-\lambda\,\xi\,\nabla\eta_0,
			\qquad
			\Delta\varphi=-\lambda^2\xi|\nabla\eta_0|^2-\lambda\,\xi\,\Delta\eta_0,
		\end{align}
		we may write
		\[
		\mathcal{B}_0\psi
		=
		-2s\nabla\varphi\cdot\nabla\psi-s(\Delta\varphi)\psi
		=
		2s\lambda\xi\,\nabla\eta_0\cdot\nabla\psi
		+
		s\lambda^2\xi|\nabla\eta_0|^2\psi
		+
		s\lambda\xi(\Delta\eta_0)\psi.
		\]
		Since \(\eta_0\in C^4(\mathbb T^n;\mathbb{R}_+)\), we multiply by \((s\xi)^{-1/2}\) and use \eqref{ineq. 1/s xi } to get  
		\begin{equation}\label{eq:B0-weighted}
			\int_Q (s\xi)^{-1}|\mathcal{B}_0\psi|^2 
			\lesssim_{\,\eta_0}
			\int_Q s\lambda^2\xi\,|\nabla\psi|^2
			+
			\dfrac{1}{\sigma_*^2}s^3\lambda^4\xi^3\,|\psi|^2.
		\end{equation}
		
		Finally, we estimate the integral involving \(\p_t \varphi \). One has
		\[
		|\p_t \varphi (x,t)|
		=\left|\dfrac{(T-2t)\varphi(x,t)}{t(T-t)}\right|
		\lesssim T\,e^{2\lambda \|\eta_0\|_\infty}\xi(x,t)^2.
		\] 
		Using \(s\ge \sigma_*(e^{2\lambda \|\eta_0\|_\infty}T+T^2)\) and $\lambda\geq 1$, it follows that 
		\begin{equation}\label{eq:phit-weighted}
			\int_Q s\xi^{-1}|\p_t \varphi |^2|\psi|^2
			\lesssim
			\dfrac{1}{\sigma_*^2} 
			\int_Q s^3\lambda^4\xi^3|\psi|^2.
		\end{equation}
		
		Combining \eqref{eq:A0-pre}, \eqref{eq:Sbeta-weighted}, \eqref{eq:B0-weighted}, and
		\eqref{eq:phit-weighted}, we arrive at
		\begin{equation}\label{eq:A0-final}
			\int_Q (s\xi)^{-1}|\mathcal{A}_0\psi|^2 
			\lesssim_{\,\alpha,\eta_0,\gamma}
			\dfrac{1}{\sigma_*}  \|\mathcal S_\beta\psi\|_{L^2(Q)}^2
			+
			\int_Q  
			s\lambda^2\xi\,|\nabla\psi|^2
			+
			\dfrac{1}{\sigma_*^2}s^3\lambda^4\xi^3\,|\psi|^2 .
		\end{equation}
		
		\smallskip
		
		\noindent\textbf{Step 2. Estimate of \(e^{-s\varphi}\Delta z\):}
		Since \(z=e^{s\varphi}\psi\), we use the definitions of \(\mathcal{A}_0\) and \(\mathcal{B}_0\) in \eqref{def. A_0 and B_0} to express
		\begin{align*}
			\Delta z
			=
			e^{s\varphi}
			\Bigl(
			\Delta\psi+2s\nabla\varphi\cdot\nabla\psi+s(\Delta\varphi)\psi+s^2|\nabla\varphi|^2\psi
			\Bigr)	=
			-e^{s\varphi}(\mathcal{A}_0\psi+\mathcal{B}_0\psi).
		\end{align*}
		Hence
		\[
		|e^{-s\varphi}\Delta z|^2
		\le
		2|\mathcal{A}_0\psi|^2+2|\mathcal{B}_0\psi|^2,
		\]
		and therefore, by \eqref{eq:A0-final} and \eqref{eq:B0-weighted},
		\begin{equation}\label{eq:Delta-final}
			\begin{aligned}
				\int_Q e^{-2s\varphi}(s\xi)^{-1}|\Delta z|^2
				&\lesssim_{\,\alpha,\eta_0,\gamma} 
\dfrac{1}{\sigma_*}  \|\mathcal S_\beta\psi\|_{L^2(Q)}^2
+
\int_Q  
s\lambda^2\xi\,|\nabla\psi|^2
+
\dfrac{1}{\sigma_*^2}s^3\lambda^4\xi^3\,|\psi|^2 .
			\end{aligned}
		\end{equation}
		
		\smallskip
		
		\noindent\textbf{Step 3. Estimate of \(e^{-s\varphi}\p_t z  \):}
		Taking square of the equation \eqref{eq:complex-backward} and then multiplying by \(e^{-2s\varphi}(s\xi)^{-1}\), we use \eqref{ineq. 1/s xi } to obtain
		\[
		\begin{aligned}
			\int_Q e^{-2s\varphi}(s\xi)^{-1}|\p_t z  |^2
			&\lesssim_{\,\alpha,\gamma}
			\int_Q e^{-2s\varphi}(s\xi)^{-1}|\Delta z|^2
			+\dfrac{1}{\sigma_*}e^{-2s\varphi}|G|^2.
		\end{aligned}
		\]
		Combining this with \eqref{eq:Delta-final} yields
		\begin{equation*}
			\begin{aligned}
				\int_Q e^{-2s\varphi}(s\xi)^{-1}|\p_t z |^2
				&\lesssim_{\,\alpha,\eta_0,\gamma}
				  \dfrac{1}{\sigma_*}\|\mathcal S_\beta\psi\|_{L^2(Q)}^2
				+
				\int_Q
				s\lambda^2\xi\,|\nabla\psi|^2
				+
				\dfrac{1}{\sigma_*^2}s^3\lambda^4\xi^3\,|\psi|^2
				+ \dfrac{1}{\sigma_*}e^{-2s\varphi}|G|^2 .
			\end{aligned}
	\end{equation*}	
As $\sigma_*\geq 1$, we can change $\dfrac{1}{\sigma_*^2}$ to $\dfrac{1}{\sigma_*}$ in the above inequality to obtain the desired result.
\end{proof}

	\begin{lemma}\label{lem lower order}
		Assume the hypotheses of \Cref{lem Carleman est principal ad eq}, and recall the definitions in
		\eqref{def. z=q1+iq2, G=F1+iF2} and \eqref{def psi}. Then there is a constant $C=C(\alpha,\eta_0,\gamma)>0$ such that for any 
        $\lambda_*\geq 1$, $\lambda\geq \lambda_*$, $\sigma_* \geq 1$ and $s\ge \sigma_*(e^{2\lambda \|\eta_0\|_\infty}T+T^2)$, the quantity
		\[
		I_{\rm low}
		:=
		\big\langle
		(
		-2\alpha s[\partial_t\varphi,\mathcal{B}_0]
		+2i\gamma s[\partial_t\varphi,\mathcal{A}_0]-s\partial_{tt}\varphi)\psi,\psi
		\big\rangle_{L^2(\mathbb T^n \times (0,T))}
		\]
		satisfies
		\begin{equation*}
			\begin{aligned}
				|I_{\rm low}|
				&\le C\left[ 
				\int_{\mathbb T^n \times (0,T)}  \dfrac{1}{\sigma_*}s\lambda^2\xi\,|\nabla\psi|^2\ 
				+\dfrac{1}{\lambda_*}    s^3\lambda^4\xi^3\,|\psi|^2 \right].
			\end{aligned}
		\end{equation*}
	\end{lemma}
	\begin{proof}
		We carry out the estimate term by term.
		
		\smallskip
		\noindent
		\textbf{Step 1. Computation of $[\p_t \varphi,\mathcal{A}_0]$:} 
		Since $\mathcal{A}_0=-\Delta-s^2|\nabla\varphi|^2$ and $s^2|\nabla\varphi|^2$ commutes with $\p_t \varphi$, we have
		\[
		[\p_t \varphi,\mathcal{A}_0]\psi
		=
		[\p_t \varphi,-\Delta]\psi.
		\]
		Expanding,
		\[
		\begin{aligned}
			[\p_t \varphi,-\Delta]\psi
			=
			-\p_t \varphi\Delta\psi+\Delta(\p_t \varphi\psi)
			=
			-\p_t \varphi\Delta\psi
			+\p_t \varphi\Delta\psi
			+2\nabla\p_t \varphi\cdot\nabla\psi
			+(\Delta\p_t \varphi)\psi,
		\end{aligned}
		\]
		whence
		\begin{equation}\label{eq:phi_t_A0}
			[\p_t \varphi,\mathcal{A}_0]\psi
			=
			2\nabla\p_t \varphi\cdot\nabla\psi
			+
			(\Delta\p_t \varphi)\psi.
		\end{equation}
		
		\smallskip
		\noindent
		\textbf{Step 2. Computation of $[\p_t \varphi,\mathcal{B}_0]$:} We first observe that $[\p_t \varphi,\,-s(\Delta\varphi)]\psi=0$ and $
		\nabla(\p_t \varphi\psi)=\nabla\p_t \varphi\,\psi+\p_t \varphi\nabla\psi$. It follows that
		\begin{equation}\label{eq:phi_t_B0}
			[\p_t \varphi,\mathcal{B}_0]\psi
			=
			-2s\p_t \varphi\,\nabla\varphi\cdot\nabla\psi
			+
			2s\nabla\varphi\cdot\nabla(\p_t \varphi\psi)
			=
			2s\,\nabla\varphi\cdot\nabla\p_t \varphi\,\psi.
		\end{equation}
		
		\smallskip
		\noindent
		\textbf{Step 3. Bounds for the derivatives of the weights:} 
		Since $
		\p_t \xi (x,t)
		=
		-e^{\lambda\eta_0(x)}\frac{T-2t}{[t(T-t)]^2}$, and therefore
		\[
		|\p_t \xi |
		\leq 
		\frac{Te^{\lambda\eta_0}}{[t(T-t)]^2}
		=
		Te^{-\lambda\eta_0}\xi^2
		\le
		T\xi^2.
		\]  
		Moreover, using \eqref{eq. D varphi, D^2 varphi}, we have
		\[
		\nabla\partial_t\varphi
		=
		-\lambda\p_t \xi \nabla\eta_0,\h{10pt} \text{and} \h{10pt}
		\Delta\partial_t\varphi
		=
		-\lambda^2\p_t \xi |\nabla\eta_0|^2-\lambda\p_t \xi \Delta\eta_0.
		\] 
		It follows that, as $\lambda\ge1$,
		\begin{equation}\label{eq:weight-bounds}
			|\nabla\partial_t\varphi|\lesssim_{\,\eta_0} T \lambda\xi^2,
			\qquad
			|\Delta\partial_t\varphi|\lesssim_{\,\eta_0} T\lambda^2\xi^2,
			\qquad
			|\nabla\varphi\cdot\nabla\partial_t\varphi|\lesssim_{\,\eta_0} T\lambda^2\xi^3.
		\end{equation} 
		Finally, we deduce
		\begin{equation}\label{eq:phi-tt-bound}
			|\p_{tt} \varphi(x,t)|
			\lesssim
			\frac{T^2e^{2\lambda\|\eta_0\|_\infty}}{t^3(T-t)^3}
			= T^2 e^{2\lambda \|\eta_0\|_\infty-3\lambda\eta_0(x)}\,\xi(x,t)^3
			\le
			T^2 e^{2\lambda \|\eta_0\|_\infty}\,\xi(x,t)^3.
		\end{equation}
		 
		\noindent
		\textbf{Step 4. Estimate of the lower-order commutators:} 
		Using \eqref{eq. D varphi, D^2 varphi}, \eqref{eq:phi_t_A0}, \eqref{eq:phi_t_B0}, \eqref{eq:weight-bounds}, and \eqref{eq:phi-tt-bound},
		we find
		\[
		\begin{aligned}
			|I_{\rm low}|
			&\le
			\int_Q s|\p_{tt} \varphi|\,|\psi|^2
			+
			2\alpha s \pig|[\partial_t\varphi,\mathcal{B}_0]\psi\pig|\,|\psi|
			+
			2|\gamma|s \pig|[\partial_t\varphi,\mathcal{A}_0]\psi\pig|\,|\psi|
			\\
			&\lesssim_{\,\alpha,\eta_0,\gamma}
			\int_Q sT^2e^{2\lambda \|\eta_0\|_\infty}\xi^3|\psi|^2
			+Ts^2\lambda^2\xi^3|\psi|^2
			+ Ts\lambda\xi^2|\nabla\psi|\,|\psi|
			+ Ts\lambda^2\xi^2|\psi|^2.
		\end{aligned}
		\] 
		Since $s\ge  \sigma_*\bigl(e^{2\lambda \|\eta_0\|_\infty}T+T^2\bigr)$, we have
		\[
		e^{2\lambda \|\eta_0\|_\infty}T^2
		\le (e^{2\lambda \|\eta_0\|_\infty}T+T^2)^2
		\le \dfrac{s^2}{\sigma_*^2}\leq s^2,
		\]
		and therefore
		\[
		sT^2e^{2\lambda \|\eta_0\|_\infty}\xi^3|\psi|^2
		\leq s^3\xi^3|\psi|^2
		\leq
		\dfrac{1}{\lambda_*^4}s^3\lambda^4\xi^3|\psi|^2.
		\]
		Let $\varepsilon>0$, we use Young's inequality and $s\ge  \sigma_*\bigl(e^{2\lambda \|\eta_0\|_\infty}T+T^2\bigr)$ to obtain
		\[
		s\lambda\xi^2|\nabla\psi|\,|\psi|
		\lesssim
		\varepsilon\, \lambda^2\xi|\nabla\psi|^2
		+
		\dfrac{1}{\varepsilon}\, s^2\xi^3|\psi|^2
		\lesssim
		\varepsilon\, \dfrac{s}{T \sigma_*}\lambda^2\xi|\nabla\psi|^2
		+
		\dfrac{1}{\varepsilon \sigma_*T \lambda_*^4}\, s^3\xi^3\lambda^4|\psi|^2.
		\]
		We again use $s\ge  \sigma_*\bigl(e^{2\lambda \|\eta_0\|_\infty}T+T^2\bigr)$, \eqref{ineq. 1/s xi } and $\lambda\geq \lambda_* \ge1$ to yield
		\[ 
		T s^2\lambda^2\xi^3|\psi|^2
		+
		T s\lambda^2\xi^2|\psi|^2
		\lesssim
		\dfrac{1}{\lambda_*^2}s^3\lambda^4\xi^3|\psi|^2.
		\] 
		Consequently,
		$$
			\begin{aligned}
				|I_{\rm low}|
				&\lesssim_{\,\alpha,\eta_0,\gamma}
				\dfrac{\varepsilon}{\sigma_*}
				\int_Q s\lambda^2\xi\,|\nabla\psi|^2 
				+\dfrac{1}{\lambda_*}
				\left(1+\dfrac{1}{\varepsilon}\right)
				\int_Q s^3\lambda^4\xi^3\,|\psi|^2.
			\end{aligned}
		$$

	\end{proof}
	Putting \eqref{eq. D varphi, D^2 varphi} into \eqref{def. A_0 and B_0} and then using the adjointness of $\mathcal{A}_0$ and $\mathcal{B}_0$, we expand the term $\langle[\mathcal{A}_0,\mathcal{B}_0]\psi,\psi\rangle_{L^2(Q)}$ in \eqref{378} by
	\begin{align}\label{eq. [A,B]psi,psi =}
		\langle[\mathcal{A}_0,\mathcal{B}_0]\psi,\psi\rangle_{L^2(Q)}
		=2\textup{Re}(	\langle\mathcal{B}_0\psi,\mathcal{A}_0\psi\rangle_{L^2(Q)})
		=
		2\mathcal Q(\psi)+\mathcal R(\psi),
	\end{align}
	where
	\[
	\mathcal Q(\psi):=
	\textup{Re}\left[\int_Q
	\Bigl(2s\lambda^2\xi|\nabla\eta_0|^2\psi+2s\lambda\,\xi\,\nabla\eta_0\cdot\nabla\psi\Bigr)
	\overline{\Bigl(-s^2\lambda^2\xi^2|\nabla\eta_0|^2\psi-\Delta\psi\Bigr)}\right]
	\]
	and
	\[
	\mathcal R(\psi)
	:=
	2\textup{Re}\left[\int_Q
	s\lambda\xi(\Delta \eta_0-\lambda |\nabla\eta_0|^2)\psi\,
	\overline{\Bigl(-s^2\lambda^2\xi^2 |\nabla\eta_0|^2\,\psi-\Delta\psi\Bigr)}\right].
	\]
	Set $X_1:=2s\lambda^2\xi|\nabla\eta_0|^2\psi$, $X_2:=2s\lambda\,\xi\,\nabla\eta_0\cdot\nabla\psi$, $Y_1:=-s^2\lambda^2\xi^2|\nabla\eta_0|^2\psi$ and $Y_2:=-\Delta\psi$. Note that only $\psi$ is complex in the expressions. Then we decompose
	\[
	\mathcal Q(\psi):=\sum_{i,j=1}^2 \textup{Re}\left(\int_Q X_i\,\overline{Y_j} \right) 
	=:Q_{11}+Q_{21}+Q_{12}+Q_{22},\quad\text{ and } \quad
		\mathcal R(\psi) 
	=R_1+R_2,
	\] 
	where
	\[
	R_1:=
	-2s\lambda\,\textup{Re}\left(\int_Q \nu\,\psi\,\overline{\Delta\psi} \right),\quad 
	R_2:=
	-2s^3\lambda^3\int_Q \xi^3
	\bigl(\Delta\eta_0-\lambda|\nabla\eta_0|^2\bigr)|\nabla\eta_0|^2\,|\psi|^2
	\]
	and
	\begin{align}\label{def nu} 
		\nu(x,t):=\xi(x,t)\bigl(\Delta\eta_0(x)-\lambda|\nabla\eta_0(x)|^2\bigr).
	\end{align}
	
	\begin{lemma}\label{lem Q11+Q21+R2} Assume the hypotheses of \Cref{lem Carleman est principal ad eq}. Then there exist constants $\lambda_\ast=\lambda_\ast(\eta_0)\ge 1$, $c=c(\eta_0)>0$ and $C=C(\eta_0)>0$, such that for every $
		\lambda\ge\lambda_\ast$, and $ s\ge e^{2\lambda\|\eta_0\|_\infty}T+T^2$, the following estimate holds:
		\begin{equation}\label{eq:Qzero}
			2Q_{11}+2Q_{21}+R_2
			\ge
			c\,s^3\lambda^4\int_{\mathbb T^n \times (0,T)} \xi^3|\psi|^2 
			-
			C\,s^3\lambda^4\int_{\omega\times(0,T)} \xi^3|\psi|^2.
		\end{equation}
	\end{lemma}
	
	\begin{proof} 
		\noindent
		\textbf{Step 1. Estimate of $Q_{11}+Q_{21}$:}
		We first recall
		\[
		Q_{11}
		=
		-2s^3\lambda^4\int_Q |\nabla\eta_0|^4\xi^3|\psi|^2
		\h{10pt} \text{ and } \h{10pt}
		Q_{21}
		=
		-2s^3\lambda^3\textup{Re}\left[\int_Q |\nabla\eta_0|^2\xi^3(\nabla\eta_0\cdot\nabla\psi)\,\overline{\psi} \right].
		\]
		Since $
		\textup{Re}\left[(\nabla\eta_0\cdot\nabla\psi)\overline{\psi}\,\right]
		=
		\frac12\,\nabla\eta_0\cdot\nabla(|\psi|^2)$, we get
		\[
		Q_{21}
		=
		-s^3\lambda^3\int_Q |\nabla\eta_0|^2\xi^3\,\nabla\eta_0\cdot\nabla(|\psi|^2)
		=
		s^3\lambda^3\int_Q \nabla \cdot \bigl(|\nabla\eta_0|^2\xi^3\nabla\eta_0\bigr)\,|\psi|^2.
		\] 
		Using $ \nabla\xi=\lambda\,\xi\,\nabla\eta_0$, we have
		\[
		\nabla \cdot \bigl(|\nabla\eta_0|^2\xi^3\nabla\eta_0\bigr)
		=
		3\lambda |\nabla\eta_0|^4\xi^3
		+
		\xi^3\,\nabla \cdot \bigl(|\nabla\eta_0|^2\nabla\eta_0\bigr).
		\]
		Therefore
		\[
		Q_{21}
		=
		3s^3\lambda^4\int_Q |\nabla\eta_0|^4\xi^3|\psi|^2
		+
		s^3\lambda^3\int_Q \nabla \cdot \bigl(|\nabla\eta_0|^2\nabla\eta_0\bigr)\xi^3|\psi|^2.
		\]
		Since $\eta_0\in C^4(\mathbb T^n;\mathbb{R}_+)$, thus
		\[
		Q_{11}+Q_{21}
		\geq
		s^3\lambda^4\int_Q |\nabla\eta_0|^4\xi^3|\psi|^2 
		-C_{\eta_0}
		s^3\lambda^3\int_Q \xi^3|\psi|^2
		\]
		for some constant $C_{\eta_0}>0$ depending only on $\eta_0$. Because $\omega_0\subset\subset \omega$ and $|\nabla\eta_0|>0$ on $\mathbb T^n\setminus\omega_0$,
		by compactness there exists $m_{\eta_0}>0$ such that 
		\[
		\begin{aligned}
			Q_{11}+Q_{21}
			&\ge
			m_{\eta_0}^4 s^3\lambda^4\int_{(\mathbb T^n\setminus\omega)\times(0,T)} \xi^3|\psi|^2
			-
			C_{\eta_0} s^3\lambda^3\int_Q \xi^3|\psi|^2 
			\\
			&=
			\bigl(m_{\eta_0}^4-C_{\eta_0}/\lambda\bigr)s^3\lambda^4
			\int_{(\mathbb T^n\setminus\omega)\times(0,T)} \xi^3|\psi|^2 
			-
			C_{\eta_0} s^3\lambda^3\int_{\omega\times(0,T)} \xi^3|\psi|^2.
		\end{aligned}
		\] 
		Choosing $\lambda\ge \lambda_*$ large enough (depending only on $\eta_0$), we obtain
		$$
		Q_{11}+Q_{21}
		\ge
		c_{\eta_0}\,s^3\lambda^4\int_Q \xi^3|\psi|^2 
		-
		C_{\eta_0}\,s^3\lambda^4\int_{\omega\times(0,T)} \xi^3|\psi|^2$$
		for some constant $c_{\eta_0}>0$ depending only on $\eta_0$. Note that the constants $c_{\eta_0}$ and $C_{\eta_0}$ may change their values from line to line, but only depend only on $\eta_0$.
		
		\noindent\textbf{Step 2: Computation of \(R_2\).}
		Expanding $R_2$ gives
		\[
		\begin{aligned}
			R_2
			&=
			-2s^3\lambda^3\int_Q \xi^3
			\bigl(\Delta\eta_0-\lambda|\nabla\eta_0|^2\bigr)|\nabla\eta_0|^2\,|\psi|^2
			\\
			&=
			2s^3\lambda^4\int_Q \xi^3|\nabla\eta_0|^4\,|\psi|^2 
			-
			2s^3\lambda^3\int_Q \xi^3(\Delta\eta_0)|\nabla\eta_0|^2\,|\psi|^2.
		\end{aligned}
		\]
		Since \(\eta_0\in C^4(\mathbb T^n;\mathbb{R}_+)\), thus
		\begin{equation*}
			R_2
			\ge
			2s^3\lambda^4\int_Q \xi^3|\nabla\eta_0|^4\,|\psi|^2 
			-
			C_{\eta_0} s^3\lambda^3\int_Q \xi^3|\psi|^2.
		\end{equation*}
		If necessary, we choose $\lambda_*$ further large enough, depending on $\eta_0$ only, then
		\begin{align*}
			R_2
			&\ge
			2s^3 \lambda^4 m_{\eta_0}^4 \int_{(\mathbb{T}^n\setminus\omega)\times(0,T)} \xi^3 \,|\psi|^2 
			-
			\dfrac{C_{\eta_0} s^3\lambda^4 }{\lambda_*}\int_{\omega\times(0,T)} \xi^3|\psi|^2
			-
			\dfrac{C_{\eta_0} s^3\lambda^4 }{\lambda_*} \int_{(\mathbb{T}^n\setminus\omega)\times(0,T)} \xi^3|\psi|^2\\
			&\ge
			s^3 \lambda^4 c_{\eta_0} \int_{(\mathbb{T}^n\setminus\omega)\times(0,T)} \xi^3 \,|\psi|^2 
			-
			C_{\eta_0} s^3\lambda^4  \int_{\omega\times(0,T)} \xi^3|\psi|^2 \\
			&=
			s^3 \lambda^4 c_{\eta_0} \int_{Q} \xi^3 \,|\psi|^2 
			-
			( c_{\eta_0}+ C_{\eta_0}) s^3\lambda^4  \int_{\omega\times(0,T)} \xi^3|\psi|^2.
		\end{align*}
		Summing with the inequality of $2Q_{11}+2Q_{21}$, it yields the result. 
	\end{proof}

	\begin{lemma}\label{lem Q12+Q22+R1}	Assume the hypotheses of \Cref{lem Carleman est principal ad eq}. Then there exist constants $C=C(\eta_0)>0$, such that for every $\varepsilon>0$, $\sigma_*\geq 1$, $\lambda_*\geq 1$, $
		\lambda\ge\lambda_\ast$ and $ s\ge \sigma_*(e^{2\lambda\|\eta_0\|_\infty}T+T^2)$, the following estimate holds:  
		\begin{equation*}
			\begin{aligned}
				2(Q_{12}+Q_{22})+R_1
				&\ge
				4s\lambda^2\int_{\mathbb T^n \times (0,T)} \xi\,|\nabla\eta_0\cdot\nabla\psi|^2 
				- 
				C\left( \varepsilon+\dfrac{1}{\lambda_*}\right)  s\lambda^2\int_{\mathbb T^n \times (0,T)} \xi|\nabla\psi|^2\\
				&\quad-
				\dfrac{C}{\sigma_*^2}\Bigl(1+\frac1\varepsilon\Bigr)s^3\lambda^4\int_{\mathbb T^n \times (0,T)} \xi^3|\psi|^2.
			\end{aligned}
		\end{equation*} 
	\end{lemma}
	
	\begin{proof}		\noindent\textbf{Step 1. Computation of \(R_1\):}
		Since \(\nu\) is real-valued, integration by parts yields
		\[
		\begin{aligned}
			R_1
			=
			2s\lambda\,\textup{Re}\left(\int_Q \nabla(\nu\psi)\cdot \nabla\overline{\psi} \right) 
			=
			2s\lambda\int_Q \nu\,|\nabla\psi|^2 
			+
			2s\lambda\,\textup{Re}\left(\int_Q \psi\,\nabla \nu\cdot\nabla\overline{\psi} \right).
		\end{aligned}
		\]
		Now $
		2\textup{Re}\bigl(\psi\,\nabla \nu\cdot\nabla\overline{\psi}\bigr)
		=
		\nabla \nu\cdot \nabla(|\psi|^2)$, hence, integrating by parts once more on \(\mathbb T^n\),
		\[
		2\textup{Re}\left(\int_Q \psi\,\nabla \nu\cdot\nabla\overline{\psi}\right) 
		=
		-\int_Q (\Delta  \nu)\,|\psi|^2.
		\]
		Therefore, the definition of $\nu$ in \eqref{def nu} implies
		\begin{align*}
			R_1
			&=
			2s\lambda\int_Q \nu\,|\nabla\psi|^2
			-
			s\lambda\int_Q (\Delta  \nu)\,|\psi|^2\nonumber\\
			&=-2s\lambda^2\int_Q \xi|\nabla\eta_0|^2|\nabla\psi|^2
			+
			2s\lambda\int_Q \xi\,(\Delta\eta_0)\,|\nabla\psi|^2
			-
			s\lambda\int_Q (\Delta\nu)\,|\psi|^2.
		\end{align*}
		\noindent
		\textbf{Step 2. Estimate of $Q_{12}+Q_{22}$:} We begin with
		\[
		Q_{12}
		=
		-2s\lambda^2\textup{Re}\left(\int_Q |\nabla\eta_0|^2\xi\,\psi\,\overline{\Delta\psi}\right).
		\]
		Integrating by parts on $\mathbb T^n$ gives
		\[
		Q_{12}
		=
		2s\lambda^2\int_Q |\nabla\eta_0|^2\xi\,|\nabla\psi|^2\,
		+
		2s\lambda^2\textup{Re}\left[\int_Q \nabla \bigl(|\nabla\eta_0|^2\xi\bigr)\cdot\nabla\psi\,\overline{\psi}\,\right].
		\]
		And
		\[
		Q_{22}
		=
		-2s\lambda\textup{Re}\left[\int_Q \xi(\nabla\eta_0\cdot\nabla\psi)\,\overline{\Delta\psi}\right].
		\]
		Writing this in components and integrating by parts,
		\[
		\begin{aligned}
			Q_{22}
			&=
			2s\lambda\textup{Re}\left[
			\int_Q \partial_k \bigl(\xi\,\partial_j\eta_0\,\partial_j\psi\bigr)\,
			\overline{\partial_k\psi}\right] \\
			&=
			2s\lambda\int_Q  \xi\,\partial_{jk}\eta_0\,\partial_j\psi\,
			\overline{\partial_k\psi}
			+
			2s\lambda^2\int_Q \xi\,|\nabla\eta_0\cdot\nabla\psi|^2
			+
			2s\lambda\textup{Re}\left(
			\int_Q \xi\,\partial_j\eta_0\,\partial_{jk}\psi\,
			\overline{\partial_k\psi}\right).
		\end{aligned}
		\]
		For the last term, we observe that
		\[
		2\textup{Re}\left( \partial_j\eta_0\,\partial_{jk}\psi\,
		\overline{\partial_k\psi}\,\right) 
		=
		\nabla\eta_0\cdot\nabla(|\nabla\psi|^2),
		\]
		hence we can integrate by parts once more on $\mathbb T^n$ to obtain,
		\[
		2s\lambda\textup{Re}\left(
		\int_Q \xi\,\partial_j\eta_0\,\partial_{jk}\psi\,
		\overline{\partial_k\psi}\right)  
		=
		-s\lambda^2\int_Q \xi |\nabla\eta_0|^2 |\nabla\psi|^2
		-
		s\lambda\int_Q \xi\,(\Delta\eta_0)\,|\nabla\psi|^2.
		\] 
		Thus, we have
		\[
		\begin{aligned}
			Q_{22}
			&=
			2s\lambda\int_Q  \xi\,\partial_{jk}\eta_0\,\partial_j\psi\,
			\overline{\partial_k\psi}
			+
			2s\lambda^2\int_Q \xi\,|\nabla\eta_0\cdot\nabla\psi|^2 \\
			&\qquad
			-
			s\lambda^2\int_Q \xi |\nabla\eta_0|^2 |\nabla\psi|^2
			-
			s\lambda\int_Q \xi\,(\Delta\eta_0)\,|\nabla\psi|^2.
		\end{aligned}
		\] 
		\noindent
		\textbf{Step 3. Summing up :}
		Adding the identities for $2Q_{12}$, $2Q_{22}$ and $R_1$, we obtain, 
		\begin{equation}\label{eq:exact-combined-gradient}
			\begin{aligned}
				2(Q_{12}+Q_{22})+R_1
				&=
				4s\lambda^2\int_Q \xi\,|\nabla\eta_0\cdot\nabla\psi|^2 
				+
				4s\lambda\int_Q \xi\,\partial_{jk}\eta_0\,\partial_j\psi\,
				\overline{\partial_k\psi}\\
				&\qquad
				+
				4s\lambda^2\textup{Re}\left(\int_Q \nabla \bigl(|\nabla\eta_0|^2\xi\bigr)\cdot\nabla\psi\,\overline{\psi} \right)
				-
				s\lambda\int_Q (\Delta\nu)\,|\psi|^2  .
			\end{aligned}
		\end{equation} 
		We now estimate the last two terms in \eqref{eq:exact-combined-gradient}.
		Since \(\eta_0\in C^4(\mathbb T^n;\mathbb{R}_+)\) and \(\nabla\xi=\lambda\xi\nabla\eta_0\), we have
		\[
		\left|\nabla\bigl(|\nabla\eta_0|^2\xi\bigr)\right|
		\le
		|\nabla(|\nabla\eta_0|^2)|\,\xi
		+
		|\nabla\eta_0|^2|\nabla\xi|
		\lesssim_{\,\eta_0}
		\lambda \xi .
		\]
		Hence, by Young's inequality, for every \(\varepsilon>0\),
		\begin{equation}\label{3326}
			\begin{aligned}
				\left|
				4s\lambda^2\textup{Re}\left(\int_Q \nabla \bigl(|\nabla\eta_0|^2\xi\bigr)\cdot\nabla\psi\,\overline{\psi} \right)
				\right| 
				&\lesssim_{\,\eta_0}
				s\lambda^3\int_Q \xi |\nabla\psi|\,|\psi|  \\
				&\lesssim_{\,\eta_0}
				\varepsilon s\lambda^2\int_Q \xi|\nabla\psi|^2 
				+
				\frac{1}{\varepsilon}s\lambda^4\int_Q \xi|\psi|^2  .
			\end{aligned}
		\end{equation} 
		We now estimate \(\Delta  \nu\) which can be expressed by
		\[
		\Delta  \nu
		=
		(\Delta\xi)\bigl(\Delta\eta_0-\lambda|\nabla\eta_0|^2\bigr)
		+
		2\nabla\xi\cdot\nabla\bigl(\Delta\eta_0-\lambda|\nabla\eta_0|^2\bigr)
		+
		\xi\,\Delta\bigl(\Delta\eta_0-\lambda|\nabla\eta_0|^2\bigr).
		\]
		Using $
		\nabla\xi=\lambda\,\xi\,\nabla\eta_0$ and $
		\Delta\xi=\lambda^2\xi|\nabla\eta_0|^2+\lambda\xi\Delta\eta_0,$
		we have
		\[
		|\nabla\xi|\lesssim_{\,\eta_0} \lambda\xi,
		\qquad
		|\Delta\xi|\lesssim_{\,\eta_0} (\lambda+\lambda^2)\xi.
		\]
		Because \(\eta_0\in C^4(\mathbb T^n;\mathbb{R}_+)\), so
		\[
		\bigl|\nabla(\Delta\eta_0-\lambda|\nabla\eta_0|^2)\bigr|
		\lesssim_{\,\eta_0} 1+\lambda,
		\qquad
		\bigl|\Delta(\Delta\eta_0-\lambda|\nabla\eta_0|^2)\bigr|
		\lesssim_{\,\eta_0} 1+\lambda.
		\]
		Consequently, $
		|\Delta\nu| \lesssim_{\,\eta_0} \lambda^3\xi,$ and therefore
		\begin{align}\label{3368}
			-s\lambda\int_Q (\Delta\nu)|\psi|^2 
			\ge
			-C_{\eta_0} s\lambda^4\int_Q \xi|\psi|^2 .
		\end{align}
		
		Substituting the two bounds in \eqref{3326} and \eqref{3368} into \eqref{eq:exact-combined-gradient}, we get
		\begin{equation*}
			\begin{aligned}
				2(Q_{12}+Q_{22})+R_1
				&\ge
				4s\lambda^2\int_Q \xi\,|\nabla\eta_0\cdot\nabla\psi|^2  
				+
				4s\lambda\int_Q \xi\,\partial_{jk}\eta_0\,\partial_j\psi\,
				\overline{\partial_k\psi}  \\
				&\qquad
				-
				\varepsilon s\lambda^2\int_Q \xi|\nabla\psi|^2 
				-
				C_{\eta_0}\Bigl(1+\frac1\varepsilon\Bigr)s\lambda^4\int_Q \xi|\psi|^2 .
			\end{aligned}
		\end{equation*}
		And then using \eqref{ineq. 1/s xi } and $\lambda\geq\lambda_*$ give 
		\begin{align*}
			2(Q_{12}+Q_{22})+R_1
			&\ge
			4s\lambda^2\int_Q \xi\,|\nabla\eta_0\cdot\nabla\psi|^2 
			- 
			C_{\eta_0}\left( \varepsilon+\dfrac{1}{\lambda_*}\right)  s\lambda^2\int_Q \xi|\nabla\psi|^2 \\
			&\quad-
			\dfrac{C_{\eta_0}}{\sigma_*^2}\Bigl(1+\frac1\varepsilon\Bigr)s^3\lambda^4\int_Q \xi^3|\psi|^2 .
		\end{align*} 
	\end{proof}  
	However, the terms $s\lambda^2\int_Q \xi|\nabla\psi|^2 $ in Lemmas \ref{lem ineq of  dt z + Delta z}, \ref{lem lower order} and \ref{lem Q12+Q22+R1} are not localised. We do not have similar term in \Cref{lem Q11+Q21+R2} to cancel it. So we use the following lemma to control it by $\int_Q (s\xi)^{-1}|\Delta\psi|^2$ and $s^3\lambda^4\int_Q \xi^3|\psi|^2 $.


	
	\begin{lemma}\label{lem:gradient-from-laplacian}
	Assume the hypotheses of \Cref{lem Carleman est principal ad eq}. There exists a constant
		\(C=C(\eta_0)>0\) such that for every $\varepsilon>0$, $\sigma_*\geq 1$, \(\lambda\ge 1\) and $ s\ge \sigma_\ast(e^{2\lambda\|\eta_0\|_\infty}T+T^2)$, it holds that
		\begin{equation*}
			s\lambda^2\int_{\mathbb T^n \times (0,T)}\xi |\nabla\psi|^2
			\le
			\varepsilon\int_{\mathbb T^n \times (0,T)} (s\xi)^{-1}|\Delta\psi|^2
			+
			C\left( \frac{1}{\varepsilon}+\dfrac{1}{\sigma_*^2}\right)  s^3\lambda^4\int_{\mathbb T^n \times (0,T)} \xi^3|\psi|^2.
		\end{equation*}
	\end{lemma}
	
	\begin{proof}
		We may integrate by parts to get 
		\begin{equation*} 
			\int_Q \xi |\nabla\psi|^2 
			=
			-\textup{Re}\left( \int_Q \xi\,\overline{\psi}\,\Delta\psi \right) 
			-
			\textup{Re}\left( \int_Q \nabla\xi\cdot\nabla\psi\,\overline{\psi}\right) .
		\end{equation*} 
		Multiplying by \(s\lambda^2\) and then using $
		\nabla\xi=\lambda \xi\nabla\eta_0$, we get
		\begin{equation}\label{eq:grad-lap-proof-3}
			s\lambda^2\int_Q \xi |\nabla\psi|^2 
			\le
			s\lambda^2\int_Q \xi |\psi|\,|\Delta\psi| 
			+
			C_{\eta_0}s\lambda^3\int_Q \xi |\nabla\psi|\,|\psi| .
		\end{equation}
		By Young's inequality, it follows that 
		\begin{equation}\label{eq:grad-lap-proof-4}
			s\lambda^2\int_Q \xi |\psi|\,|\Delta\psi| 
			\le
			\frac{\varepsilon}{2}\int_Q (s\xi)^{-1}|\Delta\psi|^2 
			+
			\frac{1}{2\varepsilon} s^3\lambda^4\int_Q \xi^3|\psi|^2 ,
		\end{equation}
		for $\varepsilon>0$. For the second term of the right-hand side in \eqref{eq:grad-lap-proof-3}, Young's inequality gives
		\begin{equation}\label{eq:grad-lap-proof-5}
			C_{\eta_0}s\lambda^3\int_Q \xi |\nabla\psi|\,|\psi| 
			\le
			\frac{\varepsilon'}{2} s\lambda^2\int_Q \xi |\nabla\psi|^2 
			+
			\dfrac{C_{\eta_0}}{\varepsilon'} s\lambda^4\int_Q \xi |\psi|^2
		\end{equation} 
		for $\varepsilon'>0$. Substituting \eqref{eq:grad-lap-proof-4} and \eqref{eq:grad-lap-proof-5} into \eqref{eq:grad-lap-proof-3}, we obtain
		\[
		\begin{aligned}
			s\lambda^2\int_Q \xi |\nabla\psi|^2 
			&\le
			\frac{\varepsilon}{2} \int_Q (s\xi)^{-1}|\Delta\psi|^2 
			+
			\frac{1}{2\varepsilon} s^3\lambda^4\int_Q \xi^3|\psi|^2  \\
			&\qquad
			+
			\frac{\varepsilon'}{2} s\lambda^2\int_Q \xi |\nabla\psi|^2 
			+
			\dfrac{C_{\eta_0}}{\varepsilon'} s\lambda^4\int_Q \xi |\psi|^2 .
		\end{aligned}
		\]
		Hence, for $\varepsilon'=1/2$, we have
		\begin{equation*}
			\frac12 s\lambda^2\int_Q \xi |\nabla\psi|^2 
			\le
			\frac{\varepsilon}{2}\int_Q (s\xi)^{-1}|\Delta\psi|^2 
			+
			\frac{1}{2\varepsilon} s^3\lambda^4\int_Q \xi^3|\psi|^2 
			+
			C_{\eta_0} s\lambda^4\int_Q \xi |\psi|^2 .
		\end{equation*}
		Applying \eqref{ineq. 1/s xi } to estimate the last term, we conclude
		$$ s\lambda^2\int_Q \xi |\nabla\psi|^2
		\le
		\varepsilon\int_Q (s\xi)^{-1}|\Delta\psi|^2
		+
		\frac{1}{\varepsilon} s^3\lambda^4\int_Q \xi^3|\psi|^2
		+
		\dfrac{C_{\eta_0}}{\sigma_*^2} s^3\lambda^4\int_Q \xi^3 |\psi|^2.$$
	\end{proof}
	
	\begin{proof}[\bf Proof \Cref{lem Carleman est principal ad eq}]
		
		\noindent{\bf Part 1. Estimate from \eqref{eq:basic-identity}:}
		We add the inequalities in Lemmas \ref{lem Q11+Q21+R2} and \ref{lem Q12+Q22+R1}. It is obtained from \eqref{eq. [A,B]psi,psi =} that 
		\begin{align*}
			\langle[\mathcal{A}_0,\mathcal{B}_0]\psi,\psi\rangle_{L^2(Q)}
			\geq\,& \left[ c_{\eta_0}-
			\dfrac{C_{\eta_0}}{\sigma_*^2}
			\Bigl(1+\frac1\varepsilon\Bigr)\right] \,s^3\lambda^4\int_Q \xi^3|\psi|^2 
			-
			C_{\eta_0}\,s^3\lambda^4\int_{\omega\times(0,T)} \xi^3|\psi|^2.\nonumber\\
			&
			- 
			C_{\eta_0}\left( \varepsilon+\dfrac{1}{\lambda_*}\right)  s\lambda^2\int_Q \xi|\nabla\psi|^2 
		\end{align*}
		Using \eqref{eq:basic-identity} and \eqref{378}, we obtain
		\[
		\begin{aligned}
			\|\mathcal S_\beta\psi\|_{L^2(Q)}^2
			+
			\|\mathcal K_\beta\psi\|_{L^2(Q)}^2
			+
			(\alpha^2+\gamma^2) \langle[\mathcal{A}_0,\mathcal{B}_0]\psi,\psi\rangle_{L^2(Q)}
			+
			I_{low}
			=
			\|e^{-s\varphi}G\|_{L^2(Q)}^2.
		\end{aligned}
		\]
		Combining \eqref{eq. [A,B]psi,psi =},
		\Cref{lem lower order,lem Q11+Q21+R2,lem Q12+Q22+R1}, we obtain
		\begin{equation}\label{eq:intermediate-main}
			\begin{aligned}
				&\|\mathcal S_\beta\psi\|_{L^2(Q)}^2
				+\|\mathcal K_\beta\psi\|_{L^2(Q)}^2
				+c_{\alpha,\eta_0,\gamma}\left[ 1-
				\dfrac{C_{\alpha,\eta_0,\gamma}}{\sigma_*^2}
				\Bigl(1+\frac1\varepsilon\Bigr)
				-\dfrac{C_{\alpha,\eta_0,\gamma}}{\lambda_*}  \right]\,s^3\lambda^4\int_Q \xi^3|\psi|^2 \\
				&\le
				C_{\alpha,\eta_0,\gamma}\left[
				\int_Q e^{-2s\varphi}|G|^2
				+
				\int_{\omega\times(0,T)} s^3\lambda^4\xi^3|\psi|^2
				\right]
				+
				C_{\alpha,\eta_0,\gamma}
				\left(
				\frac{1}{\sigma_*}+\varepsilon+\frac{1}{\lambda_*}
				\right)
				\int_Q s\lambda^2\xi\,|\nabla\psi|^2 .
			\end{aligned}
		\end{equation}

		By the definition in \eqref{def. A_0 and B_0}, we express
		\[
		\Delta\psi=-\mathcal A_0\psi-s^2|\nabla\varphi|^2\psi,
		\]
		and hence
		\[
		|\Delta\psi|^2
		\le
		2|\mathcal A_0\psi|^2
		+
		2s^4|\nabla\varphi|^4|\psi|^2.
		\]
		Using \(|\nabla\varphi|=\lambda\xi |\nabla\eta_0|\), it follows that
		\begin{equation}\label{eq:Delta-psi-by-A0}
			\int_Q (s\xi)^{-1}|\Delta\psi|^2
			\le
			2\int_Q (s\xi)^{-1}|\mathcal A_0\psi|^2
			+
			C_{\eta_0}\,s^3\lambda^4\int_Q \xi^3|\psi|^2.
		\end{equation} 
		Substituting \eqref{eq:Delta-psi-by-A0} and \eqref{eq:A0-final} into the estimate in \Cref{lem:gradient-from-laplacian}, we get
		\[
		\begin{aligned}
			\int_Q s\lambda^2\xi |\nabla\psi|^2
			&\le
			C_{\alpha,\eta_0,\gamma}\varepsilon'
			\left[
			\frac{1}{\sigma_*}\|\mathcal S_\beta\psi\|_{L^2(Q)}^2
			+
			\int_Q s\lambda^2\xi\,|\nabla\psi|^2
			+
			\left(1+\frac{1}{\sigma_*^2}\right)
			s^3\lambda^4\int_Q \xi^3|\psi|^2
			\right] \\
			&\qquad
			+
			C_{\eta_0}\left(\frac{1}{\varepsilon'}+\frac{1}{\sigma_*^2}\right)
			s^3\lambda^4\int_Q \xi^3|\psi|^2.
		\end{aligned}
		\] for some $\varepsilon'>0$. Choosing \(\varepsilon'>0\) small enough so that
		\(C_{\alpha,\eta_0,\gamma}\varepsilon'\le \frac12\), we infer
		\begin{equation}\label{eq:gradient-estimate-closed}
			\int_Q s\lambda^2\xi |\nabla\psi|^2
			\le
			\frac{C_{\alpha,\eta_0,\gamma}}{\sigma_*}\|\mathcal S_\beta\psi\|_{L^2(Q)}^2
			+
			C_{\alpha,\eta_0,\gamma}
			\left(1+\frac{1}{\sigma_*^2}\right)
			s^3\lambda^4\int_Q \xi^3|\psi|^2.
		\end{equation}
		Plugging \eqref{eq:gradient-estimate-closed} into
		\eqref{eq:intermediate-main}, we obtain
		\[
		\begin{aligned}
			&\left[1-\frac{C_{\alpha,\eta_0,\gamma}}{\sigma_*}
			\left(\frac{1}{\sigma_*}+\varepsilon+\frac{1}{\lambda_*}\right)\right]
			\|\mathcal S_\beta\psi\|_{L^2(Q)}^2
			+\|\mathcal K_\beta\psi\|_{L^2(Q)}^2 \\
			&\quad+
			c_{\alpha,\eta_0,\gamma}\left[  1-
			\dfrac{C_{\alpha,\eta_0,\gamma}}{\sigma_*^2}
			\Bigl(1+\frac1\varepsilon\Bigr)
			-\dfrac{C_{\alpha,\eta_0,\gamma}}{\lambda_*} 
			-
			C_{\alpha,\eta_0,\gamma}
			\left(\frac{1}{\sigma_*}+\varepsilon+\frac{1}{\lambda_*}\right)
			\left(1+\frac{1}{\sigma_*^2}\right)
			\right]
			s^3\lambda^4\int_Q \xi^3|\psi|^2 \\
			&\le
			C_{\alpha,\eta_0,\gamma}\left[
			\int_Q e^{-2s\varphi}|G|^2
			+
			\int_{\omega\times(0,T)} s^3\lambda^4\xi^3|\psi|^2
			\right].
		\end{aligned}
		\]
		Hence, choosing first \(\varepsilon>0\) small enough and then
		\(\sigma_*\ge 1\), \(\lambda_*\ge 1\) large enough, we deduce
		\begin{equation}\label{eq:intermediate-z-fixed}
			\begin{aligned}
				\|\mathcal S_\beta\psi\|_{L^2(Q)}^2
				+\|\mathcal K_\beta\psi\|_{L^2(Q)}^2
				+
				s^3\lambda^4\int_Q \xi^3|\psi|^2 
			\le
				C_{\alpha,\eta_0,\gamma}\left[
				\int_Q e^{-2s\varphi}|G|^2
				+
				\int_{\omega\times(0,T)} s^3\lambda^4\xi^3|\psi|^2
				\right].
			\end{aligned}
		\end{equation}  
		Adding the result in \Cref{lem ineq of  dt z + Delta z} to \eqref{eq:intermediate-z-fixed} and then applying \eqref{eq:gradient-estimate-closed}, we conclude that
		\begin{equation*}
			\begin{aligned}
				&\|\mathcal S_\beta\psi\|_{L^2(Q)}^2
				+\|\mathcal K_\beta\psi\|_{L^2(Q)}^2
				+
				\int_Q e^{-2s\varphi}(s\xi)^{-1}\left(|\partial_t z|^2+|\Delta z|^2\right)
				+
				s^3\lambda^4\int_Q \xi^3|\psi|^2 \\
				&\le
				C_{\alpha,\eta_0,\gamma}\left[
				\int_Q e^{-2s\varphi}|G|^2
				+
				\int_{\omega\times(0,T)} s^3\lambda^4\xi^3|\psi|^2
				\right].
			\end{aligned}
		\end{equation*}

	\end{proof}

	\subsection{Adjoint system with divergence terms}  
	
	We now extend the Carleman estimate obtained in \Cref{lem Carleman est principal ad eq} for the principal adjoint system to a more general backward equation with a right-hand side containing both a zeroth-order term and divergence terms. In the study of the full adjoint system, the lower-order terms can naturally be rewritten in this form. The divergence terms we considered here is only in $L^2(0,T;H^{-1}(\mathbb T^n;\mathbb R^2))$ in general, and therefore the principal estimate in \Cref{lem Carleman est principal ad eq} cannot be applied directly. Instead, we argue by duality argument based on the principal Carleman estimate. The divergence terms are then estimated through a weighted bound on an auxiliary function and its gradient.

	\begin{lemma}[\bf Carleman estimate with divergence right-hand side]
		\label{lem carleman-div}
			Let $T>0$ and $\omega\subseteq \mathbb T^n$ be a non-empty open set.
		Fix an open set $\omega_0 \subset\subset \omega$ (meaning $\overline{\omega_0} \subset \omega$) and $\eta_0\in C^4(\mathbb T^n;\mathbb{R}_+)$ such that
		$$
			\eta_0(x) >0 \h{10pt} \text{ in $ \mathbb T^n$ and} \h{10pt} 
			|\nabla \eta_0(x)|>0 \h{5pt}\text{ in $\mathbb T^n\setminus \omega_0$}.$$
		Let $F_0,F_1,\dots,F_n\in L^2(\mathbb T^n \times (0,T);\mathbb R^2)$ and $\widetilde{q}_T\in L^2(\mathbb T^n;\mathbb R^2)$.
		Let $
		\widetilde{q}\in L^2(0,T;H^1(\mathbb T^n;\mathbb R^2))
		\cap C([0,T];L^2(\mathbb T^n;\mathbb R^2))$ be the weak solution of
		\begin{equation}
			\label{eq:adj-div}
			\begin{cases}
				-\partial_t u-A^\top\Delta u=F_0+\displaystyle\sum_{i=1}^n \partial_i F_i
				&\text{in }\mathbb T^n \times (0,T);\\[1mm]
				u(\cdot,T)=\widetilde{q}_T &\text{in }\mathbb T^n.
			\end{cases}
		\end{equation}
		Then there exist $\lambda_*\ge 1$, $\sigma_*\geq 1$ and $\C{2}>0$, depending only on $\alpha$, $\eta_0$ and $\gamma$, such that for every $
		\lambda\ge \lambda_*$ and $s\ge \sigma_*\bigl(e^{2\lambda\|\eta_0\|_\infty}T+T^2\bigr)$, 
		one has
		\begin{equation}
			\label{eq:carleman-div}
			\begin{aligned}
				&\int_{\mathbb T^n \times (0,T)} e^{-2s\varphi}
				\Bigl(
				s^3\lambda^4\xi^3\,|\widetilde{q}\,|^2
				\Bigr)
				\\
				&\le
				\C{2}\Biggl(
				\int_{\mathbb T^n \times (0,T)}e^{-2s\varphi}|F_0|^2
				+
				s^2\lambda^2\sum_{i=1}^n
				\int_{\mathbb T^n \times (0,T)} e^{-2s\varphi}\xi^2 |F_i|^2
				+s^3\lambda^4
				\int_{\omega\times(0,T)} e^{-2s\varphi}\xi^3 |\widetilde{q}\,|^2
				\Biggr).
			\end{aligned}
		\end{equation}
	\end{lemma} 
\begin{remark}\label{def weak sol of general adjoint}
	We say $
\widetilde{q}$ is the weak solution of \eqref{eq:adj-div} if 
\begin{enumerate}
	\item $
	\widetilde{q}\in L^2(0,T;H^1(\mathbb T^n;\mathbb R^2))
	\cap C([0,T];L^2(\mathbb T^n;\mathbb R^2))$ and the time derivative $\p_t \widetilde{q}$ in the distributional sense belongs to $L^2(0,T;(H^1)')$;
	
	\item $-\langle \p_t \widetilde{q},  \Psi\rangle_{((H^1)',H^1)} + \displaystyle\int_{\mathbb{T}^n}\nabla (A^\top\widetilde{q}\,) \cdot \nabla \Psi 
	= \int_{\mathbb{T}^n} F_0 \cdot \Psi-\sum_{i=1}^nF_i\cdot\p_i \Psi $ for a.e. $t \in (0,T)$ and any $\Psi \in H^1(\mathbb T^n;\mathbb R^2)$;
	
	\item $\widetilde{q}(\cdot,T)=\widetilde{q}_T(\cdot)$ a.e. in $\mathbb T^n$.
\end{enumerate}   
\end{remark}

		\begin{proof}  We set two operators
			\[
			L:=\partial_t -A \Delta ,
			\qquad
			L^*:=-\partial_t -A^\top \Delta .
			\]
		 Define the inner product space $(\mathcal{X},\langle\cdot,\cdot\rangle_{\mathcal{X}})$ with \( \mathcal{X}:=C^\infty (\overline Q;\mathbb R^2)\) and the inner product 
		 \[
		 \langle p^1,p^2 \rangle_{\mathcal{X}}
		 :=
		 \int_Q e^{-2s\varphi}(L^*p^1)\cdot (L^*p^2) 
		 +
		 s^3\lambda^4
		 \int_{\omega\times(0,T)} e^{-2s\varphi}\xi^3 p^1\cdot p^2 .
		 \] 
		 The norm is naturally induced by
		\begin{align}\label{2534}
			\|p\|_{\mathcal{X}}
			:=
			\left( \int_Q e^{-2s\varphi}|L^*p|^2 
			+
			s^3\lambda^4
			\int_{\omega\times(0,T)} e^{-2s\varphi}\xi^3 |p|^2\right)^{1/2}.
		\end{align} 
		Indeed, with Lemma~\ref{lem Carleman est principal ad eq} applied to any \(p \in C^\infty (\overline Q;\mathbb R^2)\), it yields
		\begin{equation}
			\label{eq:X-controls-carleman}
			\int_Q e^{-2s\varphi}
			\Bigl[
			(s\xi)^{-1}\bigl(|\p_t p|^2+|\Delta p|^2\bigr)
			+s^3\lambda^4\xi^3 |p|^2
			\Bigr] 
			\le
			\C{1}\|p\|_{\mathcal{X}}^2 .
		\end{equation}
		Therefore, \(\|\cdot\|_{\mathcal{X}}\) is a norm on \(\mathcal{X}=C^\infty (\overline Q;\mathbb R^2)\). Define the linear functional \(\mathcal L\) on \(\mathcal{X}=C^\infty (\overline Q;\mathbb R^2)\) by
		\[
		\mathcal L(p)
		:=
		s^3\lambda^4 \int_Q e^{-2s\varphi}\xi^3 \widetilde{q} \cdot p .
		\]
		By Cauchy--Schwarz inequality and \eqref{eq:X-controls-carleman},
		\[
		|\mathcal L(p)|
		\le
		s^3\lambda^4
		\Bigl(\int_Q e^{-2s\varphi}\xi^3 |\widetilde{q}\,|^2 \Bigr)^{1/2}
		\Bigl(\int_Q e^{-2s\varphi}\xi^3 |p|^2 \Bigr)^{1/2}
		\le
		\C{1}\,s^{3/2}\lambda^2
		\Bigl(\int_Q e^{-2s\varphi}\xi^3 |\widetilde{q}\,|^2 \Bigr)^{1/2}
		\|p\|_{\mathcal{X}}.
		\] 
		Let \(\overline{\mathcal{X}}\) be the completion of \(C^\infty (\overline Q;\mathbb R^2)\) for this norm. Thus \(\mathcal L\) is continuous on \((\mathcal{X},\|\cdot\|_{\mathcal{X}})\), and so it extends uniquely to a continuous linear functional
		\[
		\overline{\mathcal L}\in \overline{\mathcal{X}}^{\,\prime}
		\] since $\mathcal{X}$ is dense in $\overline{\mathcal{X}}$. By the Riesz representation theorem, there exists a unique \(p^*\in \overline{\mathcal{X}}\) such that
		\begin{equation}
			\label{eq:LM}
			\langle p^*,p\rangle_{\overline{\mathcal{X}}}=\overline{\mathcal L}(p) \h{15pt} \text{for all $p \in  \overline{\mathcal{X}}$}.
		\end{equation} 
		
		\noindent{\bf Part 1. Estimate of $p^*$:} By definition of the completion, there exists a Cauchy sequence \( \{p_k\}_{k\in \mathbb{N}} \subset \mathcal X\) such that
		\[
		p_k\to p^*
		\qquad\text{in }\overline{\mathcal X}.
		\]
	Hence there exist $
		\widetilde{\zeta}\in L^2(Q;\mathbb{R}^2)$ and $\widetilde{\phi}\in L^2(\omega\times(0,T);\mathbb{R}^2)$, such that
		\[
		e^{-s\varphi}L^*p_k \to \widetilde{\zeta}
		\qquad\text{in }L^2(Q;\mathbb{R}^2),
\quad\text{and}\quad
		e^{-s\varphi} s^{3/2}\lambda^2 \xi^{3/2}p_k
		\to \widetilde{\phi}
		\qquad\text{in }L^2(\omega\times(0,T);\mathbb{R}^2).
		\]
		We then define
		\begin{align*}
				\zeta:=e^{-s\varphi}\widetilde{\zeta}\in L^2(Q;\mathbb{R}^2),
		\qquad
		\phi:=-\,s^{3/2}\lambda^2 e^{-s\varphi}\xi^{3/2}\widetilde{\phi}
		\in L^2(\omega\times(0,T);\mathbb{R}^2).
		\end{align*} 
		Equivalently,
		\[
		e^{-2s\varphi}L^*p_k \to \zeta
		\qquad\text{in }L^2(Q;\mathbb{R}^2),
\quad \text{and}\quad
		-\,s^3\lambda^4 e^{-2s\varphi}\xi^3 p_k
		\to \phi
		\qquad\text{in }L^2(\omega\times(0,T);\mathbb{R}^2).
		\]
		
		Let \(\Psi\in \mathcal X\) and \(k\in \mathbb{N}\), one has
		\[
		\langle p_k,\Psi\rangle_{\overline{\mathcal X}}
		=
		\int_Q e^{-2s\varphi}(L^*p_k)\cdot (L^*\Psi)
		+
		s^3\lambda^4
		\int_{\omega\times(0,T)} e^{-2s\varphi}\xi^3 p_k\cdot \Psi.
		\]
		Passing to the limit as \(k\to\infty\), and using the convergences above, we obtain
		\[
		\langle p^*,\Psi\rangle_{\overline{\mathcal X}}
		=
		\int_Q \zeta\cdot L^*\Psi
		-
		\int_{\omega\times(0,T)} \phi\cdot \Psi.
		\]
		On the other hand, by the Riesz identity \eqref{eq:LM},
				\begin{equation}
			\label{eq:zv}
		\langle p^*,\Psi\rangle_{\overline{\mathcal X}}
		=
		\overline{\mathcal L}(\Psi)=
		\mathcal L(\Psi)
		=
		s^3\lambda^4\int_Q e^{-2s\varphi}\xi^3 \widetilde q\cdot \Psi
		=
		\int_Q \zeta\cdot L^*\Psi
		-
		\int_{\omega\times(0,T)} \phi\cdot \Psi.
	\end{equation}  
		 
		From \eqref{eq:zv}, we use Cauchy--Schwarz inequality and then \eqref{eq:X-controls-carleman} to observe
		\[
		\langle p^*,p_k\rangle_{\overline{\mathcal{X}}}
		\le
		\Bigl(s^3\lambda^4 \int_Q e^{-2s\varphi}\xi^3 |\widetilde{q}\,|^2 \Bigr)^{1/2}
		\Bigl(s^3\lambda^4 \int_Q e^{-2s\varphi}\xi^3 |p_k|^2 \Bigr)^{1/2} \leq \C{1}Y^{1/2}\|p_k\|_{ \mathcal{X}},
		\]	where
		\begin{align}\label{eq. def of Y}
			Y:=s^3\lambda^4 \int_Q e^{-2s\varphi}\xi^3 |\widetilde{q}|^2 .
		\end{align} 
		Taking $k \to \infty$ on both sides, we have 
				\[
		\| p^*\|_{\overline{\mathcal{X}}} \leq \C{1}Y^{1/2},
		\]
		On the other hand, the definition in \eqref{2534} and limiting arguments yield
				\begin{equation}
			\label{eq:z-v-first}
		\|p^*\|_{\overline{\mathcal{X}}}^2
		=
		\int_Q e^{2s\varphi}|\zeta|^2 
		+
		s^{-3}\lambda^{-4}
		\int_{\omega\times(0,T)} e^{2s\varphi}\xi^{-3}|\phi|^2 
		\leq \C{1}^2Y.
			\end{equation}

		\noindent\textbf{Part 2. Estimate of \(\nabla \zeta\):} At this stage, \eqref{eq:zv} indicates that the function $\zeta$ only belongs to $L^2(Q;\mathbb R^2)$ and solves
		\[
		\partial_t\zeta-A\Delta\zeta
		=
		s^3\lambda^4 e^{-2s\varphi}\xi^3\widetilde q+\chi_\omega\phi
		\quad\text{in }\mathcal D'(Q).
		\] 
        Since the right-hand side belongs to $L^2(Q;\mathbb R^2)$, standard $L^2$-regularity for parabolic systems (see also \Cref{rmk existence of solution to general linear eq.}) yields that $
		\zeta\in C\big([0,T];L^2(\mathbb T^n;\mathbb R^2)\big)\cap L^2\big(0,T;H^1(\mathbb T^n;\mathbb R^2)\big)$ 
        and $\zeta \in W^{2,1}_2(\mathbb{T}^n \times (\delta,T))$ for any $\delta\in(0,T)$. Hence the above equation is satisfied in strong sense. Therefore, we rewrite \eqref{eq:zv} as 
		\begin{equation}
			\label{eq:zv weak}
			s^3\lambda^4\int_Q e^{-2s\varphi}\xi^3 \widetilde q\cdot \Psi
			=
			\int_Q -\zeta\cdot \p_t \Psi+\nabla (A\zeta)\cdot \nabla\Psi
			-
			\int_{\omega\times(0,T)} \phi\cdot \Psi.
		\end{equation}  
		Integrating by parts with respect to time and choosing suitable $\Psi$ yield that
\begin{align*}
			\zeta(\cdot,0)=\zeta(\cdot,T)=0 \qquad \text{a.e. in  $\mathbb T^n$}.
\end{align*} 

		We next claim that
		\begin{equation}
			\label{eq:z-grad}
			s^{-2}\lambda^{-2}
			\int_Q e^{2s\varphi}\xi^{-2} |\nabla \zeta|^2 
			\le C_{\alpha,\eta_0,\gamma}  \,Y .
		\end{equation} 
		If \(Y=\infty\), there is nothing to prove. Hence we may assume \(Y<\infty\).
		In view of \eqref{eq:z-v-first}, this implies that
		\[
		e^{s\varphi}\zeta\in L^2(Q;\mathbb R^2),
		\qquad
		e^{s\varphi}\xi^{-3/2}\phi \in L^2(\omega\times(0,T);\mathbb R^2).
		\]  
For \(\varepsilon\in(0,(1 \wedge T)/2]\), define the regularized weights
\[
\xi_\varepsilon(x,t):=
\frac{e^{\lambda\eta_0(x)}}{(t+\varepsilon)(T-t+\varepsilon)},
\qquad
\varphi_\varepsilon(x,t):=
\frac{e^{2\lambda\|\eta_0\|_\infty}-e^{\lambda\eta_0(x)}}{(t+\varepsilon)(T-t+\varepsilon)},
\]
and $
w_\varepsilon:=s^{-2}\lambda^{-2}e^{2s\varphi_\varepsilon}\xi_\varepsilon^{-2} \in C^\infty(\overline Q)\cap L^\infty(Q)$.	By density arguments, we can set $\Psi = w_\varepsilon \zeta$ in \eqref{eq:zv weak} to obtain
\begin{align*}
	&\frac12 \int_Q w_\varepsilon\,\partial_t |\zeta|^2
	+
	\int_Q (A \nabla \zeta)\cdot\nabla( w_\varepsilon\zeta)\\
	&=
	s\lambda^2\int_Q
e^{2s\varphi_\varepsilon-2s\varphi}\xi^3\xi_\varepsilon^{-2}
\,\widetilde q\cdot \zeta
+
s^{-2}\lambda^{-2}
\int_{\omega\times(0,T)} e^{2s\varphi_\varepsilon}\xi_\varepsilon^{-2}\phi\cdot \zeta.
\end{align*} 
	Therefore, using the fact that $
	(A\nabla\zeta)\cdot\nabla\zeta
	=
	\alpha |\nabla\zeta|^2$, the above relation gives
\begin{align}
	\label{eq:energy-z-1-rig}
	&\alpha s^{-2}\lambda^{-2}
	\int_Q e^{2s\varphi_\varepsilon}\xi_\varepsilon^{-2} |\nabla \zeta|^2
	\nonumber\\
	&\lesssim
	s^{-2}\lambda^{-2}
	\int_Q \bigl|\partial_t(e^{2s\varphi_\varepsilon}\xi_\varepsilon^{-2})\bigr|\,|\zeta|^2
	+
	s^{-2}\lambda^{-2}
	\int_Q |A|\,\bigl|\nabla(e^{2s\varphi_\varepsilon}\xi_\varepsilon^{-2})\bigr|
	\,|\nabla \zeta|\,|\zeta|
	\nonumber\\
	&\quad
	+
	s\lambda^2\int_Q e^{2s\varphi_\varepsilon-2s\varphi}\xi^3\xi_\varepsilon^{-2}\, |\widetilde{q}|\,|\zeta|
	+
	s^{-2}\lambda^{-2}
	\int_{\omega\times(0,T)}e^{2s\varphi_\varepsilon}\xi_\varepsilon^{-2}
	|\phi|\,|\zeta|.
\end{align} 

 We claim that for  $s\ge \sigma_*\bigl(e^{2\lambda\|\eta_0\|_\infty}T+T^2\bigr)$ and $\sigma_*$ sufficiently large, it holds that
 \begin{equation}
 	\label{eq:reg-weight-comp}
 	e^{2s\varphi_\varepsilon}\xi_\varepsilon^{-2}\le e^{2s\varphi}\xi^{-2},
 	\qquad
 	e^{2s\varphi_\varepsilon}\xi_\varepsilon^{-4}\le e^{2s\varphi}\xi^{-4}.
 \end{equation}
	For every \((x,t)\in Q\), one has 
	\[
	\frac{\varphi(x,t)}{\varphi_\varepsilon(x,t)}
	=
	\frac{\xi(x,t)}{\xi_\varepsilon(x,t)}
	=
	\frac{(t+\varepsilon)(T-t+\varepsilon)}{t(T-t)}
	\ge 1.
	\]
	Therefore, for $m=2$ or $4$, we see that
	\[
	\log\!\left(
	\frac{e^{2s\varphi}\xi^{-m}}{e^{2s\varphi_\varepsilon}\xi_\varepsilon^{-m}}
	\right)
	=
	2s\bigl(\varphi-\varphi_\varepsilon\bigr)
	-m\log\!\left(\frac{\xi}{\xi_\varepsilon}\right)
		=
	2s\varphi_\varepsilon
	\left(
	\frac{\xi}{\xi_\varepsilon}-1
	\right)
	-m\log\!\left(\frac{\xi}{\xi_\varepsilon}\right).
	\] 
	Since \(\dfrac{\xi}{\xi_\varepsilon}\ge 1\), we use the elementary inequality $
	\log r\le r-1$ for all $r\ge 1$ to obtain
	\[
	\log\!\left(
	\frac{e^{2s\varphi}\xi^{-m}}{e^{2s\varphi_\varepsilon}\xi_\varepsilon^{-m}}
	\right)
	\ge
	\left(
	\frac{\xi}{\xi_\varepsilon}-1
	\right)
	\bigl(2s\varphi_\varepsilon-m\bigr).
	\]
	Choose $\varepsilon'>\varepsilon$ and $s\ge \sigma_*\bigl[e^{2\lambda\|\eta_0\|_\infty}T+T^2+(\varepsilon')^2\bigr]$, we compute
	\begin{align*}
		s\varphi_\varepsilon(x,t)
		=
		s\,
		\frac{e^{2\lambda\|\eta_0\|_\infty}-e^{\lambda\eta_0(x)}}
		{(t+\varepsilon)(T-t+\varepsilon)}
		\ge
		s\,
		\frac{e^{2\|\eta_0\|_\infty}-e^{\|\eta_0\|_\infty}}
		{\left(\frac{T}{2}+\varepsilon'\right)^2}
		\ge
		\sigma_*\bigl[T^2+(\varepsilon')^2\bigr]
		\frac{e^{2\|\eta_0\|_\infty}-e^{\|\eta_0\|_\infty}}
		{\left(\frac{T}{2}+\varepsilon'\right)^2}.
	\end{align*}
	Now choose $\sigma_*$ large enough, depending only on $\eta_0$, such that $
		s\varphi_\varepsilon(x,t) 
		\ge 2.$ Then we have  
	\[
	\log\!\left(
	\frac{e^{2s\varphi}\xi^{-m}}{e^{2s\varphi_\varepsilon}\xi_\varepsilon^{-m}}
	\right)\ge 0,
	\]
	which yields
	\[
	e^{2s\varphi_\varepsilon}\xi_\varepsilon^{-m}\le e^{2s\varphi}\xi^{-m}\qquad\text{for $m=2,4$.}
	\] 
	 
	 Putting \eqref{eq:reg-weight-comp} into the inequality \eqref{eq:energy-z-1-rig}, then
\begin{align}\label{eq:energy-eps-2}
	&\alpha s^{-2}\lambda^{-2}
	\int_Q e^{2s\varphi_\varepsilon}\xi_\varepsilon^{-2} |\nabla \zeta|^2
	\nonumber\\
	&\lesssim
	s^{-2}\lambda^{-2}
	\int_Q \bigl|\partial_t(e^{2s\varphi_\varepsilon}\xi_\varepsilon^{-2})\bigr|\,|\zeta|^2
	+
	s^{-2}\lambda^{-2}
	\int_Q |A|\,\bigl|\nabla(e^{2s\varphi_\varepsilon}\xi_\varepsilon^{-2})\bigr|
	\,|\nabla \zeta|\,|\zeta|
	\nonumber\\
	&\quad
	+
	s\lambda^2\int_Q\xi\, |\widetilde{q}|\,|\zeta|
	+
	s^{-2}\lambda^{-2}
	\int_{\omega\times(0,T)}e^{2s\varphi}\xi^{-2}
	|\phi|\,|\zeta|.
\end{align} 

Next, exactly as for the original weights, we have $\nabla \xi_\varepsilon=\lambda \xi_\varepsilon \nabla\eta_0$ and $\nabla \varphi_\varepsilon=-\lambda \xi_\varepsilon \nabla\eta_0$ and therefore, for every \(k\in\mathbb Z\),
\[
\nabla(e^{2s\varphi_\varepsilon}\xi_\varepsilon^k)
=
e^{2s\varphi_\varepsilon}\lambda \xi_\varepsilon^k \nabla\eta_0\,(-2s\xi_\varepsilon+k).
\]
Since we still have \((s\xi_\varepsilon)^{-1}\lesssim \sigma_*^{-1}\), hence
\begin{equation}
	\label{eq:reg-weight-grad}
	\bigl|\nabla(e^{2s\varphi_\varepsilon}\xi_\varepsilon^k)\bigr|
	\lesssim_{\eta_0}
	(1+|k|)\,s\lambda\,e^{2s\varphi_\varepsilon}\xi_\varepsilon^{k+1}.
\end{equation}
Similarly, the same computation gives
\begin{equation}
	\label{eq:reg-weight-time}
	\bigl|\partial_t(e^{2s\varphi_\varepsilon}\xi_\varepsilon^{-2})\bigr|
	\lesssim_{\eta_0}
	s^2\sigma_*^{-1}\,e^{2s\varphi_\varepsilon}.
\end{equation}

	Substituting \eqref{eq:reg-weight-grad} and \eqref{eq:reg-weight-time} into
\eqref{eq:energy-eps-2}, and using Young's inequality, we obtain
\begin{align*}
	&s^{-2}\lambda^{-2}
	\int_Q e^{2s\varphi_\varepsilon}\xi_\varepsilon^{-2} |\nabla \zeta|^2
	\nonumber\\
	&\lesssim_{\,\alpha,\eta_0,\gamma}
	(\lambda^{-2}\sigma_*^{-1}+1)
	\int_Q e^{2s\varphi}|\zeta|^2
	+
	s^2\lambda^4\int_Q e^{-2s\varphi}\xi^2 |\widetilde q|^2
	+
	s^{-4}\lambda^{-4}
	\int_{\omega\times(0,T)} e^{2s\varphi}\xi^{-4}|\phi|^2 .
\end{align*}
Using \eqref{ineq. 1/s xi } and \eqref{eq:z-v-first}, we have
\begin{align*}
	s^{-2}\lambda^{-2}
	\int_Q e^{2s\varphi_\varepsilon}\xi_\varepsilon^{-2} |\nabla \zeta|^2
\lesssim_{\,\alpha,\eta_0,\gamma}
\C{1}^2Y.
\end{align*}
Then, Fatou's lemma proves \eqref{eq:z-grad} for $s\ge \sigma_*\bigl[e^{2\lambda\|\eta_0\|_\infty}T+T^2\bigr]$ as $\varepsilon'>0$ is arbitrary. 
		 
\noindent
\textbf{Part 3. Transposition identity for \(\widetilde{q}\):} Combining \eqref{eq:z-v-first} and \eqref{eq:z-grad}, we conclude that
		\begin{equation}
			\label{eq:z-v-grad-final}
			\int_Q e^{2s\varphi}|\zeta|^2 
			+
			s^{-2}\lambda^{-2}\int_Q e^{2s\varphi}\xi^{-2} |\nabla \zeta|^2 
			+
			s^{-3}\lambda^{-4}
			\int_{\omega\times(0,T)} e^{2s\varphi}\xi^{-3}|\phi|^2 
			\lesssim_{\,\alpha,\eta_0,\gamma} Y.
		\end{equation}  
		We first multiply the equation $\partial_t\zeta-A\Delta\zeta
		=
		s^3\lambda^4 e^{-2s\varphi}\xi^3\widetilde q+\chi_\omega\phi$ in the strong sense with $\widetilde{q}$. And then we consider the weak formulation of \eqref{eq:adj-div} with the test function \(\zeta\). We subtract these two relations to obtain
		\begin{equation*}
			\int_Q \widetilde{q}\cdot (s^3\lambda^4 e^{-2s\varphi}\xi^3 \widetilde{q} + \phi\,\chi_\omega ) 
			=
			\int_Q F_0\cdot \zeta 
			-
			\sum_{i=1}^n \int_Q F_i\cdot \partial_i \zeta.   
		\end{equation*}  
 Therefore, the integral $Y$ defined in \eqref{eq. def of Y} can also be expressed as 
		\begin{equation*}
			Y
			=
			\int_Q F_0\cdot \zeta 
			-
			\sum_{i=1}^n \int_Q F_i\cdot \partial_i \zeta 
			-
			\int_{\omega\times(0,T)} \widetilde{q}\cdot \phi  .
		\end{equation*}
		By employing Cauchy--Schwarz inequality and then \eqref{eq:z-v-grad-final}, we infer that
		\begin{align*}
			Y\lesssim_{\,\alpha,\eta_0,\gamma} Y^{1/2}
			\Biggl[
			&\left(\int_Q e^{-2s\varphi}|F_0|^2 \right)^{1/2}\\
			&+
			s\lambda
			\sum_{i=1}^n
			\left(\int_Q e^{-2s\varphi}\xi^2 |F_i|^2 \right)^{1/2}
			+
			s^{3/2}\lambda^2
			\left(\int_{\omega\times(0,T)} e^{-2s\varphi}\xi^3 |\widetilde{q}|^2 \right)^{1/2}
			\Biggr].
		\end{align*}
		If $Y=0$, then \eqref{eq:carleman-div} is trivial. Otherwise divide by $Y^{1/2}$ and square.
		This yields exactly \eqref{eq:carleman-div}. 
			\end{proof}

	\section{Null Controllability of Linear Systems of Parabolic Equations}
	
	This section proves the null controllability in \Cref{thm null control of linear eq}. We first obtain the observability from the Carleman estimate in \Cref{lem carleman-div} and then use the duality argument to construct a $L^2$ control. Finally, the $L^\infty$ control can be constructed from some regularity estimate and cut-off arguments.
	
	\subsection{Observability} 
	
	The following proposition is the key observability estimate for the adjoint system. It says that the size of the adjoint solution at time \(0\) can be controlled only by observing the solution inside the set \(\omega\) during the time interval \((0,T)\). In other words, if an adjoint solution is non-trivial at time \(0\), then it must be detectable in the control region \(\omega\) at some time between \(0\) and \(T\). This is exactly the property needed for controllability: since every nontrivial adjoint solution must leave a detectable trace on \(\omega\), a control supported in \(\omega\) can drive the forward system.
	
	\begin{lemma}[Observability for the adjoint system]
		\label{prop:observability}
		Let $T>0$, $q_T^*\in L^2(\mathbb T^n;\mathbb R^2)$, $\omega\subseteq \mathbb T^n$ be a non-empty open set and $B_i,\Gamma\in L^\infty(\mathbb T^n \times (0,T);\mathbb R^{2\times 2})$. Then there exists a constant $\C{5}=\C{5}(T,\omega,\alpha,\gamma,\|B\|_{\infty},\|\Gamma\|_{\infty})>0$ 
		such that the weak solution $q^*$ of
		\begin{equation}
			\label{eq:adjoint}
			\begin{cases}
				-\p_t u -A^\top \Delta u+\displaystyle\sum_{i=1}^n \partial_i(B_i^\top u)-\Gamma^\top u=0
				&\text{in }\mathbb T^n \times (0,T),\\[1mm]
				u(\cdot,T)=q^*_T &\text{in }\mathbb T^n.
			\end{cases}
		\end{equation}
		satisfies
		\begin{equation}
			\label{eq:obs-standard}
			\|q^*(\cdot,0)\|_{L^2(\mathbb T^n)}^2
			\le
			\C{5}
			\int_{\omega\times(0,T)} |q^*|^2 .
		\end{equation}
        We can choose the constant $$\C{5}
=
\dfrac{C_{\alpha,\gamma,\omega}}{T}
\exp \Big[
\pig(C\|\Gamma\|_{\infty}+\alpha^{-1}\|B\|_{\infty}^2\pig)T
+
C_{\alpha,\gamma,\omega} \pig(1+\|B\|_{\infty}^2+\|\Gamma\|_{\infty}^2\pig)(C_{\alpha,\gamma,\omega}+T^{-1}) 
\Big],$$ where $C$ is a universal constant and $\|B\|_{\infty}:=\pig(\sum\|B_i\|_{\infty}^2\pigr)^{1/2}$.
	\end{lemma}

	
	\begin{proof} In \Cref{lem carleman-div}, we choose $F_0=\Gamma^\top q^*$ and  $F_i=-B_i^\top q^* $. Hence, we obtain
		\begin{align}
			\label{eq:obs-preabsorb}
			&\int_Q e^{-2s\varphi}
			\Bigl(
			s^3\lambda^4\xi^3 |q^*|^2
			\Bigr) 
			\nonumber\\
			&\le
			\C{2}\Biggl(
			\|\Gamma\|_{\infty}^2 \int_Q e^{-2s\varphi}|q^*|^2 
			+
			s^2\lambda^2 \|B\|_{\infty}^2 \int_Q e^{-2s\varphi}\xi^2 |q^*|^2  
			+
			s^3\lambda^4
			\int_{\omega\times(0,T)} e^{-2s\varphi}\xi^3 |q^*|^2 
			\Biggr),
		\end{align}
for every $
		\lambda\ge \lambda_*$ and $s\ge \sigma_*\bigl(e^{2\lambda\|\eta_0\|_\infty}T+T^2\bigr)$. By \eqref{ineq. 1/s xi }, we can further enlarge $\sigma_*$, depending on $\alpha$, $\eta_0$, $\gamma$, $\|B\|_{\infty}$ and $\|\Gamma\|_{\infty}$, that the first two terms in the right-hand side of
		\eqref{eq:obs-preabsorb} can be absorbed by the left-hand side. Hence
		\begin{equation}
			\label{eq:obs-carleman}
			\int_Q e^{-2s\varphi}\xi^3 |q^*|^2 
			\lesssim_{\,\alpha,\eta_0,\gamma}
			\int_{\omega\times(0,T)} e^{-2s\varphi}\xi^3 |q^*|^2
		\end{equation}
		for any \(s\ge \sigma_*(e^{2\lambda \|\eta_0\|_\infty}T+T^2)\), $\lambda\geq \lambda_*$ and $\sigma_* \gtrsim_{\,\alpha,\eta_0,\gamma} 1+\|B\|_{\infty}^2+\|\Gamma\|_{\infty}^2$.
		
		We now compare the values of the weights.
		Since \(t\mapsto t(T-t)\) attains its maximum value at \(t=T/2\), and bounded from below by
		\(3T^2/16\) on \([T/4,3T/4]\), we have
		\[
		M_\varphi:=\sup_{(x,t)\in \mathbb T^n\times(T/4,3T/4)} \varphi(x,t)
		\le
		\frac{16}{3T^2}\,e^{2\lambda\|\eta_0\|_\infty},
		\]
		while
		\[
		m_\varphi := \inf_{(x,t)\in Q} \varphi(x,t) 
		=
		\frac{4}{T^2}\Bigl(e^{2\lambda\|\eta_0\|_\infty}
		-e^{\lambda\|\eta_0\|_\infty}\Bigr).
		\]
		Likewise,
		\[
		\inf_{(x,t)\in \mathbb T^n\times(T/4,3T/4)} \xi(x,t)\gtrsim \frac{1}{T^2}.
		\]
We express
\begin{align*}
e^{-2s\varphi(x,t)}\xi(x,t)^3
&=
\frac{e^{3\lambda\eta_0(x)}}{[t(T-t)]^3}
\exp \left(
-\frac{2s\bigl(e^{2\lambda\|\eta_0\|_\infty}-e^{\lambda\eta_0(x)}\bigr)}{t(T-t)}
\right)\\
&\le
e^{3\lambda\|\eta_0\|_\infty}
\frac{1}{[t(T-t)]^3}
\exp \left(
-\frac{2s\bigl(e^{2\lambda\|\eta_0\|_\infty}-e^{\lambda\|\eta_0\|_\infty}\bigr)}{t(T-t)}
\right).
\end{align*} 
Now we use the elementary inequality $
r^3e^{-r}\le 27e^{-3}$ for all $r\ge 0$ to get
\[
\begin{aligned}
&\frac{1}{[t(T-t)]^3}
\exp \left(
-\frac{2s\bigl(e^{2\lambda\|\eta_0\|_\infty}-e^{\lambda\|\eta_0\|_\infty}\bigr)}{t(T-t)}
\right)
\\
&=
\frac{1}{
\bigl[2s\bigl(e^{2\lambda\|\eta_0\|_\infty}-e^{\lambda\|\eta_0\|_\infty}\bigr)\bigr]^3}
\left[
\frac{2s\bigl(e^{2\lambda\|\eta_0\|_\infty}-e^{\lambda\|\eta_0\|_\infty}\bigr)}{t(T-t)}
\right]^3
\exp \left(
-\frac{2s\bigl(e^{2\lambda\|\eta_0\|_\infty}-e^{\lambda\|\eta_0\|_\infty}\bigr)}{t(T-t)}
\right)
\\
&\le
\frac{27}{e^3}\,
\frac{1}{
\bigl[2s\bigl(e^{2\lambda\|\eta_0\|_\infty}-e^{\lambda\|\eta_0\|_\infty}\bigr)\bigr]^3}.
\end{aligned}
\]
Thus
\[
e^{-2s\varphi(x,t)}\xi(x,t)^3
\le
\frac{27}{8e^3}\,
\frac{e^{3\lambda\|\eta_0\|_\infty}}
{s^3\bigl(e^{2\lambda\|\eta_0\|_\infty}-e^{\lambda\|\eta_0\|_\infty}\bigr)^3}.
\]
Since $
s\ge \sigma_*\bigl(e^{2\lambda\|\eta_0\|_\infty}T+T^2\bigr)$, we have
\begin{align}\label{4661}
e^{-2s\varphi(x,t)}\xi(x,t)^3
\le
\frac{27}{8e^3\sigma_*^3}\,
\frac{1}
{\bigl(e^{\lambda\|\eta_0\|_\infty}-1\bigr)^3}\,
\frac1{T^6}.
\end{align}
		
		Using \eqref{eq:obs-carleman},
		\[
		\frac{1}{T^6}e^{-2sM_\varphi}
		\int_{\mathbb T^n\times(T/4,3T/4)} |q^*|^2 
		\lesssim \int_{\omega\times(0,T)} e^{-2s\varphi}\xi^3 |q^*|^2 .
		\]
        for any \(s\ge \sigma_*(e^{2\lambda \|\eta_0\|_\infty}T+T^2)\) and $\lambda\geq \lambda_*$. On the other hand, we apply \eqref{4661} to deduce 
		\[
		\int_{\omega\times(0,T)} e^{-2s\varphi}\xi^3 |q^*|^2 
		\lesssim T^{-6} \frac{1}
{\bigl(e^{\lambda\|\eta_0\|_\infty}-1\bigr)^3}
		\int_{\omega\times(0,T)} |q^*|^2 .
		\]
 Combining the two previous inequalities, we obtain 
		\begin{equation}
			\label{eq:mid-obs}
		\int_{\mathbb T^n\times(T/4,3T/4)} |q^*|^2 
\lesssim \frac{e^{2sM_\varphi}}
{\bigl(e^{\lambda\|\eta_0\|_\infty}-1\bigr)^3}
\int_{\omega\times(0,T)} |q^*|^2 
\leq \frac{\exp\big[C s T^{-2} e^{2\lambda\|\eta_0\|_\infty}\big]}
{\bigl(e^{\lambda\|\eta_0\|_\infty}-1\bigr)^3}
\int_{\omega\times(0,T)} |q^*|^2 .
		\end{equation}
		for some universal constant $C>0$. It remains to estimate $\|q^*(\cdot,0)\|_{L^2}$ in terms of the left-hand side of
		\eqref{eq:mid-obs}. Multiplying \eqref{eq:adjoint} by $q^*$ and integrating over
		$\mathbb T^n$, we employ Lions–Magenes lemma to obtain
		\[
		-\frac12 \frac{d}{dt}\|q^*(\cdot,t)\|_{L^2}^2
		+
		\alpha \|\nabla q^*(\cdot,t)\|_{L^2}^2
		=
		\sum_{i=1}^n \int_{\mathbb T^n} B_i^\top q^* \cdot \partial_i q^*
		+
		\int_{\mathbb T^n} \Gamma^\top q^* \cdot q^*.
		\]
		for a.e. $t\in[0,T]$. By Young's inequality,
		\[
		-\frac{d}{dt}\|q^*(\cdot,t)\|_{L^2}^2
		\leq (2\|\Gamma\|_{\infty}+\alpha^{-1}\|B\|_{\infty}^2)\,\|q^*(\cdot,t)\|_{L^2}^2.
		\]
		Hence, as $q^*\in C([0,T];L^2(\mathbb T^n;\mathbb R^2))$, we have
		\[
		\|q^*(\cdot,0)\|_{L^2}^2
		\le
		e^{(2\|\Gamma\|_{\infty}+\alpha^{-1}\|B\|_{\infty}^2)t}
		\|q^*(\cdot,t)\|_{L^2}^2\qquad \text{ for every $t\in[T/4,3T/4]$,}
		\]
		Integrating this inequality with respect to $t$ on $(T/4,3T/4)$, we infer
		\[
		\|q^*(\cdot,0)\|_{L^2}^2
		\le
		\frac{2}{T}
		e^{3(2\|\Gamma\|_{\infty}+\alpha^{-1}\|B\|_{\infty}^2)T/4}
		\int_{\mathbb T^n\times(T/4,3T/4)} |q^*|^2 .
		\]
		Combining this with \eqref{eq:mid-obs}, we obtain \eqref{eq:obs-standard} by setting \(s= \sigma_*(e^{2\lambda \|\eta_0\|_\infty}T+T^2)\) and $\lambda=\lambda_*$ with $\sigma_* \gtrsim_{\,\alpha,\eta_0,\gamma} 1+\|B\|_{\infty}^2+\|\Gamma\|_{\infty}^2$.
	\end{proof}
	
	The estimate in \Cref{prop:observability} shows that the map sending an adjoint solution to its value at time \(0\) is continuous with respect to its restriction to the set \(\omega\times(0,T)\). This allows us to define a control as the Riesz representative of the corresponding linear functional, and the duality identity then shows that the associated forward solution satisfies \(u(\cdot,T)=0\).
	
	\begin{lemma}[\bf Null controllability with $L^2$ control]
		\label{thm:null-const-torus} Assume that $B_i$ and $\Gamma\in L^\infty(\mathbb T^n \times (0,T);\mathbb R^{2\times 2})$ for all $i=1,2,\dots,n$. For any $T>0$, non-empty open set $\omega\subseteq \mathbb T^n$ and initial datum $u_0\in L^2(\mathbb T^n;\mathbb R^2)$, there exists a control $
		f\in L^2(\omega\times(0,T);\mathbb R^2)$
		such that the weak solution $
		u\in C([0,T];L^2(\mathbb T^n;\mathbb R^2))
		\cap L^2(0,T;H^1(\mathbb T^n;\mathbb R^2))$ of
		\begin{equation}
			\label{eq:forward}
			\begin{cases}
				\p_t u-A \Delta u-\displaystyle\sum_{i=1}^n B_i\partial_i u-\Gamma u
				=
				\chi_\omega f
				&\text{in }\mathbb T^n \times (0,T),
				\\[1mm]
				u(\cdot,0)=u_0 &\text{in }\mathbb T^n,
			\end{cases}
		\end{equation}
		satisfies
		\[
		u(\cdot,T)=0
		\qquad\text{ a.e. in }\mathbb T^n.
		\]
		Moreover, the control satisfies
		\[
		\|f\|_{L^2(\omega\times(0,T))}
		\le \C{5}^{1/2} \|u_0\|_{L^2(\mathbb T^n)},
		\]
		where $\C{5}$ is the constant in \Cref{prop:observability}. 

	\end{lemma} 
	
	\begin{proof}  For each \(q_T\in L^2(\mathbb T^n;\mathbb R^2)\), we denote \(q=\mathscr{L}[q_T]\)  the
		solution of the adjoint system \eqref{eq:adjoint} subject to the terminal data $q_T(\cdot)$. Consider the set
		\[
		\mathcal Y_0
		:=
		\pigl\{
		q \big|_{\omega\times(0,T)}=\mathscr{L}[q_T]\big|_{\omega\times(0,T)}  :\ q_T\in L^2(\mathbb T^n;\mathbb R^2)
		\pigr\}
		\subset L^2(\omega\times(0,T);\mathbb R^2).
		\]
		Define a linear functional \(G\) on \(\mathcal Y_0\) by
		\[
		G\!\left(q \big|_{\omega\times(0,T)}\right)
		:=
		-\int_{\mathbb T^n} u_0\cdot q (\cdot,0).
		\] 
		This is well defined, indeed, if $\check{q}=\mathscr{L}[\check{q}_T]$ for some $\check{q}_T \in L^2(\mathbb T^n;\mathbb R^2)$, then \(w:=q-\check{q}\) satisfies the adjoint equation \eqref{eq:adjoint}. If moreover $
		q  
		=
		\check{q} $ a.e. on $\omega\times(0,T)$, then $w$ vanishes on
		\(\omega\times(0,T)\). Hence, by Proposition~\ref{prop:observability},
		\[
		\|w(\cdot,0)\|_{L^2(\mathbb T^n)}^2
		\le
		\C{5} \int_{\omega\times(0,T)} |w|^2
		=0,
		\]
		so \(q (\cdot,0)=\check{q}(\cdot,0)\) a.e. in $\mathbb{T}^n$.
		
		Moreover, by Proposition~\ref{prop:observability},
		\[
		\left|G\!\left(q \big|_{\omega\times(0,T)}\right)\right|
		\le
		\|u_0\|_{L^2(\mathbb T^n)}
		\|q (\cdot,0)\|_{L^2(\mathbb T^n)}
		\le
		\C{5}^{1/2}\|u_0\|_{L^2(\mathbb T^n)}
		\|q \|_{L^2(\omega\times(0,T))}.
		\]
		Hence \(G\) is continuous on \(\mathcal Y_0\) in the \(L^2(\omega\times(0,T))\)-norm. Let \(\overline{\mathcal Y_0}\) be the completion of \(\mathcal Y_0\) with
		\(L^2(\omega\times(0,T))\)-norm. Then \(G\) extends uniquely to a continuous
		linear functional on \(\overline{\mathcal Y_0}\), still denoted by \(G\). By the Riesz representation theorem, there exists a unique
		\(f\in \overline{\mathcal Y_0}\subseteq L^2(\omega\times(0,T);\mathbb R^2)\) such that
		\begin{equation}
			\label{eq:HUM-Riesz-Y}
			\int_{\omega\times(0,T)} f\cdot y
			=
			G(y)
			\h{5pt} \text{for all  $y\in \overline{\mathcal Y_0}$.}
		\end{equation}
		In particular,
		\[
		\|f\|_{L^2(\omega\times(0,T))}
		=
		\|G\|_{\overline{\mathcal Y_0}^{\prime}}
		\le
		\C{5}^{1/2}\|u_0\|_{L^2(\mathbb T^n)}.
		\]
		 
		Let \(u\) be the weak solution of \eqref{eq:forward} corresponding to this control \(f\). By using the weak formulation of 
		\eqref{eq:forward} with the test function \(q=\mathscr{L}[q_T]\) and then that of the adjoint system \eqref{eq:adjoint} with the test function \(u\), we subtract the two relations and apply Lions–Magenes lemma to obtain
		\[
		\int_{\mathbb T^n} u(\cdot,T)\cdot q_T
		-
		\int_{\mathbb T^n} u_0\cdot q(\cdot,0)
		=
		\int_{\omega\times(0,T)} f\cdot q.
		\]
		Since \(q|_{\omega\times(0,T)}\in \mathcal Y_0\subseteq \overline{\mathcal Y_0}\), the defining property
		\eqref{eq:HUM-Riesz-Y} yields
		\[
		\int_{\omega\times(0,T)} f\cdot q
		=
		G\!\left(q|_{\omega\times(0,T)}\right)
		=
		-\int_{\mathbb T^n} u_0\cdot q(\cdot,0).
		\]
		Therefore
		\[
		\int_{\mathbb T^n} u(\cdot,T)\cdot q_T=0 \h{5pt} \text{for any $q_T\in L^2(\mathbb T^n;\mathbb R^2)$}.
		\]
		Hence \(u(\cdot,T)=0\) a.e. in \(\mathbb T^n\) .
	\end{proof}

	\subsection{Construction of $L^\infty$ control}
	We now turn to the construction of a $L^\infty$ control. We need an additional ingredient about the interior regularity result, which allows us to upgrade local \(L^\infty\) bounds on the forcing term into local \(W^{1,\infty}\) bounds for the solution. The proof of the following lemma can be proved in the same way as in \cite[Lemma 1]{FernándezCara06}.
	
	\begin{lemma}[\bf Interior regularity]
		\label{lem:torus-local-reg} 
		Let \(\mathcal{O}',\mathcal{O}\subset\mathbb T^n\) be non-empty open sets such that $
		\mathcal{O}'\subset\subset \mathcal{O}.$ Assume that \(B_i,\Gamma\in L^\infty\big(\mathbb T^n \times (0,T);\mathbb R^{2\times 2}\big)\) and \(F\in L^2\big(\mathbb T^n \times (0,T);\mathbb R^2\big)\cap L^\infty\big(\mathcal{O}\times(0,T);\mathbb R^2\big)\). For every \(\delta\in(0,T)\), there exists a constant $
		\C{6}=
		\C{6}(n,\mathcal{O},\mathcal{O}',\alpha,\gamma)$
		such that the weak solution 
		\(u\in L^2\big(0,T;H^1(\mathbb T^n;\mathbb R^2)\big)
		\cap C\big([0,T];L^2(\mathbb T^n;\mathbb R^2)\big)\) of \[
		\p_t u-A\Delta u-\sum_{i=1}^n B_i \partial_i u-\Gamma u=F
		\qquad\text{in }\mathbb T^n \times (0,T),
		\] satisfies
\begin{align*}
		\|u\|_{L^\infty(\delta,T;W^{1,\infty}(\mathcal{O}'))}
		\le
		\C{6}'
		\Bigl(
		\|u\|_{L^2(0,T;H^1)}
        +\|u\|_{C([0,T];L^2)}
		+
		\|F\|_{L^\infty(\mathcal{O}\times(0,T))}
		\Bigr).
\end{align*}
	where $	\C{6}'=\C{6}(1+T^{1/2}+T^{1/4})^n(1+\delta^{-1}+|A|+\|B\|_{\infty}+\|\Gamma\|_{\infty})^{n+1}$. If in addition \(u(\cdot,0)=0\), then the same conclusion holds with \(\delta=0\) and $\C{6}'=\C{6}(1+T^{1/2}+T^{1/4})^n(1+|A|+\|B\|_{\infty}+\|\Gamma\|_{\infty})^{n+1}$.
	\end{lemma}
We are now in a position to prove \Cref{thm null control of linear eq}. The idea is to start from the \(L^2\) control given by \Cref{thm:null-const-torus}, localize it inside a smaller subset of \(\omega\), and then use suitable cut-off functions in space and time to construct a new control supported in \(\omega\). The interior regularity lemma is then applied to the corresponding solutions in order to show that this new control actually belongs to \(L^\infty(\omega\times(0,T))\).

	\begin{proof}[\bf Proof of \Cref{thm null control of linear eq}]  Choose open sets
		\[
		\omega_4\subset\subset\omega_3\subset\subset\omega_2\subset\subset\omega_1\subset\subset\omega.
		\]
		By \Cref{thm:null-const-torus} applied with control region \(\omega_4\), there exists $
		f^*\in L^2(\omega_4\times(0,T);\mathbb R^2)$ such that 
		\begin{equation}\label{eq:fstar-L2}
			\|f^*\|_{L^2(\omega_4\times(0,T))}
			\le
			\C{5}^{1/2}\|u_0\|_{L^2(\mathbb T^n)}.
		\end{equation} and the corresponding solution \(u^*\) of \eqref{eq:forward}, subject to initial condition $u_0$ and control $f^*$, satisfies $
		u^*(\cdot,T)=0$ a.e. in $\mathbb{T}^n$.

		\smallskip
		\noindent
		\textbf{Step 1. Cut-off in space and time:}
		Let \(\Psi_2\in C^\infty([0,T];[0,1])\) satisfy
		\[
		\Psi_2(t)=1 \text{ on }[0,T/4),
		\qquad
		\Psi_2(t)=0 \text{ on }(3T/4,T],
		\qquad
		|\Psi_2'(t)|\le \frac{4}{T} \text{ on }[0,T].
		\]
		Let \(\Upsilon\in C\big([0,T];L^2(\mathbb T^n;\mathbb R^2)\big)\cap L^2\big(0,T;H^1(\mathbb T^n;\mathbb R^2)\big)\) be the weak solution of the homogeneous system
		\[
		\begin{cases}
			\p_tu -A\Delta u-\displaystyle\sum_{i=1}^n B_i\partial_i u-\Gamma  u=0
			&\text{in }Q,
			\\[1mm]
			u(\cdot,0)=u_0
			&\text{in }\mathbb T^n.
		\end{cases}
		\]
		Then \(w^*:=u^*-\Psi_2\Upsilon\) solves
		\begin{equation*}
				\begin{cases}
			\p_t u-A\Delta u-\displaystyle\sum_{i=1}^n B_i\partial_i u-\Gamma u
			=
			-\Psi_2'\Upsilon+\chi_{\omega_4}f^*
			&\text{in }Q,
			\\[1mm]
			u(\cdot,0)=u(\cdot,T)=0
			&\text{in }\mathbb T^n.
		\end{cases}
		\end{equation*}  
		Choose \(\Psi_3\in C_c^2(\omega_2;[0,1])\) such that $
		0\le \Psi_3\le 1$ and $ \Psi_3\equiv 1$ on $\omega_3$. Set $
		w^\dagger:=(1-\Psi_3)w^*.$ Since \(\omega_4\subset\subset\omega_3\) and \(\Psi_3\equiv1\) on \(\omega_3\), we have
		\[
		(1-\Psi_3)\chi_{\omega_4}f^*=0.
		\]
		A direct computation gives $
		\partial_i w^\dagger=(1-\Psi_3)\partial_i w^*-(\partial_i\Psi_3)w^*$ and $\Delta w^\dagger=(1-\Psi_3)\Delta w^*-2(\nabla\Psi_3\cdot\nabla) w^*-(\Delta\Psi_3)w^*.$ Therefore
		\[
		\begin{aligned}
			&\p_t w^\dagger-A\Delta w^\dagger-\sum_{i=1}^n B_i\partial_i w^\dagger-\Gamma w^\dagger
			\\
			&=
			-(1-\Psi_3)\Psi_2'\Upsilon
			+
			A\Bigl[2(\nabla\Psi_3\cdot\nabla) w^*+(\Delta\Psi_3)w^*\Bigr]
			+
			\sum_{i=1}^n B_i(\partial_i\Psi_3)w^* .
		\end{aligned}
		\]
		Then $
		u^\dagger:=w^\dagger+\Psi_2\Upsilon$
		solves
		\begin{equation*}
			\begin{cases}
			\p_t u-A\Delta u-\displaystyle\sum_{i=1}^n B_i\partial_i u-\Gamma u=\chi_\omega f_0,
				&\text{in }Q,
				\\[1mm]
				u(\cdot,0)=u_0,\qquad u(\cdot,T)=0
				&\text{in }\mathbb T^n,
			\end{cases}
		\end{equation*}   
		with
		\begin{equation}\label{eq:new-control}
			f_0
			=
			\Psi_2'\Psi_3\Upsilon
			+
			A\Bigl[2(\nabla\Psi_3\cdot\nabla) w^*+(\Delta\Psi_3)w^*\Bigr]
			+
			\sum_{i=1}^n B_i(\partial_i\Psi_3)w^* .
		\end{equation}
		Since \(\textup{supp}( \Psi_3)\subset\omega_2\subset\subset\omega\), the function $f_0$ in \eqref{eq:new-control} is supported
		in \(\omega\times(0,T)\).

		\noindent
		\textbf{Step 2. \(L^\infty\)-estimate of the new control $f_0$.}
		We first estimate \(\Psi_2'\Psi_3\Upsilon\). By Lemma~\ref{lem:torus-local-reg}, applied to \(\Upsilon\) with \(\mathcal{O}=\omega\), \(\mathcal{O}'=\omega_2\),
		and \(\delta=T/8\), we obtain
		\[
		\|\Upsilon\|_{L^\infty(T/8,T;W^{1,\infty}(\omega_2))}
		\le \C{6}' \big( \|\Upsilon\|_{L^2(0,T;H^1)}
        +\|\Upsilon\|_{C([0,T];L^2)}\big).
		\]
		Since \(\Psi_2'\equiv0\) on \((0,T/4)\) and $\Psi_3=0$ outside $\omega_2$, it follows that
		\begin{equation}\label{eq:eta-chi-est}
			\|\Psi_2'\Psi_3\Upsilon\|_{L^\infty(\omega\times(0,T))}
			\lesssim
			\frac{1}{T}\,\|\Upsilon\|_{L^\infty(\omega_2\times(T/8,T))}
			\leq
			\frac{1}{T} \C{6}' \big( \|\Upsilon\|_{L^2(0,T;H^1)}
        +\|\Upsilon\|_{C([0,T];L^2)}\big).
		\end{equation}
		Replace $\mathcal O'=\omega_2$ by $\mathcal O'=\omega_1$. It yields
		\[
		\|\Upsilon\|_{L^\infty(\omega_1\times(T/8,T))}
		\le\|\Upsilon\|_{L^\infty(T/8,T;W^{1,\infty}(\omega_1))}
		\le \C{6}' \big( \|\Upsilon\|_{L^2(0,T;H^1)}
        +\|\Upsilon\|_{C([0,T];L^2)}\big).
		\]  
		Now, by construction of \(\Psi_2\), we have \(\textup{supp} (\Psi_2')\subseteq [T/4,3T/4]\subset [T/8,T]\). Therefore,
\begin{align}\label{5193}
	\|\Psi_2'\Upsilon\|_{L^\infty(\omega_1\times(0,T))} 
	\le
	\|\Psi_2'\|_{L^\infty(0,T)}
	\|\Upsilon\|_{L^\infty(\omega_1\times(T/8,T))}
	\le
	\frac{1}{T}\, \C{6}'\big( \|\Upsilon\|_{L^2(0,T;H^1)}
        +\|\Upsilon\|_{C([0,T];L^2)}\big).
\end{align}
		Consequently, as $\chi_{\omega_4}f^*=0$ on $\omega_1\setminus\overline{\omega_4}$, we obtain
		\[
		-\Psi_2'\Upsilon+\chi_{\omega_4}f^*
		\in L^\infty( (\omega_1\setminus\overline{\omega_4})\times(0,T))
		\]
		Hence, applying Lemma~\ref{lem:torus-local-reg} to \(w^*\) with $
		\mathcal{O}'=\omega_2\setminus\overline{\omega_3}\subset\subset \mathcal{O}=\omega_1\setminus\overline{\omega_4}$,
		we get
		\[
		\|w^*\|_{L^\infty(0,T;W^{1,\infty}(\omega_2\setminus\overline{\omega_3}))}
		\le \C{6} '
		\Bigl(
        \|w^*\|_{L^2(0,T;H^1)}
        +\|w^*\|_{C([0,T];L^2)} 
		+
		\|\Psi_2'\Upsilon\|_{L^\infty(\omega_1\times(0,T))}
		\Bigr).
		\] 
		We observe that $
		\textup{supp}(\nabla\Psi_3)\cup \textup{supp}(\Delta\Psi_3)
		\subset \omega_2\setminus\overline{\omega_3},$ whereas the \(L^2\)-control term \(\chi_{\omega_4}f^*\) is supported in
		\(\omega_4\subset\subset\omega_3\). Consequently,
		\begin{equation}\label{eq:cutoff-terms-est}
			\begin{aligned}
				&\left\|
				A\Bigl[2(\nabla\Psi_3\cdot\nabla) w^*+(\Delta\Psi_3)w^*\Bigr]
				+
				\sum_{i=1}^n B_i(\partial_i\Psi_3)w^*
				\right\|_{L^\infty(\omega\times(0,T))}
				\\
				&\lesssim_{\,\omega}
				\Bigl(
				 |A|+\|B\|_{\infty}
				\Bigr)
				\|w^*\|_{L^\infty(0,T;W^{1,\infty}(\omega_2\setminus\overline{\omega_3}))}
				\\
				&\le \C{6}' \Bigl(
        \|w^*\|_{L^2(0,T;H^1)}
        +\|w^*\|_{C([0,T];L^2)} 
		+
		\|\Psi_2'\Upsilon\|_{L^\infty(\omega_1\times(0,T))}
		\Bigr).
			\end{aligned}
		\end{equation}
		 
		The standard $L^2$ energy estimate and Gr\"onwall's inequality applied to \(\Upsilon\) yield
		\begin{equation}\label{eq:chi-Y}
			 \|\Upsilon\|_{L^2(0,T;H^1)}
        +\|\Upsilon\|_{C([0,T];L^2)}
			\leq C_{\alpha,\|B\|_{\infty},\|\Gamma\|_{\infty}}e^{C_{\alpha,\|B\|_{\infty},\|\Gamma\|_{\infty}}T}
			\|u_0\|_{L^2(\mathbb T^n)}.
		\end{equation}
		The same standard estimate applied to \(w^*\), together with \eqref{eq:fstar-L2} and \eqref{eq:chi-Y}, it gives
		\begin{align}\label{eq:wstar-Y}
			 \|w^*\|_{L^2(0,T;H^1)}
        +\|w^*\|_{C([0,T];L^2)} 
			&\leq C_{\alpha,\|B\|_{\infty},\|\Gamma\|_{\infty}}e^{C_{\alpha,\|B\|_{\infty},\|\Gamma\|_{\infty}}T}
			\Bigl(
			\|\Psi_2'\Upsilon\|_{L^2(Q)}+\|f^*\|_{L^2(\omega_4\times(0,T))}
			\Bigr)\nonumber\\
            &\leq \pig[C_{\alpha,\|B\|_{\infty},\|\Gamma\|_{\infty}}T^{-1}e^{C_{\alpha,\|B\|_{\infty},\|\Gamma\|_{\infty}}T}+\C{5}^{1/2}\pig]\|u_0\|_{L^2(\mathbb T^n)}.
		\end{align} 
		Putting \eqref{eq:wstar-Y}, \eqref{eq:chi-Y} and \eqref{5193} into \eqref{eq:cutoff-terms-est}; and also inserting \eqref{eq:chi-Y} into \eqref{eq:eta-chi-est}, the definition of $f_0$ in \eqref{eq:new-control} implies $
		\|f_0\|_{L^\infty(\omega\times(0,T))}
		\le
		\C{4}\,\|u_0\|_{L^2(\mathbb T^n)}.$ Here the constant 
        \begin{align}
		\C{4}=P_{\alpha,\gamma,\omega,\|B\|_{\infty},\|\Gamma\|_{\infty}}(T)e^{C_{\alpha,\gamma,\omega,\|B\|_{\infty},\|\Gamma\|_{\infty}}T}\pig[1+e^{C_{\alpha,\gamma,\omega,\|B\|_{\infty},\|\Gamma\|_{\infty}}\frac{1}{T}}\pig]\label{5129}
		\end{align} where $P_{\alpha,\|B\|_{\infty},\|\Gamma\|_{\infty}}(T)$ is a polynomial in $T$ (may have negative degree and the degrees are universal) with coefficient depending on $\alpha$, $\gamma$, $\omega$, $\|B\|_{\infty}$, $\|\Gamma\|_{\infty}$.
	\end{proof}

	\section{Fixed point arguments and proofs of main results}\label{sec. Fixed point arguments and proof of main results}
	We first establish the maximal regularity result for the equation \eqref{eq. projected and linearised eq}:
	\begin{lemma}\label{lem higher est}
		Fix $T>0$, non-empty open set $\omega\subseteq \mathbb T^2$ and $p\in(1,\infty)$. Assume that $B_1,\,B_2,\,\Gamma \in L^\infty(\mathbb{T}^2\times(0,T);\mathbb{R}^{2\times 2})$, $f\in L^p\big(\omega \times(0,T)\big)$ and $v_0\in W^{2,p}(\mathbb{T}^2;\mathbb{R}^2)$. Then there is a constant $\C{7}=\C{7}(p,T,\alpha,\gamma,\|B\|_{\infty},\|\Gamma\|_\infty)$ such that the solution $v$ of equation \eqref{eq. general linear eq.}, subject to $v(\cdot,0)=v_0$ and control $f$, satisfies 
		$$ \|v\|_{L^p(0,T;W^{2,p})} 
		+\|\p_t v\|_{L^p(0,T;L^{p})} 
		\leq \C{7}\pig(\|v_0\|_{W^{2,p}(\mathbb{T}^2;\mathbb{R}^2)}
		+\| f \|_{L^p(\omega \times(0,T))}\pig).$$
	\end{lemma}
	\begin{proof}
		It is a direct consequence of \cite[Theorem 2.1]{denk2007optimal}, with setting $m=1$, $G=\mathbb{T}^2$, $E=\mathbb{R}^2$, $\mathcal{B}(E)=\mathbb{R}^{2\times 2}$, $\mathcal{A}(t,x,D)=\sum_{|\alpha|\leq 2}a_\alpha D^\alpha=-A(\Delta+B\cdot \nabla +c)$ and $u_0=v_0$. Condition (E) in \cite{denk2007optimal} is satisfied as $\mathcal{A}_{\#}(t,x,D)=-A\Delta$ and hence its spectrum is $\sigma(\mathcal{A}_{\#})=\sigma(\alpha I+\gamma J)=\{\alpha\pm i\gamma\}\subset\mathbb{C}_+$. Condition (SD) is also satisfied as all the coefficients of the equation is bounded and $A$ is constant. 
        Condition (D) is valid as well by the assumption of this lemma. Therefore, we can apply \cite[Theorem 2.1]{denk2007optimal} to obtain a bijective bounded linear operator from the solution space $W^{1,p}\big(0,T;L^{p}(\mathbb{T}^2;\mathbb{R}^2)\big)\cap L^{p}\big(0,T;W^{2,p}(\mathbb{T}^2;\mathbb{R}^2)\big)$ to the product space containing the initial condition and non-homogeneous term of the equation. Thus, we obtain
		$$ \|v\|_{L^p(0,T;W^{2,p})} 
		+\|\p_t v\|_{L^p(0,T;L^{p})} 
		\leq C_{p,T,\alpha,\gamma,\|B\|_{\infty},\|\Gamma\|_\infty}\pig(\|v_0\|_{B^{2(1-1/p)}_{p,p}(\mathbb{T}^2;\mathbb{R}^2)}
		+\| \chi_\omega f\|_{L^p(\mathbb{T}^2 \times(0,T))}\pig)$$
		where $B^{2(1-1/p)}_{p,p}$ is the Besov space. The estimate is concluded by the continuous embedding from $W^{2,p}$ into $B^{2(1-1/p)}_{p,p}$.
	\end{proof}

	A simplified version of \cite[Theorem 16.53 of Section 16.9]{aliprantis2006infinite} is obtained by:
	\begin{theorem}[\bf Kakutani Coincidence Theorem]\label{thm Kakutani}
		Let $K$ be a non-empty compact convex subspace of a locally convex Hausdorff space $X$. Let $\varphi,\psi : K \to 2^K$ be set-valued maps satisfying the following
		\begin{enumerate}
			\item for every $x\in K$, the sets $\varphi(x)$ and $\psi(x)$ are non-empty, closed and convex in $X$;
			\item for every continuous linear functional $\ell\in X'$ and every $\alpha\in\mathbb R$, if we set $
			H:=\{y\in X:\ \ell(y)<\alpha\}$, then the sets
			\[
			\{x\in K:\ \varphi(x)\subset H\}\quad\text{ and } \quad\{x\in K:\ \psi(x)\subset H\}
			\]
			are open in $K$ (with the subspace topology);
			\item for each $x$ in $K$, there exist $u\in\varphi(x)$, $v\in\psi(x)$, and a real number $\lambda>0$ such that $x+\lambda(u-v)\in K.$
		\end{enumerate}
		Then there exists $x \in K$ satisfying $
		\varphi(x)\cap\psi(x)\neq\varnothing.$ 
	\end{theorem}
	\begin{remark}
	In particular, we can set $\psi(x)=\{x\}$ such that conclusion would become: there exists $x \in K$ satisfying $x\in \varphi(x)$. In this case  and hence we only have to validate (1)-(2).  
	\end{remark}

Consider the equation
		\begin{equation}\label{eq:nonlin-A-control}
			\left\{
			\begin{aligned}
				\partial_t u
				&=
				A\Big[\Delta u-2(\nabla u)\,\nabla\big(\log h(u)\big)
				+\frac{2|\nabla u|^2}{h(u)}\,u\Big]
				+\chi_\omega f&&\text{in }Q;
				\\
				u(\cdot,0)&=y_0&&\text{in }\mathbb{T}^2.
			\end{aligned}
			\right.
		\end{equation}
		where $h(v):=1+|v|^2$.  For any $M>0$, $\varepsilon>0$ and \(y_0\in W^{2,p}(\mathbb T^2;\mathbb R^2)\), we define the target set
\[
\mathcal T(\varepsilon,M,T,y_0):=
\left\{
\overline y(\cdot,T)\;:\;
\begin{aligned}
&\overline y\in X \text{ solves \eqref{eq:nonlin-A-control} with $f=0$ and initial condition}\\
& \text{$\overline y_0 \in W^{2,p}(\mathbb T^2;\mathbb R^2) $; and } \|\overline y\|_X\le M,
\, \|y_0-\overline y_0\|_{W^{2,p}(\mathbb T^2)}\le \varepsilon
\end{aligned}
\right\}.
\] 

\begin{proposition}[\bf Local exact controllability]\label{prop: null control of nonlinear projected eq} \sloppy Let \(p>4\), \(T>0\), \(M>0\), \(y_0\in W^{2,p}(\mathbb T^2;\mathbb R^2)\) and \(\omega\subset \mathbb T^2\) be a nonempty open set. Then there exists $\varepsilon_*=\varepsilon_*(M,T,\omega,\alpha,\gamma,p)>0$ 
such that for every target \( \overline y_T \in \mathcal T(\varepsilon_*,M,T,y_0)\), there exists a control $
f\in L^\infty(\omega\times(0,T);\mathbb R^2)$ for which \eqref{eq:nonlin-A-control} with initial datum \(y_0\) and control \(\chi_\omega f\), admits a solution $y \in X $ satisfying 
$y(\cdot,T)=\overline y_T$ in $\mathbb T^2$.

Moreover, if \(Y=\overline y(\cdot,T)\) for some uncontrolled trajectory \(\overline y\in X\) which solves \eqref{eq:nonlin-A-control} with $f=0$ and initial condition $\overline y_0$, then the control can be chosen so that
\[
\|f\|_{L^\infty(\omega\times(0,T))}
\le
C\,\|y_0-\overline y(\cdot,0)\|_{L^2(\mathbb T^2)}.
\]
where $C=C(M,T,\omega,\alpha,\gamma,p)>0$
\end{proposition}

\begin{proof}
We split the proof into several steps.

\smallskip
\noindent
\textbf{Step 1. Subtracting the equations:} Suppose that $y$ is the solution of the equation \eqref{eq:nonlin-A-control} with control $f$ and initial condition $y(\cdot,0)=y_0$; and $\overline{y}$ is the solution of the uncontrolled version of equation \eqref{eq:nonlin-A-control} with control $f=0$ and initial condition $\overline{y}(\cdot,0)=\overline{y}_0$. For $w:=y-\overline{y}$ and $w_0:=y_0-\overline{y}_0$, the exact controllability problem for $y$ to reaching $\overline{y}(\cdot,T)$ at time $T$ is equivalent to the null controllability problem
for $w$:
\begin{equation}\label{5512}
\begin{cases}
\partial_t w
=
A\Delta w
+
A\Bigl[
-2(\nabla(\overline{y}+w))\nabla\log h(\overline{y}+w)
+\dfrac{2|\nabla(\overline{y}+w)|^2}{h(\overline{y}+w)}(\overline{y}+w)
\Bigr]\\[2ex]
\hspace{4.6cm}
-
A\Bigl[
-2(\nabla\overline{y})\nabla\log h(\overline{y})
+\dfrac{2|\nabla\overline{y}|^2}{h(\overline{y})}\,\overline{y}
\Bigr]
+\chi_\omega f,
\\[1ex]
w(\cdot,0)=w_0.
\end{cases}
\end{equation}
By the mean-value theorem, we can reduce the equation in \eqref{5512} and consider the controlled equation
\[
\begin{cases}
\partial_t w
=
A\Big(\Delta w
+
\sum_{i=1}^2 \widetilde B_i(w)\,\partial_i w
+
\widetilde\Gamma(w)w\Big)
+\chi_\omega f,
\\[1ex]
w(\cdot,0)=w_0.
\end{cases}
\]
where
\begin{align*}
\widetilde B_i(z)
=\,& 
\int_0^1
\frac{4}{1+|\overline y+\theta z|^2}
\Bigl\{
(\overline y+\theta z)(\partial_i\overline y+\theta\partial_i z)^\top
-
(\partial_i\overline y+\theta\partial_i z)(\overline y+\theta z)^\top
\\
&\qquad\qquad\qquad\qquad
-
\bigl[(\overline y+\theta z)\cdot(\partial_i\overline y+\theta\partial_i z)\bigr]I
\Bigr\}\,d\theta,
\end{align*}
and
\begin{align*}
\widetilde\Gamma(z)
=\,&
\int_0^1 \Biggl\{
\frac{2|\nabla\overline y+\theta\nabla z|^2}{1+|\overline y+\theta z|^2}\,I
-\frac{4}{1+|\overline y+\theta z|^2}\,
(\nabla\overline y+\theta\nabla z)(\nabla\overline y+\theta\nabla z)^\top
\\
&\qquad\qquad
+\frac{8}{(1+|\overline y+\theta z|^2)^2}
\Bigl[
(\nabla\overline y+\theta\nabla z)(\nabla\overline y+\theta\nabla z)^\top
(\overline y+\theta z)
\Bigr] (\overline y+\theta z)^\top
\\
&\qquad\qquad
-\frac{4|\nabla\overline y+\theta\nabla z|^2}{(1+|\overline y+\theta z|^2)^2}
(\overline y+\theta z)(\overline y+\theta z)^\top
\Biggr\}\,d\theta.
\end{align*}
 
\noindent
\textbf{Step 2. Functional spaces:} Fix $R>1$ (to be specified) and define the space to construct the fixed point
\[
\mathcal{F}_R:=
\Bigl\{
z\in X:\ \|z\|_X\le R,\ z(\cdot,0)=w_0
\Bigr\}.
\]
The interpolation \cite[Chapter III, Theorem 4.10.2]{Amann1995} (between $W^{2,p}$ and $L^p$) yields
\begin{align}\label{5001}
			X=L^p\big(0,T;W^{2,p}\big(\mathbb{T}^2)\big) \cap W^{1,p}\big(0,T;L^p(\mathbb{T}^2)\big)
	\hookrightarrow C\big([0,T];W^{2-2/p,p}(\mathbb{T}^2)\big) \quad \text{continuously}.
\end{align}
Thus, the trace map $\gamma_0$ from $X$ to $W^{2-2/p,p}(\mathbb{T}^2)$ defined by $\gamma_0(z) :=z(\cdot,0)$ is continuous. Therefore, the set $\gamma_0^{-1}(\{w_0\})$ is closed in norm and weak topology, and thus $\mathcal{F}_R$ is weakly compact in $X$. It is clear that the set $\mathcal{F}_R$ is also non-empty and convex.

For a given $z\in \mathcal{F}_R$, consider the linear controlled system
\begin{align}\label{eq:frozen-linear}
\begin{cases}
\partial_t u
=
A\Big(\Delta u
+
\sum_{i=1}^2 \widetilde B_i(z)\,\partial_i u
+
\widetilde\Gamma(z)u\Big)
+\chi_\omega f,
&\text{in }\mathbb T^2\times(0,T),
\\[1ex]
u(\cdot,0)=w_0
&\text{in }\mathbb T^2.
\end{cases}
\end{align}
Since $p>4$, Sobolev embedding \cite[Theorem 8.2]{di2012hitchhiker} implies $
W^{2-2/p,p}(\mathbb{T}^2)\hookrightarrow C^{1,\alpha''}(\mathbb{T}^2)$ continuously for $\alpha''=1-\frac{4}{p}>0,$ and the Arzel{\`a}-Ascoli theorem implies $C^{1,\alpha''}(\mathbb{T}^2)\hookrightarrow\hookrightarrow
C^1(\mathbb{T}^2) $ compactly. Thus, it is clear that 	\begin{align}\label{embedding}
	W^{2-2/p,p}(\mathbb{T}^2)\hookrightarrow\hookrightarrow C^1(\mathbb{T}^2)\hookrightarrow L^p(\mathbb{T}^2).
\end{align} 
Therefore, by the Aubin-Lions-Simon lemma and the embeddings in \eqref{embedding} and \eqref{5001}, we have
\begin{align}
X \hookrightarrow L^\infty(0,T; W^{2-2/p,p}(\mathbb T^2)) \cap W^{1,p}(0,T;L^p(\mathbb T^2))
\hookrightarrow \hookrightarrow C\big([0,T];C^1(\mathbb T^2)\big) \quad \text{compactly.}
\label{5627}
\end{align}
Hence, for $z\in \mathcal{F}_R$,
\[
\|z\|_{C([0,T];C^1(\mathbb T^2))}\lesssim_{\,T,p} R
\]
Since also $
\|\overline{y}\|_{C([0,T];C^1(\mathbb T^2))}\le M$,
the pair $\|(\overline{y}+\theta z,\nabla\overline{y}+\theta\nabla z)\|_{L^\infty} \lesssim_{\,T,p} R+M$. Because $\Gamma$ and $B_i$ are smooth, there exists a constant
$C_{M,R,T,p}>0$ such that
\[
\|\widetilde\Gamma(z)\|_{L^\infty(\mathbb T^2\times(0,T))}
+
\sum_{i=1}^2
\|\widetilde B_i(z)\|_{L^\infty(\mathbb T^2\times(0,T))}
\le C_{M,R,T,p}
\qquad\text{for any } z\in \mathcal{F}_R.
\]
 
By the linear null-controllability result in \Cref{thm null control of linear eq.}, there exists a control
$f_z\in L^\infty\big(\omega\times(0,T);\mathbb{R}^2\big)$ such that the corresponding solution $w_{f_z,z}$ of
\eqref{eq:frozen-linear} satisfies $w_{f_z,z}(\cdot,T)=0$ a.e. in $\mathbb{T}^2$. Moreover, we define the set of admissible controls 

\begin{equation}\label{eq:control-class}
\mathcal C(z)\!:=
\left\{
f\in L^\infty(\omega\times(0,T);\mathbb{R}^2) 
:\!
\begin{array}{l}
\text{the solution $w_{f,z}$ of \eqref{eq:frozen-linear} satisfies $w_{f,z}(\cdot,T)=0$}
,\\[2mm]
\text{a.e. in $\mathbb{T}^2$ and } \|f\|_{L^\infty(\omega\times(0,T))}\le C_R\,\|w_0\|_{L^2(\mathbb{T}^2)}
\end{array}
\right\}
\end{equation}
for some $C_R>0$ determined later.

For any $z\in \mathcal{F}_R$ and any $f_z\in\mathcal C(z)$, the maximal $L^p$-regularity in \Cref{lem higher est} yields  
\begin{equation}\label{eq:maxreg}
	\|w_{f_z,z}\|_{L^{p}(0,T;W^{2,p})}+\|\partial_t w_{f_z,z}\|_{L^{p}(0,T;L^{p})}
	\le \C{7}\Big(\|w_0\|_{W^{2,p}}+\|f_z\|_{L^{p}(\omega\times (0,T))}\Big).
\end{equation} 
Since $f_z\in \mathcal C(z)$, we can use 
\eqref{eq:control-class} to establish $\|f_z\|_{L^{p}(\omega\times (0,T))}\lesssim \|f_z\|_{\infty} \leq C_R\|w_0\|_{L^2}$. Therefore, we choose $ \varepsilon _\ast>0$ in the assumptions so small that 
\begin{align}\label{5733}
2 \C{7}(1+C_R)\,\varepsilon_\ast\le R.
\end{align}
Then for all $ w_0\in W^{2,p}(\mathbb{T}^2;\mathbb{R}^2)$ with $\|w_0\|_{W^{2,p}}\le\varepsilon_\ast$, every corresponding $w_{f_z,z}$ produced from \eqref{eq:frozen-linear} with
$f_z\in\mathcal C(z)$ satisfies $\|w_{f_z,z}\|_{X}\le R$. Therefore, $w_{f_z,z} \in \mathcal{F}_R$ for any $z\in \mathcal{F}_R$. Hence, we can define the multivalued map $\Phi:\mathcal{F}_R\to 2^{\mathcal{F}_R}$ by
\begin{align}\label{def fixed point map Phi}
	\Phi(z):=\Big\{w\in \mathcal{F}_R:\ w=w_{f_z,z}\text{ solves \eqref{eq:frozen-linear} for some }f=f_z\in\mathcal C(z)\Big\}.
\end{align}
 
\noindent\textbf{Step 3. Properties of Kakutani map  $\Phi$ (item (1) of \Cref{thm Kakutani}):} By construction, $\Phi(z)$ is non-empty for every $z\in \mathcal{F}_R$. The convexity of $\Phi(z)$ is trivial as the \eqref{eq:frozen-linear} is linear in the unknown $u$ and also linear in control term, for fixed $z\in \mathcal{F}_R$. To verify compactness, let $\{w_n\}_{n\in \mathbb{N}} \subset \Phi(z)$ and assume that
\[
w_n\rightharpoonup w_* \qquad\text{weakly in }X.
\] By \eqref{eq:maxreg},
$\{w_n\}_{n\in \mathbb{N}}$ is bounded in $X$. Therefore, by the compact embedding in \eqref{5627}, we have $w_*\in C\big([0,T];C^1(\mathbb T^2)\big)$ such that up to a subsequence, 
\[
w_n\to w_*\qquad\text{strongly in }C([0,T];C^1(\mathbb T^2)).
\]
For each $n \in \mathbb{N}$, we choose $f_n\in\mathcal C(z)$ such that $w_n$ solves the equation  
\[
\int_{\mathbb{T}^n\times(0,T)} \bigl[- w_n\cdot \partial_t\varphi
+ \nabla (Aw_n) \cdot \nabla\varphi
-  \sum^2_{i=1}(A\widetilde{B}_i(z) \partial_i w_n)\cdot \varphi
- (A \widetilde{\Gamma}(z)w_n)\cdot \varphi \bigr]
= \int_{\omega\times(0,T)} f_n\cdot \varphi ,
\]
for any $\varphi\in C_c^\infty(\mathbb{T}^n\times(0,T);\mathbb R^2)$. By the uniform bound in $\mathcal C(z)$, after extraction of subsequence,
\[
f_n\rightharpoonup f_* \qquad\text{weak-star in }L^\infty(\omega\times(0,T)).
\] 
Passing to the limit in the equation, we infer that $w_*$ solves the equation \eqref{eq:frozen-linear} with control $f=f_*$.
The strong convergence in $C([0,T];C^1(\mathbb T^2))$ yields $w_*(\cdot,T)=0$. Therefore, $f_* \in \mathcal{C}(z)$ and $w_*\in\Phi(z)$. This shows that $\Phi(z)$ is weakly closed in $\mathcal{F}_R$. Since $\mathcal{F}_R$ is weakly compact in $X$, the set $\Phi(z)$ is weakly compact as a subset of $X$. \\
  
\noindent\textbf{Step 4. Upper semicontinuity (item (2) of \Cref{thm Kakutani}):} We want to prove that for every $\ell\in X'$ and $\alpha\in\mathbb R$, with \[
H:=\{y\in X:\ell(y)<\alpha\},
\qquad
U_\Phi(H):=\{z\in \mathcal{F}_R:\Phi(z)\subset H\},
\]
the set $U_\Phi(H)$ is open in $\mathcal{F}_R$. It is equivalent to show that any $z \in U_\Phi(H)$ is an interior point of $U_\Phi(H)$.

Assume, by contradiction, there exists $z_\dagger \in U_\Phi(H)$ and a sequence $\{z_n\}_{n\in\mathbb N}\subset \mathcal{F}_R\setminus U_\Phi(H)$ such that
$z_n \rightharpoonup z_\dagger$ in $\mathcal{F}_R$.
For each $n \in \mathbb{N}$, since $z_n\notin U_\Phi(H)$, we have $\Phi(z_n)\not\subset H$; hence we can choose $w_n\in \Phi(z_n)\setminus H$, so that 
\begin{align}\label{5337}
\ell(w_n)\ge \alpha.
\end{align} By the definition of $\Phi$ (see \eqref{def fixed point map Phi}), there exists
a control $f_n\in \mathcal{C}(z_n)$ such that $w_n$ solves the linear system
\[
\partial_t w_n
=
A\Big(\Delta w_n
+
\sum_{i=1}^2 \widetilde B_i(w_n)\,\partial_i w_n
+
\widetilde\Gamma(w_n)w_n\Big)
+\chi_\omega f_n,
\quad\text{in }Q,
\quad
w_n(\cdot,0)=w_0,
\quad
w_n(\cdot,T)=0.
\]
Moreover, since $f_n\in \mathcal{C}(z_n)$, the controls are uniformly bounded in $L^\infty$ by \eqref{eq:control-class} and hence $\|w_n\|_X$ is uniformly bounded by \eqref{eq:maxreg}. Hence there is $w_\dagger \in X$ such that $w_n \rightharpoonup w_\dagger$ in $X$. By the similar compactness arguments in Step 2, we may (after possibly extracting further subsequence) assume $
z_n\to z_\dagger$, $w_n\to w_\dagger$ in $C\big([0,T];C^1(\mathbb T^2)\big)$ strongly and $f_n \rightharpoonup f_\dagger$ weak-star in $L^\infty\big(\omega\times(0,T);\mathbb R^2\big)$. In particular, passing to the limit at $t=T$ gives $w_\dagger(\cdot,T)=\lim_n w_n(\cdot,T)=0$. We also have
\[
\|z_n-z_\dagger\|_{L^\infty(Q)}+\|\nabla z_n-\nabla z_\dagger\|_{L^\infty(Q)}\to 0.
\] 
Therefore,
\[
B(z_n)\to B(z)\quad\text{and}\quad c(z_n)\to c(z)\quad\text{in }L^\infty(Q).
\]
Passing limit in the equation, we infer that $w_\dagger$ solves the equation \eqref{eq:frozen-linear} with control $f=f_\dagger$ and $z=z_\dagger$. Hence $w_\dagger\in \Phi(z_\dagger)$ by definition of $\Phi$.

Since $z_\dagger\in U_\Phi(H)$ by the assumption, we have $\Phi(z_\dagger)\subset H$, hence $w_\dagger\in H$ and thus $\ell(w_\dagger)<\alpha$.
On the other hand, $\ell$ is continuous on $X$ under the weak topology and $w_n\rightharpoonup w_\dagger$ weakly in $X$, so \eqref{5337} implies
\[
\ell(w_\dagger)=\lim_{n\to\infty}\ell(w_n)\ge \alpha,
\]
contradicting $\ell(w_\dagger)<\alpha$, so $U_\Phi(H)$ is open in $\mathcal{F}_R$.

Therefore $\Phi$ satisfies the hypotheses of \Cref{thm Kakutani}. Hence, by setting one map to be $\Phi$ and another map to be identity in \Cref{thm Kakutani}, there exists
$w\in \mathcal{F}_R$ such that $w=\Phi(w)$. For this fixed point, $w$ solves \eqref{eq:frozen-linear}
with $z=w$, i.e.\ it solves \eqref{eq:nonlin-A-control}, and satisfies $w(\cdot,T)=0$. The $L^\infty$ estimate on $f$ follows from \eqref{eq:control-class}. Now, we can choose $R>0$ arbitrarily and $C_R$ is set to be the constant $\C{4}$ obtained in \Cref{thm null control of linear eq} with the fixed choices of $B$ and $\Gamma$ in this proof.
 
\end{proof}

		\begin{lemma}\label{lem recovery of control from proj eq} Let $v$ and $f\in\mathbb R^2$. Then there exists $u=u(v,f)\in\mathbb R^3$ taking the form 		\begin{equation*}
			u(v,f)=\nabla_v\Psi_0(v)\,(\alpha I+\gamma J)^{-1}f.
		\end{equation*} such that the following identity holds:
		\begin{equation*}
			(\alpha I+\gamma J)\left\{\frac{h(v)^2}{4}\,(\nabla_v\Psi_0(v))^{ \top}u\right\}=f.
		\end{equation*}  

	\end{lemma}
	
	\begin{proof} It is a direct consequence of \eqref{eq. D Psi top D Psi = 4I/h^2 and inverse of D Psi}.
	\end{proof}  

    \subsection{Controllability with small initial energy}\label{sec. Controllability with small initial energy}
	
	\begin{proof}[\bf Proof of \Cref{thm:nullcontrol}]  
    
By Lemma \ref{lem:small-energy-const}, there is $p_0 \in \mathbb{S}^2$ such that $$\|m_0-p_0\|_{H^1(\mathbb T^2)}\lesssim \sqrt{E(m_0)}\leq \sqrt{\varepsilon}.$$ 
Let $\varepsilon_1>0$ and fix any target $m_T \in \mathcal{T}(\varepsilon_1,T,m_0)$. We can find $\overline{m}^{\varepsilon_1}_0\in  H^1(\mathbb T^2;\mathbb R^2)$ such that $\|m_0-\overline{m}^{\varepsilon_1}_0\|_{H^1}\leq \varepsilon_1$ and $m_T=\overline{m}^{\varepsilon_1}(\cdot,T)$ where $\overline{m}^{\varepsilon_1}$ solves the uncontrolled equation \eqref{eq. uncontrolled LLG}, subject to initial condition $\overline{m}^{\varepsilon_1}_0$. Then, we have
$$\|\overline{m}^{\varepsilon_1}_0-p_0\|_{H^1(\mathbb T^2)} \leq
\|\overline{m}^{\varepsilon_1}_0-m_0\|_{H^1(\mathbb T^2)}
+\|m_0-p_0\|_{H^1(\mathbb T^2)}
\leq \sqrt{\varepsilon} + \varepsilon_1.$$
Let $T_0=T/2>0$, $\varepsilon_*>0$ and $p>4$. If we choose $\varepsilon$ and $\varepsilon_1$ small enough (depending on $\varepsilon_*$, $\alpha$, $\gamma$, $T$ and $p$), then Lemma~\ref{lem:LLG-hemisphere-entry} implies that there is $t_0\in(0,T_0]$ such that the solution $m^{0,*}$ of the uncontrolled equation \eqref{eq. uncontrolled LLG}, subject to initial condition $m_0$, satisfies  
\begin{align}
&\|m^{0,*}(\cdot,t_0)-p_0\|_{L^\infty(\mathbb T^2)}+
\|m^{0,*}(\cdot,t_0)-\overline{m}^{\varepsilon_1}(\cdot,t_0)\|_{L^\infty(\mathbb T^2)}
+
\|m^{0,*}(\cdot,t_0)-\overline{m}^{\varepsilon_1}(\cdot,t_0)\|_{W^{2,p}(\mathbb T^2)}\nonumber\\
&\le \varepsilon_*.
\label{5750}
\end{align}
If we simply set either one of the initial conditions in  Lemma~\ref{lem:LLG-hemisphere-entry} to be $p_0$, then we obtain
\begin{align}\label{5753}
\|m^{0,*}(\cdot,t_0)-p_0\|_{W^{2,p}(\mathbb T^2)}
+
\| \overline{m}^{\varepsilon_1}(\cdot,t_0) - p_0\|_{W^{2,p}(\mathbb T^2)}
\le \varepsilon_*.
\end{align}

Note that the uncontrolled equation \eqref{eq. uncontrolled LLG} is rotationally invariant in the sense that if $m$ is a strong solution of \eqref{eq. uncontrolled LLG}, subject to the initial condition $m_0$, then $\mathcal{R}m$ is also a strong solution of \eqref{eq. uncontrolled LLG}, subject to the initial condition $\mathcal{R}m_0$ for any $\mathcal{R}\in SO(3)$. We take a $\mathcal{R}_0\in SO(3)$ such that $\mathcal{R}_0 p_0=e_3$. Hence, by \eqref{5753}, \eqref{5750} and Sobolev embedding, we have
\begin{align*}
&\left\|\Psi_0^{-1}\big(\mathcal{R}_0m^{0,*}(\cdot,t_0)\big)\right\|_{L^\infty(\mathbb T^2)}
+
\left\|\Psi_0^{-1}\big(\mathcal{R}_0m^{0,*}(\cdot,t_0)\big)-\Psi_0^{-1}\big(\mathcal{R}_0\overline{m}^{\varepsilon_1}(\cdot,t_0)\big)\right\|_{L^\infty(\mathbb T^2)}\\
&+
\left\|\Psi_0^{-1}\big(\mathcal{R}_0m^{0,*}(\cdot,t_0)\big)-\Psi_0^{-1}\big(\mathcal{R}_0\overline{m}^{\varepsilon_1}(\cdot,t_0)\big)\right\|_{W^{2,p}(\mathbb T^2)}\\
&\lesssim \varepsilon_*.
\end{align*}
Let $T^*_0=T-t_0>0$. Moreover, if we shrink $\varepsilon>0$ and $\varepsilon_1>0$ further (depending on $\alpha$, $\gamma$, $T_0$, and $p$), then $\varepsilon_*$ can be chosen so small that we can apply \Cref{prop: null control of nonlinear projected eq} to deduce that there is a control $f_0 \in L^\infty \big(\omega \times (0,T_0^*);\mathbb{R}^2\big)$ such that the equation \eqref{eq:nonlin-A-control}, subject to initial condition $\Psi_0^{-1}\big(\mathcal{R}_0m^{0,*}(\cdot,t_0)\big)$ and control $f_0$, has a solution $y_0$ such that $y_0(\cdot,T_0^*)=Y$ in $\mathbb{T}^2$. Here the target is $Y=\overline{y}(\cdot,T^*_0)$ where $\overline{y}$ solves \eqref{eq:nonlin-A-control} with $f=0$ and initial condition $\overline y_0 = \Psi_0^{-1}\big(\overline{m}^{\varepsilon_1}(\cdot,t_0)\big) \in W^{2,p}(\mathbb T^2;\mathbb R^2) $.

    By \Cref{lem recovery of control from proj eq}, there is $u_0=u_0(y_0,f_0)$ such that 
	$$ 			A\left\{\frac{h(y_0)^2}{4}\,(\nabla_v\Psi_0(y_0))^{ \top}u_0\right\}=f_0. $$
    Therefore, by \eqref{eq. projected eq}, we can transform equation \eqref{eq:nonlin-A-control} back to the controlled LLG equation \eqref{eq. controlled LLG}. Hence, the solution $m^{0,u_0}:=\Psi_0(y_0)$ of the controlled LLG equation \eqref{eq. controlled LLG}, subject to initial condition $\mathcal{R}_0m^{0,*}(\cdot,t_0)$ and control $u_0$, satisfies $m^{0,u_0}(\cdot,T_0^*)=\Psi_0\big(Y\big)$. We can paste the uncontrolled solution $m^{0,*}$ and controlled solution $\mathcal{R}_0^{-1}m^{0,u_0}$ to form a new solution $$m^0(x,t):=\chi_{[0,t_0]}m^{0,*}(x,t)+\chi_{(t_0,t_0+T^*_0]}\mathcal{R}_0^{-1} m^{0,u_0}(x,t-t_0)$$ 
    of the controlled equation \eqref{eq. controlled LLG} from $t=0$ to $t_0+T^*_0=T$, subject to initial condition $m_0$ and control $\chi_{\omega\times (t_0,t_0+T^*_0]}\mathcal{R}_0^{-1}u_0(x,t-t_0)$. Note that $m^0(\cdot,t_0+T^*_0)=\Psi_0\big(Y\big)=m_T$ in $\mathbb{T}^2$. 
		
	\end{proof}

The main obstacle prevent us to steer from a particular constant state to any constant state is that $\varepsilon_*$ in \Cref{prop: null control of nonlinear projected eq} depends on $T$, more precisely,
Here the constant 
        \begin{align}\label{5871}
		\varepsilon_* \lesssim \dfrac{e^{C_{\alpha}T}}{P_{\alpha}(T)\pig[1+e^{C_{\alpha}\frac{1}{T}}\pig]}
		\end{align}
        from \eqref{5129} and \eqref{5733}. Here $P_{\alpha}(T)$ is a polynomial in $T$ (may have negative degree and the degrees are universal) with coefficient depending on $\alpha$. Suppose now we have $m(\cdot,0)=p_0$ and we want to drive it to $m(\cdot,T)=p$ for any $T>0$ and $p>0$. We have to cut the geodesic connecting $p_0$ and $p$ into $N \in \mathbb{N}$ uniform pieces such that $\dfrac {|p_0-p|}{ N} \lesssim \varepsilon_{*,N,T}\lesssim  \dfrac{e^{C_{\alpha}\frac{T}{N}}}{P_{\alpha}(\frac{T}{N})\pig(1+e^{C_{\alpha}\frac{N}{T}}\pig)}$. However, the existence of such $N$ cannot be ensured for arbitrary $T>0$ and $p\in \mathbb{S}^2$. Moreover, the work \cite{DUYCKAERTS20081} shows that the constant $\C{4}$ in \eqref{5129} is sharp and hence $\varepsilon_*$. Therefore, the best exact controllability to any constant state is the one in \Cref{cor:null control} in which the terminal time can not be too small, if we proceed in Carlman estimate and observability method.
 
 	\begin{proof}[\bf Proof of \Cref{cor:null control}]
    
    Fix the time horizon in \Cref{thm:nullcontrol} to be \(1\) for simplicity, and let \(\varepsilon>0\) be the corresponding smallness constant. Choose \(\varepsilon_{1}>0\) so small that whenever
\(E(m_{0})<\varepsilon_{1}\), \Cref{lem:small-energy-const} yields a constant
\(p_{0}\in \mathbb S^{2}\) such that
\[
\|m_{0}-p_{0}\|_{H^{1}(\mathbb T^{2})}\le \frac{\varepsilon}{2}.
\]
Let \(p^{*}\in \mathbb S^{2}\) be arbitrary. We first steer \(m_{0}\) to \(p_{0}\), and then
steer \(p_{0}\) to \(p^{*}\) through finitely many nearby constant states.

Since the uncontrolled solution of \eqref{eq. uncontrolled LLG} with constant initial datum
\(p_{0}\) is just the constant map \(p_{0}\), the state \(p_{0}\) belongs to the target set $\mathcal{T}(\varepsilon,1,m_0)$ in \Cref{thm:nullcontrol}. Therefore, \Cref{thm:nullcontrol} gives a control $u_{0}\in L^{\infty}(0,1;L^{2}(\mathbb T^{2}))$ such that the corresponding unique solution \(m^{0,u_{0}}\) of
\eqref{eq. controlled LLG}, subject to control $u_0$ and initial condition $m_0$, satisfies
\[
m^{0,u_{0}}(\cdot,1)=p_{0}.
\]

Next, choose \(N\in \mathbb N\) and constants $p^{1},p^{2},\dots,p^{N-1}, p^{*}\in \mathbb S^{2}$ on a minimizing geodesic joining \(p_{0}=p^0\) and \(p^{*}\), with
\(\operatorname{dist}_{\mathbb S^{2}}(p^{i-1},p^{i})\) independent of \(i\). Then
\[
\operatorname{dist}_{\mathbb S^{2}}(p^{i-1},p^{i})
=\frac{1}{N}\operatorname{dist}_{\mathbb S^{2}}(p_{0},p^{*})
\le \frac{\pi}{N},
\qquad i=1,\dots,N.
\] 
As \(p^{i}-p^{i-1}\) is constant on \(\mathbb T^{2}\),
\[
\|p^{i}-p^{i-1}\|_{H^{1}(\mathbb T^{2})} 
\le \frac{2\pi^{2}}{N}.
\]
Choose \(N\) so large that $
\frac{2\pi^{2}}{N}\le \frac{\varepsilon}{2}.$ Then, for every \(i=1,\dots,N\),
\[
\|p^{i}-p^{i-1}\|_{H^{1}(\mathbb T^{2})}\le \frac{\varepsilon}{2}<\varepsilon.
\]

Now fix \(i\in\{1,\dots,N\}\). The uncontrolled solution of \eqref{eq. uncontrolled LLG} with initial datum \(p^{i}\) is again the constant map \(p^{i}\). Therefore, \Cref{thm:nullcontrol} yields a control $
u_{i}\in L^{\infty}(0,1;L^{2}(\mathbb T^{2}))$ such that the corresponding unique solution \(m^{i,u_{i}}\) of \eqref{eq. controlled LLG} with initial datum \(p^{i-1}\) and control $u^i$ satisfies $
m^{i,u_{i}}(\cdot,1)=p^{i}.$ We now concatenate these controls. Define \(u\) on \((0,N+1)\) by
\[
u(\cdot,t):=
\begin{cases}
u_{0}(\cdot,t), & t\in (0,1),\\[1mm]
u_{i}(\cdot,t-i), & t\in (i,i+1),\quad i=1,\dots,N.
\end{cases}
\]
Then \(u\in L^{\infty}(0,N+1;L^{2}(\mathbb T^{2}))\). Set $T=N+1$ and let \(m^{u}\) be the corresponding solution of \eqref{eq. controlled LLG} with initial datum
\(m_{0}\) and control $u$, we have \(m^{u}(\cdot,T)=p^{*}\). 
\end{proof}

\subsection{Controllability with hemisphere condition}\label{sec. Controllability with hemisphere condition}

We first consider the equation \eqref{eq:LLG-two-stage} when $B(x,t)=H p$ for some constant $H\in \mathbb{R}$ and $p\in \mathbb{S}^2$. We choose $\mathcal{R}\in SO(3)$ with $\mathcal{R}p=e_3$ and let $m$ be solution of \eqref{eq:LLG-two-stage}. Then $\mathcal{R}m$ solves \eqref{eq:LLG-two-stage} with $p$ replaced by $e_3$ and initial
datum $\mathcal{R}m_0$. We therefore assume $p=e_3$ and work in stereographic coordinates in \eqref{def. stereographic projection and inverse} with $ m=\Psi_0(z)$ with $  z \in \mathbb{C}$. 
The initial datum
$z_0:=(m_{0,1}+im_{0,2})/(1+m_{0,3})$
belongs to $H^2(\mathbb{T}^2;\mathbb{C})$ as $
\inf_{x\in \mathbb T^2} m_0(x)\cdot e_3>0$.

Using the projected equation in \eqref{eq. projected eq}, the equation \eqref{eq:LLG-two-stage} becomes
\begin{equation}
\label{eq:z-equation}
  \partial_t z
  =(\alpha+i\gamma)\bigl[\Delta z-Hz+\mathcal{N}(z)\bigr],
  \qquad
  z(\cdot,0)=z_0(\cdot)
\end{equation}
where\[
  \mathcal{N}(z)
  :=-\frac{2\overline{z}}{1+|z|^2}\sum_{j=1}^2(\partial_j z)^2.
\]
We recall that $\beta := \alpha - i\gamma$ and prove the lemma
\begin{lemma}[\bf Global existence under a large constant field]
\label{lem:z-existence-large-H}
For any $z_0\in H^2(\mathbb{T}^2;\mathbb{C})$, there exists $H_* = H_*(z_0,\alpha,\gamma) > 0$ such that, for every $L>0$ and $H \ge H_*$,
equation~\eqref{eq:z-equation} admits a strong solution $z \in C\!\left([0,L];H^2(\mathbb{T}^2)\right)\cap L^2\!\left(0,L;H^3(\mathbb{T}^2)\right)$. 
\end{lemma}
\begin{proof}
We proceed in four steps.


\noindent\textbf{Step 1. Galerkin approximation:}
For $N \ge 1$, let $P_N$ denote the $L^2$-orthogonal projection of Fourier series
\[
  P_N f := \sum_{|k| \le N} \widehat{f}(k)\,e^{ik\cdot x}, \qquad
  \widehat{f}(k) := \frac{1}{(2\pi)^2}\int_{\mathbb{T}^2} f(x)\,e^{-ik\cdot x}\,dx,
\]
onto $E_N := \operatorname{span}\{e^{ik\cdot x} : |k| \le N\}$.
We seek $z_N \in E_N$ satisfying the Galerkin equation
\begin{equation}
\label{eq:Galerkin-z}
  \partial_t z_N = \overline{\beta}\,P_N\!\bigl[\Delta z_N - Hz_N + \mathcal{N}(z_N)\bigr],
  \qquad z_N(\cdot,0) = P_N z_0(\cdot).
\end{equation}
Writing $z_N(x,t) = \sum_{|k|\le N} a_k^N(t)\,e^{ik\cdot x}$, equation~\eqref{eq:Galerkin-z}
is equivalent to
\begin{equation}
\label{eq:ode-coefficients}
  \dfrac{d}{dt} a_k^N = \overline{\beta}\!\left[-(|k|^2 + H)a_k^N + \widehat{\mathcal{N}(z_N)}(k)\right],
  \qquad |k| \le N.
\end{equation}
This is a system of ODEs whose right-hand side is locally Lipschitz in
$a^N = (a_k^N)_{|k|\le N} \in \mathbb{C}^{d_N}$ (where $d_N := |\{k\in\mathbb{Z}^2:|k|\le N\}|$).
By the Picard--Lindelöf theorem, there exists a unique maximal solution
$a^N \in C^1([0,T_N^*);\mathbb{C}^{d_N})$ for some $T_N^* \in (0,\infty]$,
giving $z_N \in C^1([0,T_N^*);E_N) \subset C^1([0,T_N^*);C^\infty(\mathbb{T}^2))$.

 
\noindent\textbf{Step 2. Uniform $H^2$ energy estimate:}
Set
\[
  Z_N(t) := \|(1-\Delta)z_N(\cdot,t)\|_{L^2(\mathbb{T}^2)}^2,
\] 
which satisfies $\|z_N(\cdot,t)\|_{H^2(\mathbb{T}^2)}^2 \lesssim Z_N(t) \lesssim \|z_N(\cdot,t)\|_{H^2(\mathbb{T}^2)}^2$. 
Applying $1-\Delta$ to~\eqref{eq:Galerkin-z} and 
taking the real part of the $L^2$-inner product with $(1-\Delta)z_N$, we use the fact that
$P_N$, $\nabla$, and $\Delta$ mutually commute on $E_N$, together with
integration by parts on $\mathbb{T}^2$, to obtain
\begin{equation}
\label{eq:H3-energy}
  \frac{1}{2}\frac{d}{dt}Z_N + \alpha\|(1-\Delta)z_N\|_{L^2(\mathbb{T}^2)}^2
  + \alpha H Z_N
  = \operatorname{Re}\!\left[\,
    \overline{\beta}\bigl\langle(1-\Delta)\mathcal{N}(z_N),\,(1-\Delta)z_N\bigr\rangle_{L^2(\mathbb{T}^2)}
  \right].
\end{equation}

Write $\mathcal{N}(w) = -\mathcal{N}_1(w) \mathcal{N}_2(w)$ with
$\mathcal{N}_1(w) := \frac{2\overline{w}}{1+|w|^2}$ and $\mathcal{N}_2(w) := \sum_{j=1}^{2}(\partial_j w)^2$.
Since $|\mathcal{N}_1(w)| \le 1$ pointwise, the tame $H^2$ product estimate gives
\[
  \|\mathcal{N}(w)\|_{H^2(\mathbb{T}^2)}
  \lesssim \|\mathcal{N}_1(w)\|_{H^2(\mathbb{T}^2)}\|\mathcal{N}_2(w)\|_{L^\infty(\mathbb{T}^2)} + \|\mathcal{N}_2(w)\|_{H^2(\mathbb{T}^2)}.
\]
We estimate each factor using the Sobolev embedding
$H^2(\mathbb{T}^2) \hookrightarrow L^{\infty}(\mathbb{T}^2)$ and the Gagliardo-Nirenberg interpolation inequalities $\|f\|_{L^\infty(\mathbb{T}^2)} \lesssim \|f\|_{L^2}^{\frac{1}{2}} \|f\|_{H^2}^{\frac{1}{2}} $ and $\|f\|_{L^4(\mathbb{T}^2)} \lesssim \|f\|_{L^2}^{\frac{1}{2}} \|f\|_{H^1}^{\frac{1}{2}} $:
\begin{align*}
  \|\mathcal{N}_1(w)\|_{H^2(\mathbb{T}^2)} &\lesssim (1+\|w\|_{H^1(\mathbb{T}^2)})\|w\|_{H^2(\mathbb{T}^2)}, 
  \\
  \|\mathcal{N}_2(w)\|_{L^\infty(\mathbb{T}^2)} &\lesssim \|\nabla w\|_{L^\infty(\mathbb{T}^2)}^2 \lesssim \|w\|_{H^1(\mathbb{T}^2)} \|w\|_{H^3(\mathbb{T}^2)}, \\
  \|\mathcal{N}_2(w)\|_{H^2(\mathbb{T}^2)} &\lesssim \|\nabla w\|_{L^\infty(\mathbb{T}^2)}\|\nabla w\|_{H^2(\mathbb{T}^2)}
                  \lesssim \|w\|_{H^1(\mathbb{T}^2)}\|w\|_{H^3(\mathbb{T}^2)}^\frac{3}{2}.
\end{align*}
Combining and using $\|w\|_{H^1(\mathbb{T}^2)} \le \|w\|_{H^2(\mathbb{T}^2)}$:
\begin{equation}
\label{eq:nonlinear-estimate}
  \|\mathcal{N}(w)\|_{H^2(\mathbb{T}^2)}
  \lesssim 
  (1+\|w\|_{H^2(\mathbb{T}^2)}^3) \|w\|_{H^3(\mathbb{T}^2)} + \|w\|_{H^1(\mathbb{T}^2)}\|w\|_{H^3(\mathbb{T}^2)}^\frac{3}{2}.
\end{equation}
Applying \eqref{eq:nonlinear-estimate} with $w = z_N$, and writing
$\|z_N\|_{H^2(\mathbb{T}^2)} \lesssim Z_N^{1/2}$, yields
\begin{align*}
  &\left|\operatorname{Re}\!\left[\,
    \overline{\beta}\bigl\langle(1-\Delta)\mathcal{N}(z_N),(1-\Delta)z_N\bigr\rangle_{L^2(\mathbb{T}^2)}
  \right]\right| \\
  &\leq C_{\alpha,\gamma}
  \|\mathcal{N}(z_N)\|_{H^2(\mathbb{T}^2)}\|z_N\|_{H^2(\mathbb{T}^2)}\\
  &\leq C_{\alpha,\gamma} \left[ (1+Z_N^\frac{3}{2})\,Z_N^\frac{1}{2} \|z_N\|_{H^3(\mathbb{T}^2)} 
  + Z_N \|z_N\|_{H^3(\mathbb{T}^2)}^\frac{3}{2} \right]
\end{align*}
for some constant $C_{\alpha,\gamma}>0$ depending only on $\alpha$ and $\gamma$. Let $c_{\alpha}>0$ be small, we use Young's inequality to obtain
\begin{equation}
\label{eq:Young-nonlinear}
  C_{\alpha,\gamma} \left[ (1+Z_N^\frac{3}{2})\,Z_N^\frac{1}{2} \|z_N\|_{H^3(\mathbb{T}^2)} 
  + \|z_N\|_{H^3(\mathbb{T}^2)}^\frac{3}{2} \right]
  \le 
  c_{\alpha} \|z_N\|_{H^3(\mathbb{T}^2)}^2 + \dfrac{2C_{\alpha,\gamma}^2}{c_{\alpha}}(1+Z_N^3) Z_N
  + \dfrac{C_{\alpha,\gamma}^4}{c_{\alpha}^3}Z_N^4, \\
\end{equation}
Moreover, differentiating the Fourier series of $z_N$ directly gives
\begin{equation}
\label{eq:H3-vs-grad}
  \|z_N\|_{H^3(\mathbb{T}^2)}^2
  \lesssim \|\nabla(1-\Delta)z_N\|_{L^2(\mathbb{T}^2)}^2 + Z_N.
\end{equation}
Substituting \eqref{eq:Young-nonlinear}--\eqref{eq:H3-vs-grad}
into~\eqref{eq:H3-energy} and absorbing the gradient term into the left-hand side,
we arrive at the differential inequality
\begin{equation}
\label{eq:ZN-diff-ineq}
  \frac{d}{dt}Z_N(t) + c_\alpha\|\nabla(1-\Delta)z_N(\cdot,t)\|_{L^2(\mathbb{T}^2)}^2
  + 2\alpha H\,Z_N(t)
  \le C_{\alpha,\gamma} \bigl(1+Z_N^3(t)\bigr)Z_N(t).
\end{equation}
 
Set $
  R := 2C\|z_0\|_{H^2(\mathbb{T}^2)}^2 + 2,$ so that $Z_N(0)\leq C\|z_0\|_{H^2(\mathbb{T}^2)}^2 < R$. Choose $H_* \ge 1$ large enough that
\begin{equation}
\label{eq:H-choice}
  2\alpha H_* \ge C_{\alpha,\gamma}(1+R^3) + 1.
\end{equation}
We consider $H \ge H_*$ and define
\[
  \tau_* := \sup\bigl\{T \in [0,T_N^*) : Z_N(t) < R \text{ for all } t \in [0,T]\bigr\}.
\]
Since $Z_N(0) < R$ and $t \mapsto Z_N(t)$ is continuous, we have $\tau_* > 0$.
Suppose that $\tau_* < T_N^*$. For all $t \in [0,\tau_*)$
we have $Z_N(t) < R$, hence $Z_N'(t) \le 0$ from \eqref{eq:ZN-diff-ineq} and therefore $Z_N(t) \le Z_N(0)$.
Continuity at $t = \tau_*$ gives $Z_N(\tau_*) \le Z_N(0) < R$, and then
continuity for $t$ slightly beyond $\tau_*$ contradicts the definition of $\tau_*$.
Therefore $\tau_* = T_N^*$, and
\begin{equation}
\label{eq:ZN-uniform-bound}
  Z_N(t) \le Z_N(0) \lesssim \|z_0\|_{H^2(\mathbb{T}^2)}^2
  \qquad\text{for all } t \in [0, T_N^*].
\end{equation}

 
\noindent\textbf{Step 3. Global existence:}
We show that the Galerkin solution extends to $[0,L]$. Suppose for contradiction that $T_N^* \le L$. By \eqref{eq:ZN-uniform-bound},
\begin{align}\label{6364}
  \sup_{t \in [0,T_N^*]} \|z_N(\cdot,t)\|_{H^2(\mathbb{T}^2)}^2 \lesssim Z_N(0),
\end{align}
so the coefficient vector $a^N(t) = \{a_k^N(t)\}_{|k|\le N}$ remains bounded
in $\mathbb{C}^{d_N}$ on $[0,T_N^*)$. Since the right-hand side
of~\eqref{eq:ode-coefficients} is locally Lipschitz in $a^N$ and $a^N(t) = \{a_k^N(t)\}_{|k|\le N}$ can be extended up to $T_N^*$, the solution extends beyond $T_N^*$ by applying the Picard–Lindelöf theorem again, using the terminal value as a new initial condition, contradicting
its maximality. Hence $T_N^* > L$.

In particular, \eqref{eq:ZN-uniform-bound} holds on $[0,L]$, and
integrating \eqref{eq:ZN-diff-ineq} over $[0,L]$ gives, using \eqref{eq:H3-vs-grad} and \eqref{eq:H-choice},
\begin{align}
  \int_0^L \|\nabla(1-\Delta)z_N(\cdot,t)\|_{L^2(\mathbb{T}^2)}^2\,dt
  \lesssim \|z_0\|_{H^2(\mathbb{T}^2)}^2,\text{ and }
  \int_0^L \|z_N(\cdot,t)\|_{H^3(\mathbb{T}^2)}^2\,dt
  \lesssim (1+L)\|z_0\|_{H^2(\mathbb{T}^2)}^2.
  \label{eq:uniform-L2H4}
\end{align}

\noindent\textbf{Step 4. Passing $N \to \infty$:}
The bounds \eqref{eq:ZN-uniform-bound}--\eqref{eq:uniform-L2H4} are uniform in $N$.
From the Galerkin equation~\eqref{eq:Galerkin-z} and the estimate~\eqref{eq:nonlinear-estimate} for the nonlinear term, we obtain
\[
  \|\partial_t z_N\|_{H^1}
  \lesssim \|z_N\|_{H^2(\mathbb{T}^2)} + H\|z_N\|_{H^2(\mathbb{T}^2)} + \|\mathcal{N}(z_N)\|_{H^1},
\]
where each term belongs to $L^2(0,L)$ uniformly in $N$ (with $H$ fixed), so
\begin{equation}
\label{eq:dt-zN-bound}
  \{\partial_t z_N\}_{N\ge 1}
  \quad\text{is bounded in } L^2(0,L;H^1(\mathbb{T}^2)).
\end{equation}
By \eqref{eq:ZN-uniform-bound}, \eqref{eq:uniform-L2H4}, and \eqref{eq:dt-zN-bound},
we may pass to a subsequence (still labelled $z_N$) such that
\begin{alignat}{2}
  z_N &\rightharpoonup z
  &&\quad\text{weakly in } L^2(0,L;H^3(\mathbb{T}^2)),
  \label{eq:weak-H4}\\
  z_N &\stackrel{*}{\rightharpoonup} z
  &&\quad\text{weakly-$*$ in } L^\infty(0,L;H^2(\mathbb{T}^2)).
  \label{eq:weak-star-H3}
\end{alignat}
Since $H^3(\mathbb{T}^2) \hookrightarrow\hookrightarrow H^s(\mathbb{T}^2) \hookrightarrow H^1(\mathbb{T}^2)$
compactly for every $1 < s < 3$, the Aubin--Lions lemma applied to \eqref{eq:dt-zN-bound} and \eqref{eq:weak-star-H3} yields
\begin{equation}
\label{eq:strong-convergence}
  z_N \to z \quad\text{strongly in } C([0,L];H^s(\mathbb{T}^2))
  \quad \text{for every } 1<s < 3.
\end{equation}
The strong convergence~\eqref{eq:strong-convergence} suffices to pass to the limit in
the nonlinear term $\mathcal{N}(z_N)$ (using~\eqref{eq:nonlinear-estimate} and dominated
convergence), and $z$ is a strong solution of~\eqref{eq:z-equation} on $[0,L]$.

It remains to upgrade the regularity.  From \eqref{eq:dt-zN-bound}-\eqref{eq:weak-H4}
we have
\[
  z \in L^2(0,L;H^3(\mathbb{T}^2)), \qquad \partial_t z \in L^2(0,L;H^1(\mathbb{T}^2)).
\]
Since $[H^3(\mathbb{T}^2),\,H^1(\mathbb{T}^2)]_{1/2} = H^2(\mathbb{T}^2)$, interpolation theorem \cite[Chapter III, Theorem 4.10.2]{Amann1995} gives
\[
  z \in C([0,L];H^2(\mathbb{T}^2)).
\]
Passing $N\to \infty$ in \eqref{6364}, the weak lower semicontinuity yields
\begin{align}\label{6429}
  \sup_{t \in [0,L]} \|z(\cdot,t)\|_{H^2(\mathbb{T}^2)}^2 \lesssim \|z_0\|_{H^2(\mathbb{T}^2)}^2,
\end{align}
\sloppy Finally, since $\Psi_0 : H^2(\mathbb{T}^2;\mathbb{C}) \to H^2(\mathbb{T}^2;\mathbb{S}^2)$
is continuous, we conclude $
  m(x,t) := \Psi_0(z(x,t)) \in C([0,L];H^2(\mathbb{T}^2;\mathbb{S}^2)).$ Using \eqref{eq:LLG-two-stage}, Sobolev embedding and standard parabolic estimate, we conclude the desired regularity.
\end{proof}

\begin{lemma}[\bf Large-field stabilisation]
\label{lem:large-field-stabilization}
Let $L>0$, $p^\dagger\in\mathbb{S}^2$,
and let $m_0\in H^2(\mathbb{T}^2;\mathbb{S}^2)$ for some $p>2$ satisfy
\[
\inf_{x\in\mathbb{T}^2} m_0(x)\cdot p^\dagger > 0.
\]
Then for every $\rho>0$ there exists a constant
$H_*=H_*(L,\rho,m_0,p^\dagger,\alpha,\gamma)>0$ such that for every constant $H\ge H_*$ the solution $m$ of
\begin{equation}
\label{eq:LLG-field}
  \partial_t m
  =\alpha\bigl[\Delta m+|\nabla m|^2 m\bigr]
  +\gamma\, m\times\Delta m
  +\gamma H\, m\times p^\dagger
  -\alpha H\, m\times(m\times p^\dagger),
  \quad
  m(\cdot,0)=m_0(\cdot),
\end{equation}
on $\mathbb{T}^2\times(0,L)$ satisfies
$m\in C((0,L];H^4(\mathbb{T}^2;\mathbb{S}^2))$ and
\[
  \|m(\cdot,L)-p^\dagger\|_{H^1(\mathbb{T}^2)}
  +\|m(\cdot,L)-p^\dagger\|_{L^\infty(\mathbb{T}^2)}
  <\rho.
\]
\end{lemma}

\begin{proof} By the rotational equivariance of \eqref{eq:LLG-field}, it suffices to treat the case $p^\dagger=e_3$. Therefore, for $H\geq H_*$, we consider the strong solution $z \in C\!\left([0,L];H^2(\mathbb{T}^2)\right)$ of \eqref{eq:z-equation}, given by \Cref{lem:z-existence-large-H}, with initial
datum \(z_0:=\Psi_0^{-1}(Rm_0)\), where \(R\in SO(3)\) is chosen so that
\(Rp^\dagger=e_3\). Set $X(t):=\|z(\cdot,t)\|_{L^2(\mathbb{T}^2)}^2$. Taking the $L^2$-inner product of \eqref{eq:z-equation} with $z$ and extracting the real part gives
\[
  \tfrac{1}{2}X'(t)
  +\alpha\|\nabla z(\cdot,t)\|_{L^2(\mathbb{T}^2)}^2
  +\alpha HX(t)
  =\operatorname{Re}\left[(\alpha+i\gamma)
  \int_{\mathbb{T}^2}\mathcal{N}(z)\,\overline{z}\,dx\right].
\]
Using $|\mathcal{N}(z)|\le|\nabla z|^2$ together with the uniform bound
$\|z\|_{L^\infty(\mathbb{T}^2)}\lesssim \|z\|_{H^2(\mathbb{T}^2)}\lesssim \|z_0\|_{H^2(\mathbb{T}^2)}$ from \eqref{6429},
we estimate
\[
  \left|\operatorname{Re}\left[(\alpha+i\gamma)
  \int_{\mathbb{T}^2}\mathcal{N}(z)\,\overline{z}\,dx\right]\right|
  \lesssim_{\,\alpha,\gamma} \|z\|_{L^\infty(\mathbb{T}^2)}\|\nabla z\|_{L^2(\mathbb{T}^2)}^2 
   \lesssim \|z_0\|_{H^2(\mathbb{T}^2)}^3.
\] 
We arrive at
\[
  X'(t)+2\alpha HX(t) \lesssim_{\,\alpha,\gamma} \|z_0\|_{H^2(\mathbb{T}^2)}^3,
\]
Gronwall's inequality gives
\begin{equation}
\label{eq:L2-decay}
  X(L)\le e^{-2\alpha HL}X(0)+\frac{C_{\alpha,\gamma}\|z_0\|_{H^2(\mathbb{T}^2)}^3}{2\alpha H}
  \rightarrow0\quad \textup{as $H\to\infty$.}
\end{equation} 
By interpolation between $L^2$ and $H^2$, applied to the uniform bound
\eqref{6429},
\[
  \|z(\cdot,L)\|_{H^1 \cap L^\infty} 
  \lesssim \|z(\cdot,L)\|_{L^2(\mathbb{T}^2)}^{1/2}\|z(\cdot,L)\|_{H^2(\mathbb{T}^2)}^{1/2}
  \lesssim \|z_0\|_{H^2(\mathbb{T}^2)}^{1/2} \|z(\cdot,L)\|_{L^2(\mathbb{T}^2)}^{1/2}.
\]
Thus, by \eqref{eq:L2-decay},
\[
  \|z(\cdot,L)\|_{H^1 \cap L^\infty(\mathbb{T}^2)}
  \lesssim_{\,\|z_0\|_{H^2(\mathbb{T}^2)}} \|z(\cdot,L)\|_{L^2(\mathbb{T}^2)}^{1/2}
  \longrightarrow 0
  \quad\text{as }H\to\infty.
\]
Since $m=\Psi_0(z)$ and $\Psi_0(0)=e_3$, 
we have
\[
  \|m(\cdot,L)-e_3\|_{H^1}+\|m(\cdot,L)-e_3\|_{L^\infty(\mathbb{T}^2)}
  \lesssim_{\|m_0\|_{H^2(\mathbb{T}^2)}} 
    \|z(\cdot,L)\|_{H^1}+\|z(\cdot,L)\|_{L^\infty(\mathbb{T}^2)}.
\]
Hence there
exists $H_*$ (depending on $L$, $\rho$, $\|m_0\|_{H^2(\mathbb{T}^2)}$, $\alpha$, $\gamma$)
such that for all $H\ge H_*$,
\[
  \|m(\cdot,L)-e_3\|_{H^1}+\|m(\cdot,L)-e_3\|_{L^\infty(\mathbb{T}^2)}<\rho.
\] 
\end{proof}

\begin{proof}[\bf Proof of \Cref{cor:two-stage-global-a}] 
Since \(m_0\in H^2(\mathbb T^2;\mathbb{S}^2)\) with \(p>2\), we have
	 	\(m_0\in C(\mathbb T^2)\). Hence
	 	\[
	 	\delta_0:=\min_{x\in\mathbb T^2}m_0(x)\cdot e_3>0.
	 	\]
	 	First apply \Cref{lem:large-field-stabilization} with $\rho=\frac{p_{0,3}}{4}$, $L:=\frac{T}{3}$, $
	 	p^\dagger=e_3$ and initial condition $m_0$. 
        Thus there exists \(H_1\ge0\) such that the solution of \eqref{eq:LLG-two-stage} on \([0,L]\) with
	 	\(B=H_1e_3\) satisfies
	 	\[
	 	\|m(\cdot,L)-e_3\|_{L^\infty(\mathbb T^2)}
	 	<\frac{p_{0,3}}{4}.
	 	\]
	 	Consequently, for every \(x\in\mathbb T^2\),
	 	\[
	 	m(x,L)\cdot p_0
	 	=
	 	e_3\cdot p_0+(m(x,L)-e_3)\cdot p_0
	 	\ge
	 	p_{0,3}-|m(x,L)-e_3|
	 	>
	 	\frac{3p_{0,3}}{4}>0.
	 	\] 

        Let \(\varepsilon >0\) be the small number given by \Cref{thm:nullcontrol} for the time interval of length \(L:=\frac{T}{3}\) and the control
	 	region \(\omega\). Now apply \Cref{lem:large-field-stabilization} again with $\rho=\varepsilon/2$, time horizon $L$, $
	 	p^\dagger=p_0$ and $m_0(\cdot)=m(\cdot,L)$. Therefore, there exists
	 	\(H_2\ge0\) such that the solution on \([L,2L]\) with \(B=H_2p_0\) satisfies
	 	\[
	 	\|m(\cdot,2L)-p_0\|_{H^1(\mathbb T^2)}<\varepsilon/2.
	 	\]
	 	Finally, we apply \Cref{thm:nullcontrol} on the last
	 	interval \([2T/3,T]\). Hence, for every $
	 	m_T\in \mathcal T(\varepsilon/2,T/3,p_0)\subset \mathcal T(\varepsilon,T/3,m(2T/3))$ there exists a control
	 	\[
	 	u\in L^\infty(\omega\times(2T/3,T);\mathbb{R}^3)
	 	\]
	 	such that the solution of the controlled LLG equation \eqref{eq:LLG-two-stage} with control $
	 	B(x,t)=\chi_\omega u(x,t)$ for $t\in(2T/3,T)$ 
	 	satisfies
	 	\[
	 	m(\cdot,T)=m_T(\cdot).
	 	\] 
	 	Concatenating the three controls $H_1e_3$, $H_2p_0$ and $\chi_\omega u(x,t)$ proves the desired result.
	 \end{proof}

		 \noindent\textbf{Acknowledgment:} The authors would like to express their gratitude Professor Beniamin Goldys for valuable discussion and guidance of this article. The authors would like to express their gratitude to Professor Marius Tucsnak for pointing out very useful references. H. M. Tai is partially supported by Australian Research Council Discovery Project DP240100781.  

\appendix

	\section{Derivation of projected equation \eqref{eq. projected eq} under stereographic coordinates}\label{app. Derivation of projected equation under stereographic coordinates}
	Let $m$ be a solution to the controlled equation \eqref{eq. controlled LLG} subject to control $u$. We recall $v$ and $\Psi_0$ given in \eqref{def. stereographic projection and inverse} such that $\Psi_0(v)=m$. It is computed that \[
	\nabla_v \Psi_0(v)
	=
	\frac{2}{h^2}
	\begin{pmatrix}
		1-v_1^2+v_2^2 & -2v_1v_2\\
		-2v_1v_2 & 1+v_1^2-v_2^2\\
		-2v_1 & -2v_2
	\end{pmatrix}
	\]
	where $h=h(v):=1+|v|^2$. It is immediate that
	\begin{align}
		(\nabla_v \Psi_0)^\top\nabla_v \Psi_0=\dfrac{4}{h^2}I\h{15pt} \text{or equivalently} \h{15pt}
		\partial_{v_\ell}\Psi_0\cdot \partial_{v_j}\Psi_0=\frac{4}{h^2}\delta_{\ell j}\h{5pt} \text{for $\ell$, $j=1,2$.}
		\label{eq. D Psi top D Psi = 4I/h^2 and inverse of D Psi}
	\end{align}
	Unless specified, the map $\Psi_0$ and its derivatives are all evaluated at $v$.
	
	\paragraph{\bf Step 1. Harmonic map heat flow part:} Direct differentiation shows that
	\begin{align}\label{eq. proj eq part 1a}
		\p_t v = \dfrac{h^2}{4}(\nabla_v \Psi_0)^\top\p_t m	\h{10pt} \text{and} \h{10pt}
		\dfrac{h^2}{4}(\nabla_v \Psi_0)^\top \Delta m=\Delta v 
		-2(\nabla v)\nabla \log h+ \dfrac{2}{h}|\nabla v|^2 v
	\end{align} 
	where $
	\nabla \log h(v)
	=
	\nabla_x[\log h(v(x))]
	=
	\frac{1}{h(v)}(\nabla v)^\top \nabla_v h(v).$
	By $|\Psi_0|^2=1$, we have
	\begin{align}\label{eq. proj eq part 1b}\dfrac{h^2}{4}(\nabla_v \Psi_0)^\top |\nabla m|^2 m=0.
	\end{align}
	\paragraph{\bf Step 2. Precession term $\dfrac{h^2}{4}(\nabla_v \Psi_0)^\top (m\times \Delta m)$:} 
	Identity \eqref{eq. D Psi top D Psi = 4I/h^2 and inverse of D Psi} implies that  $\{\partial_{v_1}\Psi_0,\partial_{v_2}\Psi_0\}$ spans the tangent plane $T_m\mathbb S^2$. There exist scalars $a_1,a_2,\beta$ such that
	\begin{equation}\label{eq:DeltaDecomp}
		\Delta m = \sum^2_{j=1}a_j\partial_{v_j}\Psi_0   + \beta m.
	\end{equation}
	For each $\ell=1,2$, we take inner product in \eqref{eq:DeltaDecomp} with $\partial_{v_\ell}\Psi_0 $:
	\begin{equation}\label{eq:aformula}
		\partial_{v_\ell}\Psi_0 \cdot\Delta m 
		= 4a_\ell/h^2.
	\end{equation}
	Moreover, the normal part does not contribute to the cross product: $m\times\Delta m = \sum^2_{j=1} m\times(a_j\partial_{v_j}\Psi_0  ).$ Therefore, we define the vector $b$ and express it by
	\begin{equation*} 
		b_\ell
		:=\Big[\frac{h^2}{4}(\nabla_v \Psi_0 )^{\top}(m\times\Delta m)\Big]_\ell
		=\sum^2_{j=1}\frac{h^2}{4}\,\p_{v_\ell} \Psi_{0,i} \,(m\times\Delta m)_i 
		=\sum^2_{j=1}\Big[\frac{h^2}{4}\,\partial_{v_\ell}\Psi_0 \cdot(m\times \partial_{v_j}\Psi_0 )\Big]a_j.
	\end{equation*} 
	Define the matrix $\widetilde{J}$ with entries
	\begin{equation}\label{eq:defC}
		\widetilde{J}_{\ell j}:=\frac{h^2}{4}\,\partial_{v_\ell}\Psi_0 \cdot(m\times \partial_{v_j}\Psi_0 ),
		\qquad\text{so that}\qquad
		b=\widetilde{J} a.
	\end{equation}
	
	Since $\partial_{v_1}\Psi_0\times \partial_{v_2}\Psi_0$ is parallel to $m$. There is $\sigma(v)\in \mathbb{R}$ such that $
	\partial_{v_1}\Psi_0\times \partial_{v_2}\Psi_0=\sigma(v)\,m.$ Taking norms and using $|\partial_{v_1}\Psi_0|=|\partial_{v_2}\Psi_0|=\frac{2}{h}$, we see that $|\sigma(v)|=\frac{4}{h^2}$. Evaluating at $v=0$, it yields that $h(0)=1$, $\Psi_0(0)=(0,0,1)^\top$, $\partial_{v_1}\Psi_0(0)=(2,0,0)^\top$ and $\partial_{v_2}\Psi_0(0)=(0,2,0)^\top$. Hence
	\[
	\partial_{v_1}\Psi_0(0)\times \partial_{v_2}\Psi_0(0)=(2,0,0)^\top\times(0,2,0)^\top=(0,0,4)^\top=4\Psi_0(0),
	\]
	so $\sigma(v)=\frac{4}{h^2}$ for all $v$. Therefore
	\begin{equation}\label{eq:mxej2}
		\partial_{v_1}\Psi_0\times \partial_{v_2}\Psi_0=\frac{4}{h^2}\,m.
	\end{equation}

	Use scalar triple product 
    and \eqref{eq:mxej2}:\[(m\times \partial_{v_1}\Psi_0)\cdot \partial_{v_2}\Psi_0 = m\cdot(\partial_{v_1}\Psi_0\times \partial_{v_2}\Psi_0)
	=\frac{4}{h^2}.\]
	Because $\{m=\Psi_0(v),\partial_{v_1}\Psi_0(m),\partial_{v_2}\Psi_0(m)\}$ is a orthogonal set and $|\partial_{v_2}\Psi_0|^2=\frac{4}{h^2}$, therefore
	\begin{equation}\label{eq:mxej}
		m\times \partial_{v_1}\Psi_0=\partial_{v_2}\Psi_0.
	\end{equation}
	Similarly, one gets
	\begin{equation}\label{eq:mxej12}
		m\times \partial_{v_2}\Psi_0=-\partial_{v_1}\Psi_0.
	\end{equation}

	Using \eqref{eq:defC}, \eqref{eq:mxej}, \eqref{eq:mxej12}, and $\partial_{v_1}\Psi_0\perp \partial_{v_2}\Psi_0$, we have
	\begin{align*}
		\widetilde{J}_{11}&=\frac{h^2}{4}\partial_{v_1}\Psi_0\cdot(m\times \partial_{v_1}\Psi_0)
		=0,\qquad
		\widetilde{J}_{12}=\frac{h^2}{4}\partial_{v_1}\Psi_0\cdot(m\times \partial_{v_2}\Psi_0) 
		=-1,\\
		\widetilde{J}_{21}&=\frac{h^2}{4}\partial_{v_2}\Psi_0\cdot(m\times \partial_{v_1}\Psi_0) 
		=1,\qquad
		\widetilde{J}_{22}=\frac{h^2}{4}\partial_{v_2}\Psi_0\cdot(m\times \partial_{v_2}\Psi_0)
		=0.
	\end{align*}
	Therefore
	\begin{equation}\label{eq:CisJ}
		\widetilde{J}=\begin{pmatrix}0&-1\\[2pt]1&0\end{pmatrix}=J.
	\end{equation} 
	From \eqref{eq:defC} and \eqref{eq:CisJ}, we have $b=Ja$.
	Using \eqref{eq:aformula}, we have $a=\frac{h^2}{4}(\nabla_v \Psi_0 )^\top\Delta m$, hence \eqref{eq. proj eq part 1a} implies
	\begin{align}\label{eq. proj eq part 2}
		b=\frac{h^2}{4}(\nabla_v \Psi_0 )^{\top}\bigl(m\times\Delta m\bigr)
		=J\Big[\frac{h^2}{4}(\nabla_v \Psi_0 )^{\top}\Delta m\Big]
		=J\Big[\Delta v 
		-2(\nabla v)\nabla \log h+ \dfrac{2}{h}|\nabla v|^2 v\Big].
	\end{align}
	
	\paragraph{\bf Step 3. First part with control:} Write
	\[
	u=u^\parallel+u^\perp,\qquad u^\perp:=(u\cdot m)m,\qquad u^\parallel :=u-(u\cdot m)m\in T_m\mathbb S^2.
	\]
	Since $u^\parallel \in\mathrm{span}\{\partial_{v_1}\Psi_0 ,\partial_{v_2}\Psi_0 \}$, there exists $b^*=(b_1^*,b_2^*)\in\mathbb R^2$ such that
	\begin{equation}\label{eq:utop}
		u^\parallel =\sum^2_{j=1}\partial_{v_j}\Psi_0  b^*_j.
	\end{equation}
	Let $\ell=1,2$. Taking inner product in \eqref{eq:utop} with $\partial_{v_\ell}\Psi_0 $, the relation \eqref{eq. D Psi top D Psi = 4I/h^2 and inverse of D Psi} and the fact that $\partial_{v_\ell}\Psi_0 \perp m$ imply
	\begin{equation}\label{eq:bformula}
		\partial_{v_\ell}\Psi_0 \cdot u
		=\partial_{v_\ell}\Psi_0 \cdot u^\parallel  
		= \sum^2_{j=1}(\partial_{v_\ell}\Psi_0 \cdot \partial_{v_j}\Psi_0 )b^*_j=\frac{4}{h^2}b^*_\ell.
	\end{equation}
	Also, we have
	\[
	m\times u=m\times u^\parallel =\sum^2_{j=1}m\times(\partial_{v_j}\Psi_0  b^*_j).
	\]
	Define the vector
	\begin{equation*}
		a^*_\ell
		:=\Big[\frac{h^2}{4}(\nabla_v \Psi_0 )^{\top}(m\times u)\Big]_\ell
		=\frac{h^2}{4}\,\p_{v_\ell} \Psi_{0} \cdot (m\times u)
		=\sum^2_{j=1}\Big[\frac{h^2}{4}\partial_{v_\ell}\Psi_0 \cdot(m\times \partial_{v_j}\Psi_0 )\Big]b^*_j.
	\end{equation*} 
	and the matrix
	\[
	H_{\ell j}:=\frac{h^2}{4}\partial_{v_\ell}\Psi_0 \cdot(m\times \partial_{v_j}\Psi_0 ),
	\qquad\text{so that}\qquad a^*=Hb^*.
	\]
	
	Using \eqref{eq:mxej}--\eqref{eq:mxej2} and orthogonality:
	\begin{align*}
		H_{11}&=\frac{h^2}{4}\partial_{v_1}\Psi_0 \cdot(m\times \partial_{v_1}\Psi_0 )
		=0,\qquad
		H_{12}=\frac{h^2}{4}\partial_{v_1}\Psi_0 \cdot(m\times \partial_{v_2}\Psi_0 )
		=-1,\\
		H_{21}&=\frac{h^2}{4}\partial_{v_2}\Psi_0 \cdot(m\times \partial_{v_1}\Psi_0 )
		=1,\qquad
		H_{22}=\frac{h^2}{4}\partial_{v_2}\Psi_0 \cdot(m\times \partial_{v_2}\Psi_0 )
		=0.
	\end{align*}
	Therefore
	\begin{equation*}
		H=\begin{pmatrix}0&-1\\[2pt]1&0\end{pmatrix}=J.
	\end{equation*}
Hence, using \eqref{eq:bformula} and $a^*=Hb^*=Jb^*$, we finally obtain
	\begin{align}\label{eq. proj eq part 3}
		a^*=\frac{h^2}{4}(\nabla_v \Psi_0 )^{\top}(m\times u)
		=
		J\Big[\frac{h^2}{4}(\nabla_v \Psi_0 )^{\top}u\Big].
	\end{align}
	
	\paragraph{\bf Step 4. Second part of control:}  
	Define the vector
	\begin{equation}\label{eq:defQ-mmx}
		N_\ell
		:=\Big[\frac{h^2}{4}(\nabla_v \Psi_0 )^{\top}\big(m\times(m\times u)\big)\Big]_\ell
		=\sum^2_{i=1}\frac{h^2}{4}\,\p_{v_\ell} \Psi_{0,i} \,\big(m\times(m\times u)\big)_i .
	\end{equation}
	Insert the triple product identity $m\times(m\times u)=m(m\cdot u)-u$ into \eqref{eq:defQ-mmx}:
	\begin{align*}
		N_\ell
		=\sum^2_{i=1}\frac{h^2}{4}\,\p_{v_\ell} \Psi_{0,i} \,\big[m_i(m\cdot u)-u_i\big]
		&=\sum^2_{i=1}\frac{h^2}{4}(m\cdot u)\,\p_{v_\ell} \Psi_{0,i} m_i \;-\;\frac{h^2}{4}\,\p_{v_\ell} \Psi_{0,i} u_i \\
		&=-\;\sum^2_{i=1}\frac{h^2}{4}\,\p_{v_\ell} \Psi_{0,i} u_i 
	\end{align*} 
	We conclude
	\begin{align}\label{eq. proj eq part 4}
		\frac{h^2}{4}(\nabla_v \Psi_0 )^{\top}\big[m\times(m\times u)\big]
		=-\,\frac{h^2}{4}(\nabla_v \Psi_0 )^{\top}u
	\end{align}
	
	Summing up \eqref{eq. proj eq part 1a}, \eqref{eq. proj eq part 1b}, \eqref{eq. proj eq part 2}, \eqref{eq. proj eq part 3}, \eqref{eq. proj eq part 4}, we obtain \eqref{eq. projected eq}.

	\bibliographystyle{abbrv} 
	\bibliography{bio}
	
\end{document}